\documentclass[11pt,reqno]{amsart}

\usepackage{amsmath}

\usepackage{graphicx}
\usepackage{mathrsfs}
\usepackage{color}
\usepackage{comment}
\usepackage{amssymb}
\usepackage{enumitem}
\usepackage{makecell}
\usepackage[margin = 1in] {geometry}

\usepackage{listings}
\usepackage{galois}
\usepackage[toc,page]{appendix}
\usepackage{xcolor}
\usepackage[colorlinks, linkcolor = blue, anchorcolor = blue, citecolor = blue,urlcolor= black]{hyperref} 
\usepackage{marginnote}
\usepackage{comment}

\usepackage{array}
\usepackage{seqsplit}

\usepackage{stmaryrd}
\usepackage{amsfonts}

\allowdisplaybreaks

\numberwithin{equation}{section}

\newtheorem{prop}{Proposition}[section]
\newtheorem{theo}[prop]{Theorem}
\newtheorem{lemm}[prop]{Lemma}

\newtheorem*{claim*}{Claim}

\theoremstyle{definition}

\newtheorem{rema}[prop]{Remark}

\newcommand{\BC}{\mathbf{C}}

\newcommand{\NN}{\mathbb{N}}

\newcommand{\RR}{\mathbb{R}}
\newcommand{\CC}{\mathbb{C}}
\newcommand{\HH}{\mathbb{H}}
\renewcommand{\SS}{\mathbb{S}}

\newcommand{\ZZ}{\mathbb{Z}}

\newcommand{\bi}{\mathbf{i}}
\newcommand{\bj}{\mathbf{j}}
\newcommand{\bk}{\mathbf{k}}

\newcommand{\di}{\partial_{1}}
\renewcommand{\dj}{\partial_{2}}
\newcommand{\dk}{\partial_{3}}

\newcommand{\cH}{\mathcal H}

\newcommand{\cJ}{\mathcal J}
\newcommand{\cK}{\mathcal K}
\newcommand{\cL}{\mathcal L}

\newcommand{\cN}{\mathcal N}

\newcommand{\cP}{\mathcal P}
\newcommand{\cQ}{\mathcal Q}

\newcommand{\cS}{\mathcal S}
\newcommand{\cT}{\mathcal T}

\newcommand{\cV}{\mathcal V}

\def\fX{\mathfrak{X}}

\def\fs{\mathfrak{s}}
\def\fo{\mathfrak{o}}

\DeclareMathOperator{\tr}{tr}

\DeclareMathOperator{\imag}{im}

\DeclareMathOperator{\Span}{span}
\DeclareMathOperator{\diag}{diag}

\newcommand{\ii}{\sqrt{-1}}

\DeclareMathOperator{\spt}{spt}

\newcommand{\la}{\Delta}

\DeclareMathOperator{\Div}{div}

\renewcommand{\d}{\nabla}

\newcommand{\df}{\nabla^\bot}

\newcommand{\pa}[2]{\frac{\partial #1}{\partial #2}}

\newcommand{\pd}[2]{\frac{d #1}{d #2}}

\newcommand{\eps}{\varepsilon}
\renewcommand{\bar}{\overline}
\renewcommand{\hat}{\widehat}
\renewcommand{\tilde}{\widetilde}
\newcommand{\db}{\bar \d}

\renewcommand{\arraystretch}{2}

\newcommand{\abs}[1]{\lvert#1\rvert}

\def\Xint#1{\mathchoice
	{\XXint\displaystyle\textstyle{#1}}%
	{\XXint\textstyle\scriptstyle{#1}}%
	{\XXint\scriptstyle\scriptscriptstyle{#1}}%
	{\XXint\scriptscriptstyle\scriptscriptstyle{#1}}%
	\!\int}
\def\XXint#1#2#3{{\setbox0=\hbox{$#1{#2#3}{\int}$ }
		\vcenter{\hbox{$#2#3$ }}\kern-.6\wd0}}

\def\dashint{\Xint-}

\newcommand{\ang}[1]{\langle #1 \rangle}

\newcommand{\weakto}{\rightharpoonup}

\newcommand{\de}{\partial}
\newcommand{\fh}[2]{\frac{#1}{#2}}

\renewcommand{\leq}{\leqslant}
\renewcommand{\geq}{\geqslant}

\newcommand{\norm}[1]{\left\lVert#1\right\rVert}

\title[Integrability of Lawson--Osserman Cone and its Applications]{Integrability of Lawson--Osserman Cone and its Applications}
\author{Qi Ding}
\email{dingqi@fudan.edu.cn}
\author{Lei Zhang}
\email{22110840013@m.fudan.edu.cn}
\address{Shanghai Center for Mathematical Sciences, Fudan University, Shanghai 200438, China}
\date{}

\begin{document}

\begin{abstract}
       In this paper, we characterize all eigenfunctions corresponding to nonpositive eigenvalues of the Jacobi operator of the link $M$ of the Lawson--Osserman cone $\BC$ in $\mathbb{R}^7$. In particular, we prove that $\BC$ is integrable, i.e., all Jacobi fields on $\BC$ of homogeneous degree 1 and 0, are generated by
rotations and translations in $\RR^7$.
As applications, we prove that $M$ is rigid as minimal submanifolds in $\SS^6$,
and 
derive the optimal decay order for minimal submanifolds in $\mathbb{R}^7$ asymptotic to $\BC$ at infinity.
\end{abstract}

\maketitle

\setcounter{tocdepth}{1}
\tableofcontents

\section{Introduction}

Let $\BC_0$ be a regular minimal cone in Euclidean space $\RR^{m+1}$, and $\Sigma$ be the link of $\BC_0$. Then $\Sigma$ is minimal in the unit sphere $\SS^m\subset\RR^{m+1}$.
Let $\cJ_\Sigma$ denote the Jacobi operator of $\Sigma$ in $\SS^m$.  A smooth normal vector $Y$ on $\Sigma$ is called a \emph{Jacobi field} on $\Sigma$ if $\cJ_\Sigma Y=0$ (see Simons \cite{Simons}).
Allard--Almgren \cite{allard.radial} introduced a concept for the regular $\BC_0 $ as follows:
\begin{center}
($\dagger$)\qquad \ every Jacobi field on $\Sigma$ arises as the variational vector field of some 1-parameter family of minimal submanifolds in $\SS^m$,
\end{center}
which is referred to as an integrablility condition.
Allard--Almgren \cite{allard.radial} proved the celebrated uniqueness theorem of tangent cones for minimal submanifolds provided one of tangent cones has multiplicity one and satisfies the condition ($\dagger$). 
However, ($\dagger$) does not hold for all regular minimal cones pointed out by Allard--Almgren \cite[page 216]{allard.radial}.
In \cite{asymptotics}, Simon established the uniqueness theorem for the non-integrable case, where he developed a famous method, called Simon--\L{}ojasiewicz inequality nowadays. Nevertheless, the condition ($\dagger$) also acts a key role in the study of the structure of minimal submanifolds with non-isolated singularities. 
Simon \cite{Sim94} proved rectifiability and local finiteness of measure of the singular sets for minimal submanifolds close to cylindrical minimal cones of the form $\BC_0\times\RR^k$, where $\BC_0$ satisfies ($\dagger$).

In \cite{Simonuniqueness}, Simon proposed an  integrability condition stronger than ($\dagger$) on the regular minimal cone $C$, which contains that:
\begin{center}
($\ddagger$)\qquad \ all Jacobi fields on $C$ of homogeneous degree 1 and 0, are generated by
rotations and translations in $\RR^n$.
\end{center}
Here, 1-homogeneous Jacobi fields on $C$ correspond to Jacobi fields on the link of $C$, and 0-homogeneous Jacobi fields on $C$ correspond to eigenfunctions of the Jacobi operator on the link of $C$ with eigenvalue $1-\dim\,C$.  
For convenience, we also call ($\ddagger$) \emph{integrablility}, which is strictly weaker than one in \cite{Simonuniqueness}. Though ($\ddagger$) was originally proposed by Simon for the codimension 1 case, it clearly extends to higher codimensions as well.

The integrability condition ($\ddagger$) imposed on minimal cones plays an essential role in the study of the uniqueness of minimal hypersurfaces converging to these cones or their products with Euclidean factors.
For the codimension $1$ case, all the quadratic minimal cones 
$$C_{p,q}: = \{(x,y)\in \RR^{p+1}\times \RR^{q+1}:q|x|^2 = p|y|^2 \}$$ 
satisfy the  integrability condition ($\ddagger$) proved by Allard--Almgren in \cite[Chapter 6]{allard.radial} (see also Simon--Solomon \cite[Proposition 2.7]{simonsolo}). Making full use of ($\ddagger$),   Simon--Solomon \cite{simonsolo} established the uniqueness of complete minimal hypersurfaces asymptotic to an area-minimizing $C_{p,q}$ (see also Mazet \cite{Mazet1}). Furthermore, Edelen--Spolaor \cite{edelenluca} resolved the uniqueness of minimal hypersurfaces in the unit ball close to an area-minimizing $C_{p,q}$ and their result also implies the Bernstein
type result in \cite{simonsolo}.
 For cylindrical cones, Simon proved the uniqueness of multiplicity one tangent cones of the form $C\times\RR$ for area-minimizing hypersurfaces under the  integrability condition on $C$ in \cite{Simonuniqueness}, where both $C_{3,3}$ and $C_{2,4}$ do not satisfy the condition. Nevertheless, Sz\'ekelyhidi \cite{szekelyhidi2020} can resolve the uniqueness of the cylindrical tangent cones $C_{3,3}\times\RR$ for area-minimizing hypersurfaces in $\RR^9$. Recently, Firester--Tsiamis--Wang \cite{firester2025} proved the remaining case of $C_{2,4}\times\RR$  by adopting the strategy of Sz\'ekelyhidi \cite{szekelyhidi2020}.
See Collins--Li \cite{Collins_Li} for the Lagrangian case under a Lagrangian  integrability condition.

For infinitesimal variations without Lagrangian or complex structure, the Jacobi operator of the link of a minimal cone of high codimension is an elliptic system of second-order. 
Let $M$ be an $n$-dimensional closed minimal submanifold immersed in $\SS^{n+q}$. Simons \cite{Simons} proved that the space of Jacobi fields on $M$ has dimension $\geq q(n+1)$ with the equality achieved only if $M$ is the totally geodesic submanifold $\SS^n\subset\SS^{n+q}$. 
In general, the precise computation of eigenvalues and eigenspaces of Jacobi operators for minimal submanifolds of high codimensions is highly nontrivial except the products of standard spheres.

Lawson--Osserman \cite{Lawson} constructed a celebrated 4-dimensional entire minimal cone in $\RR^7$ using the Hopf map \begin{equation}\label{eq.Hopf.map}
    \eta : \SS^3 \to \SS^2\,, \quad \eta(z_1,z_2) = (\abs{z_1}^2 - \abs{z_2}^2,2z_1\bar{z}_2)
\end{equation}
   for each $z_1,z_2 \in \CC$ with
      $\abs{z_1}^2 + \abs{z_2}^2 = 1,$
       where $\SS^3$ is considered as the unit sphere in $\CC^2 \cong \RR ^4$, and $\SS^2$ as the unit sphere in $\RR \oplus \CC \cong \RR ^3$.  
       Let $u$ be a Lipschitz function from $\RR^4$ to $\RR^3$ defined by
      \begin{equation}\label{eq.lo.lip.graph}
      u(x) = \fh{\sqrt{5}}{2}\abs{x}\eta\left(\fh{x}{\abs{x}}\right)\qquad \text{for each } x\in\RR^4.\end{equation}
      Then 
      the graph of $u$ in $\RR^7$ is the celebrated \emph{Lawson--Osserman cone} \cite{Lawson}, which is the first example of Lipschitz entire minimal graph that is not $C^{1}$. 
The Lawson--Osserman cone is the simplest non-flat graphic minimal cone in the Euclidean space as far as we know (see more higher dimensional examples in \cite{Lawson}). This cone also serves as a significant counterexample in the study of Bernstein theorem for minimal graphs of high codimensions (see \cite{H-J-W, Jost-Xin, Jost-Xin-Yang, Ding-Jost-Xin} for instance).
Moreover, the Lawson--Osserman cone is coassociative, i.e., a calibrated, hence minimizing submanifold of a 7-manifold with $G_2$ structure (see Harvey--Lawson \cite{Harvey_Lawson}).

The unit 6-sphere $\SS^6\subset\RR^7$ inherits a nearly K$\mathrm{\ddot{a}}$hler structure with a non-degenerate 2-form $\omega$ from
the standard $G_2$ structure on $\RR^7$.
Moreover, the link 
$M$ of the Lawson--Osserman cone is a Lagrangian submanifold of $(\SS^6,\omega)$. Lotay \cite[Corollary 5.8]{Lotay} reduced the linear first-order system corresponding to the infinitesimal deformation of $M$ as a Lagrangian submanifold in $(\SS^6,\omega)$ to some linear equation, and proved that the space of Lagrangian Jacobi fields on the link of Lawson--Osserman cone has dimension $10$.

Now, we present a complete characterization of all eigenfunctions corresponding to nonpositive eigenvalues of the Jacobi operator of $M$ in $\SS^6$ without Lagrangian condition as below.
\begin{theo} \label{theo.dim}
         All Jacobi fields of the link $M$ of the Lawson--Osserman cone in $\RR^7$ are Killing Jacobi fields, and the cone is integrable, i.e., ($\ddagger$) holds. In particular, the first eigenvalue of the Jacobi operator of $M$ is $-\fh{15}{4}$ with multiplicity 1, the second eigenvalue is $-3$ with multiplicity $7,$ and the third eigenvalue is $0$ with multiplicity $17$. 
     \end{theo}
Here, Killing Jacobi fields are the normal projection of Killing fields on $\SS^6$ (see \S 2 for details), and the Killing fields on $\SS^6$ are the infinitesimal generators of its orientation-preserving isometry group SO$(7)$.  Theorem \ref{theo.dim} tells us that the eigenfunction space with eigenvalue $-3$ of the Jacobi operator of  $M$ consists of normal projection of coordinate functions. Furthermore,  Theorem \ref{theo.dim} indicates that Morse index of $M$ is $8$ and the space of Jacobi fields has dimension $17$. Morse index is an important object in the research of minimal submanifolds, and is useful in the proof of the Willmore conjecture solved by Marques--Neves \cite{MarquesNeves}. We refer the readers to Urbano's book \cite{bookurbano} for more details. 
\begin{rema}
In Theorem \ref{theo.dim}, 
 the unit normal vector field $e_4$  defined in \eqref{subsec.frame} is an eigenfunction of the first eigenvalue $-\fh{15}{4}$ (see more details in the proof of Theorem \ref{theo.dim} in \S5). Furthermore, the link $M$ is a submanifold contained in $\SS^3\left(\fh{2}{3}\right) \times \SS^2\left(\fh{\sqrt{5}}{3}\right)\subset\SS^6$, and $e_4$ coincides with the restriction of the unit normal vector field of $\SS^3\left(\fh{2}{3}\right) \times \SS^2\left(\fh{\sqrt{5}}{3}\right)$ in $\SS^6$ to $M$.
Thus, Theorem \ref{theo.dim} gives the key information for perturbation of minimal submanifolds close to the Lawson--Osserman cone $\BC$.
Compared with the codimension 1 case,  we expect that Theorem \ref{theo.dim} is useful in the study of the uniqueness or the local structure of minimal submanifolds with tangent cones $\BC$ or $\BC\times\RR$.
\end{rema}
There are two key ingredients in the proof of Theorem \ref{theo.dim}. One is a delicate selection of global frames on the tangent bundle and the normal bundle of the link $M$ of Lawson--Osserman cone with the help of a natural group homomorphism $\Psi:\mathrm{SU}(2)\to \mathrm{SO}(7)$ defined in \eqref{eq.isometry.action} so that we can write out the Jacobi operator $\cJ_M$ clearly under our frames. The other is that we successfully reduce the computation of eigenvalues and eigenfunction spaces of $\cJ_M$ to the eigenfunction
spaces of Laplacian on $\SS^3$ through a crucial observation that the operator $L$ in \eqref{oper.L}  can commute with the operator $\tilde L$ in \eqref{oper.L1}. 

As an application of Theorem \ref{theo.dim}, we prove a rigidity result as below.
\begin{theo} \label{link-LO-rigidty}
The link $M$ of the Lawson--Osserman cone $\BC$ is rigid as minimal submanifolds in $\SS^6$. Namely, any embedded minimal submanifold in $\SS^6$ which is sufficiently $C^3$-close to $M$
must coincide with $M$ itself up to isometry.
\end{theo}
The above $C^3$-close condition is not essential, which can be weakened with the help of Allard's regularity theorem. For completeness, we give the proof in \S6 by following a standard approach. We refer Colding--Minicozzi \cite{Colding-Minicozzi}, Evans--Lotay--Schulze \cite{E-L-S}, Schulze \cite{Schulze}, Sun--Zhu \cite{Sun-Zhu} for rigidity of round cylinders and products of spheres of high codimension.
\begin{rema}
The following problem was collected by Yau in \cite[P9, Problem 31]{Yau}: is there a nontrivial continuous family of compact, embedded codimension one minimal hypersurface in $\SS^n$? Theorem \ref{link-LO-rigidty} yields a negative answer for the link of the Lawson--Osserman cone in the higher-codimensional setting.
\end{rema}

At last, we use Theorem \ref{theo.dim} to study the asymptotic behavior of minimal submanifolds converging to the Lawson--Osserman cone $\BC$ at infinity, which is given in \S 7. In particular, we derive decay rates for minimal submanifolds (with boundary) converging to $\BC$ as below.
\begin{theo}\label{decay-rate-LO*}
    Let $M$ be a minimal submanifold in $\RR^7\setminus \bar{B_1(0)}$ and $M$ can be written as a graph over $\BC\setminus\bar{B_2(0)}$ with the graphic function $Z$. Then $\abs{Z(x)-E} = O(|x|^{-\fh{1}{2}})$ with some constant vector $E\in \RR^7$ if $|Z(x)|=o(|x|)$,  and $\abs{Z(x)} = O(|x|^{-\fh{3}{2}})$ if $\abs{Z(x)} = O(|x|^{-(\fh{1}{2}+\delta_0)})$ for some small $\delta_0>0$. 
\end{theo}
\begin{rema}
   Ding--Yuan \cite{Ding_Yuan}  resolved the singularities of the Lawson--Osserman cone $\BC$ by families of
smooth minimal graphs. In particular, they proved that there exists a family of  complete minimal graphs on $\RR^4$ converging to $\BC$ at infinity whose decay rate is
$O(|x|^{-\fh{1}{2}})$, and a family of smooth minimal graphs on $\RR^4\setminus \{0\}$ converging to $\BC$ at infinity whose decay rate is $O(|x|^{-\fh{3}{2}})$ at infinity. Hence, the decay orders in Theorem \ref{decay-rate-LO*} are optimal.
\end{rema}

\section{Notation and Preliminaries } \label{sec.pre}

Let $\HH = \RR\{1,\bi,\bj,\bk\}$ be the quaternions, where $\bi$, $\bj$ and $\bk$ satisfy the
condition $$\bi\cdot\bj = - \bj\cdot\bi= \bk,\ \bj\cdot\bk = -\bk\cdot\bj =  \bi,\ \bk\cdot\bi = -\bi\cdot\bk= \bj,\ \bi^2 = \bj^2 = \bk^2 = -1.$$ The conjugate
of $a +b\bi+c\bj+d\bk$ is $a-b\bi-c\bj-d\bk,$ where $a,b,c,d \in \RR.$
We use the notation $\mathbb{F}^{\ell\times n}$ for the set of all $\ell \times n$ matrices with entries from $\mathbb{F},$ where $\mathbb{F} = \RR$ or $\CC.$
A matrix $A \in  \mathbb{F}^{\ell\times n} $ is usually written as 
\begin{gather}
    A = \left[\begin{smallmatrix}
        a_{11}& \cdots &a_{1n} \\
        \vdots&        &\vdots\\
        a_{\ell1}& \cdots &a_{\ell n}
    \end{smallmatrix}\right]\,,
\end{gather}
or as $A =[a_{ij}]$ for short with $a_{ij} \in \mathbb{F}$ for $i = 1,\cdots,\ell$; $j = 1,\cdots,n$. The number
$a_{ij}$ which occurs at the entry in the $i$-th row and the $j$-th column is referred to as the $(i, j)$-th
entry of the matrix $[a_{ij}]$.
We denote  $O_n:=[a_{ij}]\in \mathbb{F}^{n\times n} $ with $a_{ij} = 0 $ for $1\leq i,j\leq n $ and denote $I_{n} :=[a_{ij}]\in \mathbb{F}^{n\times n} $ with $a_{ij} = \delta_{ij} $ for $1\leq i,j\leq n $. For a matrix $A = [a_{ij}] \in \CC^{n\times n}$, the matrix whose entries are the complex conjugates of the respective entries of $A$, is denoted by $\bar A,$ i.e., $\bar A = [e_{ij}]$ where $e_{ij} = \bar a_{ij}$ for $1\leq i,j \leq n$. Then the \emph{adjoint}, of $ A$, denoted by $\bar A^t$,
is the transpose of $\bar A$, i.e., $\bar A^t = [b_{ij}]$ where $b_{ij} = \bar a_{ji}$ for $1\leq i,j\leq n $. A \emph{Hermitian} matrix is one for which its adjoint is the same
as itself, and all eigenvalues of a Hermitian matrix are real numbers. Given linear spaces $V_1$ and $V_2$ over the field $\CC $ or $\RR$ with fixed bases $\mathfrak{S}_1 = \{u_1,\cdots,u_n\}$ and  $\mathfrak{S}_2 = \{v_1,\cdots,v_m\}$ respectively, then the direct sum space $V_1\oplus V_2$ has the fixed basis $\mathfrak{S}_1 \oplus \mathfrak{S}_2  = \{u_1,\cdots,u_n,v_1,\cdots,v_m\}.$

Given a Riemannian manifold $(M,g),$ we use the notations $\d^M,$ $\la_M$ to denote the Levi--Civita connection, the Laplacian operator on $M$, respectively. 
Let $d_M(x,y)$ be the (Riemannian) distance between $x$ and $y$ on $M.$
Let $\fX(M)$ denote the set of all smooth vector fields on $M\,.$ A vector field $X\in\fX(M)$ is  \emph{Killing} on $M$ if and only if $\cL_{X}g = 0,$ where $\cL_{X}g$ denotes the Lie derivative of the Riemannian metic $g$ under the flow generated by the vector filed $X.$ If $X$ is a Killing field on $M$, then $\fX(M)\ni Y \to\d^M_Y X$ is a skew symmetric
(1,1)-tensor on $M$, and $X$ is divergence free, i.e., $\mathrm{div}_gX = 0.$

For $m \geq 1,$  
let $B_r(p) = \{x \in \RR^{m+1}: |x-p|<r\}$ denote the ball in the Euclidean space $\RR^{m+1}$ with radius $r>0$ and centered at $p\in\RR^{m+1}$. Let $\SS^m(r)=\partial B_r(0)$ be the hypersurface in $\RR^{m+1}$ with induced metric, and $\SS^m=\SS^m(1)$ for short.
For a subset $U\subset \RR^{m+1}$ and $r>0$, let $B_r(U)$ denote the $r$-neighborhood of $U$, i.e., $B_r(U)= \bigcup_{x\in U} B_{r}(x)$.
Given two subsets $U\, , V \subset \RR^{m+1},$ we denote $d_{\cH}(U,V)$ as the \emph{Hausdorff distance} between $U$ and $V, $
i.e., $$d_{\cH}(U,V) = \inf \{\varepsilon >0: V \subset B_\eps(U) \text{ and } U\subset B_{\eps}(V)\}\, . $$

Let $\bar{\d}$ denote the Levi--Civita connection on $\RR^{m+1}$ and let $\d$ denote the Levi--Civita connection on $\SS^{m},$ i.e., $\d_X Y: = (\db_X Y)^T$ for each $ X\,,Y \in \fX(\SS^{m}),$ where $(\cdots)^T$ denotes the projection onto the tangent bundle $T\SS^{m}$(see Xin \cite{xin} for instance).
Given a submanifold $M^n$ in $\SS^{m},$ the second fundamental form $B$ of $M$ in $\SS^{m}$ is defined by $$B(X,Y) : = (\d_XY)^N= \d_XY - \d^M_XY \qquad\qquad\text{ for each } X,Y  \in \fX(M)\,, $$  where $(\cdots)^N$ denotes the projection onto the normal bundle $NM$ of $M$.  Let $\d^\bot,\la^{\bot}$ be the connection and the Laplacian operator on the normal bundle $NM,$ respectively. Namely,
\begin{gather*}
\df_{X} Y =(\d_X Y)^N,\qquad  \la^{\bot} Y = \sum_{i =1}^{n}\left(\df_{e_i}\df_{e_i}Y-\df_{(\d_{e_i}^Me_i)} Y\right)
\end{gather*}
for each $X\in \fX(M)$ and $Y \in \Gamma(NM)$,
where $ \{e_i\}_{i=1}^n$ is an orthonormal tangent frame on $M$.

Let $H_M$ denote the mean curvature vector of $M^n$ in $\SS^{m},$ defined by the trace of $B$, i.e., $H_M =\sum_{i=1}^n B(e_i , e_i )$.  If $H_M \equiv 0$, $M$ is minimal in $\SS^{m}.$ We define a linear operator 
$$\widetilde{B} (Y) :=\sum_{i,j=1}^n\ang{B(e_i,e_j),Y}B(e_i,e_j)\qquad\qquad \text{ for each } Y\in \Gamma (NM). $$
We use $\bar B$ to denote the second fundamental form of $M$ in $\RR^{m+1}$ defined by
$$\bar B(X,Y):=\bar\d_{X}Y -\d^M_{X}Y\qquad\qquad \text{ for each } X,Y  \in \fX(M)\,. $$

 Let $\fs \fo(m+1)\subset \RR^{(m+1)\times(m+1)}$ be the set of   anti-symmetric matrices with $\dim \fs \fo(m+1) = \fh{1}{2}m(m+1)$. For each Killing field $Y$ on $\SS^{m}$, there is an $A \in \fs \fo(m+1)$ so that
$$Y=\{xA: x \in \SS^{m}\}.$$ 
Let $\mathfrak{i}\fs \fo(\SS^{m})$ denote the space of Killing fields on $\SS^{m}$. 
For an embedded minimal submanifold $M^n\subset\SS^{m},$ let $\cK(M)$ denote the space of \emph{Killing Jacobi fields} on $M$, where each Killing Jacobi field $Y$ is a normal field on $M$ so that
$$Y=\{(xA)^N: x \in M\subset \SS^{m}\}\qquad\qquad\text{ for some } A \in \fs \fo(m+1).$$ 
Simons proved in \cite{Simons} that all Killing Jacobi fields on $M$ are Jacobi fields.
 We introduce following vector spaces:
\begin{itemize}
    \item $\cV(\SS^{m}) =\{X|_{\SS^{m}}: X$ is a parallel vector field on $\RR^{m+1} \} ;$
    \item $\cV(M) = \{(X|_{M})^N: X\in\cV(\SS^{m})\}$, where $(\cdots)^N$ is the projection onto the normal bundle $NM$.
\end{itemize}
Simons proved in \cite{Simons} that the space $\cV(M)$ is a subspace of the eigenfunction space with  eigenvalue $-n$ of the Jacobi operator.

 Given complex vectors $z=(z_1,\cdots, z_\ell),w= (w_1,\cdots,w_\ell) \in \CC^\ell,$ we denote $\bar z = (\bar{ z}_1,\cdots ,\bar{ z}_\ell)$ and $\ang{z,w} := \sum_{i=1}^\ell z_i\bar w_i$ and $\abs{z}^2 = \ang{z,z}.$
For a Riemannian manifold $(M,g)$, we denote $\mu_g$ as the volume element of $M.$
We shall work with the following function spaces on $M$:
\begin{itemize}
 \item $L^2(M;\RR^\ell)$ measurable functions $f:M \to \RR^\ell$ with $\int_M |f|^2 d\mu_g < +\infty$;
 \item $L^2(M;\CC^\ell)$ measurable functions $f:M \to \CC^\ell$ with $\int_M |f|^2 d\mu_g < +\infty$.

  \end{itemize}
Recall a classical definition of weighted norm for vector-valued functions on Riemannian manifold 
$(M \times  \RR^+,g_M+dt^2)$. 
We denote $\d$ as the Levi--Civita connection of the Riemannian manifold $(M\times \RR^+,g_M+dt^2)$. If $Y\in C^k(M \times  \RR^+;\RR^m)$, we define its weighted norm:

  $$\|Y\|_{k,\beta}:=\sup \bigg\{e^{\beta t}\sum_{i=0}^k|\d^i Y(x,t)|: (x,t)\in M \times \RR^+\bigg\}\,\qquad \text{ for every } k\in \NN\, ,\beta>0 .$$
We denote $\|Y\|_\beta: = \|Y\|_{0,\beta}$ for short.
If $ \|Y\|_{\beta} < \infty\,$, we will also write
$Y=O(e^{-\beta t})\,$.

\section{Harmonic functions on 3-spheres and  Hopf map} \label{sec.harmo.hopf}

\subsection{Harmonic functions on 3-spheres}

Recall that $\RR^4 \cong \CC^2 \cong \HH$ with $(a,b,c,d) \leftrightarrow  (a+b\bi, c+ d \bi) \leftrightarrow  a + b\bi + c \bj + d \bk,$ where $a,b,c,d \in\RR$.  We use $\cS^3=\{(z_1,z_2) \in \CC^2: \abs{z_1}^2 + \abs{z_2}^2 = 1 \} \subset \CC^2$ to denote a smooth manifold without Riemannian metric (Noting Riemannian manifold $\SS^3$ has the standard induced metric from $\RR^4$).
We can also identify $\RR^2 $  with $\CC,$ and $\cS^3$ is identified with SU(2). Moreover, SU(2) is isomorphic to the unit quaternions Sp(1), i.e., 
   \begin{gather} \label{identity.number.field}
       (z_1,z_2) \leftrightarrow 
         \left[
         \begin{smallmatrix}
               z_1  & z_2 \\
         -\bar{z}_2  & \bar {z}_1
         \end{smallmatrix}
         \right]
      \leftrightarrow z_1+z_2\bj \text{ for any } z_1,z_2 \in \CC \text{ with }
      \abs{z_1}^2 + \abs{z_2}^2 = 1\,.
   \end{gather}
Since $\cS^3$ has Lie group structure by identifying $\cS^3$ with Sp(1),
there are right invariant vector fields on $\cS^3,$ denoted by $\di, \dj\,$  and  $\dk$, i.e., for any $q\in \HH$ with $\abs{q}= 1,$ 
\begin{equation}\label{harmonic.vec.field}
\di|_q = \bi q\,, \hspace{1cm}\dj|_q=\bj q\,, \hspace{1cm}\dk|_q= \bk q\,.
\end{equation}
We define two functions $z,w\in C^\infty(\cS^3;\CC)$ by
\begin{gather}
\label{harmonic.fun1}
    z(z_1+z_2\bj) =z_1\,, \hspace{1cm}   
    w(z_1+z_2\bj) = z_2\, 
\end{gather} 
for any $z_1, z_2 \in \CC$ with $\abs{z_1}^2 + \abs{z_2}^2 =1\,.$ 
Then we have 
\begin{gather} \label{harmonic.func1.con}
\bar{z}(z_1+z_2\bj)=\bar{z}_1, \hspace{1cm} \bar{w}(z_1+z_2\bj)= \bar{z}_2\,.
\end{gather}
For each $f\in C^\infty(\cS^3;\CC),$  combined \eqref{harmonic.vec.field} we have
\begin{gather}\label{harmonic.der}
\begin{aligned}
    \di f(z_1+z_2\bj ) &= \pd{}{t}\bigg|_{t=0} f(e^{t\bi}(z_1+z_2\bj))\,,\\
    \dj f(z_1+z_2\bj ) &= \pd{}{t}\bigg|_{t=0} f(e^{t\bj}(z_1+z_2\bj))\,,\\
    \dk f(z_1+z_2\bj ) &= \pd{}{t}\bigg|_{t=0} f(e^{t\bk}(z_1+z_2\bj))\,.
\end{aligned}
\end{gather}
Moreover, \begin{gather}\label{harmonic.bar}
\de_{i}\bar f = \bar{\de_if} \hspace{0.5cm} \text{ for } i=1,2,3.
\end{gather} 
Combined \eqref{harmonic.fun1}, \eqref{harmonic.func1.con},  \eqref{harmonic.der} and \eqref{harmonic.bar}, we have
\begin{gather}\label{harmonic.fun.der}
\begin{aligned}
    \di z&= \sqrt{-1}z\,,& \di w&= \sqrt{-1}w\,,& \di \bar{z}&= -\sqrt{-1}\bar{z}\,,& \di \bar{w}&= -\sqrt{-1}\bar{w} \,;\\
    \dj z&= -\bar{w}\,,& \dj w&= \bar{z}\,,& \dj\bar{z} &= -w\,,&\dj\bar{w} &= z \,;\\
    \dk z&=-\sqrt{-1}\bar{w}\,,&\dk w&= \sqrt{-1}\bar{z}\,,& \dk \bar{z}
    &= \sqrt{-1} w\,,& \dk \bar{w}& = -\sqrt{-1} z \,.
\end{aligned}
\end{gather}
For $i=1,2,3$, we define 1-forms $\omega_i$ on $\cS^3$ by the expression: $df= \de_1f\omega_1+\de_2f\omega_2+\de_3f\omega_3.$
If $q= q_0 +q_1\bi +q_2 \bj +q_3 \bk \in \cS^3,$ then $d q =d q_0 + dq_1\bi + dq_2 \bj + dq_3 \bk.$ Let $\omega  = (d q)q^{-1},$ then $ \omega = \omega_1 \bi+\omega_2 \bj+\omega_3 \bk\,.$ Since
$$ 0=d d q = (d \omega) q -\omega\wedge d q= (d \omega -\omega\wedge \omega)q,$$ we have $d \omega =\omega\wedge \omega.$
In particular, we get that
\begin{gather}\label{harmonic.diff}
d \omega_1 = 2\omega_2 \wedge \omega_3\,, \hspace{1cm} d \omega_2=- 2\omega_1 \wedge \omega_3\,, \hspace{1cm} d \omega_3= 2\omega_1 \wedge \omega_2\,.
\end{gather}
Since $d \omega (X,Y) = X(\omega(Y)) - Y(\omega (X)) - \omega([X,Y])$ for any $X,Y \in \fX(\cS^3)$ and any $\omega \in \Omega^1(\cS^3),$ combined \eqref{harmonic.diff}, we have
\begin{gather}\label{harmonic.lie.bra}
 [\dj,\dk] =-2\di\,, \hspace{1cm} [\di,\dk] =2\dj\,, \hspace{1cm}[\di,\dj] =-2\dk\,.
\end{gather}

Given a Riemannian manifold $(\cS^3,g_{\tau})$   with  
\begin{gather} \label{sph.var.metric}
g_{\tau} = \tau \omega_1^2 + \omega_2^2 +\omega_3^2 \, \text{ and } \tau>0,
\end{gather}
combined \eqref{harmonic.diff}, the Laplacian operator of the Riemannian manifold $(\cS^3,g_{\tau})$ is 
\begin{gather} \label{eq.lap.tau}
\la_{\tau} = \fh{1}{\tau} \di^2 +\dj^2 +\dk^2.
\end{gather}
Due to \eqref{harmonic.lie.bra}, the vector field $\di$ on $\cS^3$ can commute with the differential operator  $\la_{\tau}$, i.e., 
\begin{gather}
\di \la_{\tau} f = \la_{\tau} \di f
\end{gather}
for each $f \in C^\infty (\cS^3;\CC).$ Combined \eqref{harmonic.lie.bra}, we have  
\begin{gather}
\cL_{\di}g_{\tau} = \cL_{\di}(\tau \omega_1^2 + \omega_2^2 +\omega_3^2) = 0\qquad \text{ for any } \tau>0,
\end{gather}
where $\cL_{\di}g_{\tau}$ denotes the Lie derivative of the Riemannian metic $g_\tau$ under the flow generated by the vector filed $\di.$
So $\di$ is a Killing field on $(\cS^3,g_{\tau})$ for any $\tau >0,$  
 and $ \Div_{g_\tau} \di = 0$ for any $\tau >0.$ We denote $g=g_1$ and $\la=\la_1$ for short. By Stokes' formula,

    \begin{gather}
    \begin{aligned}
    \int_{\SS^3}\zeta_2\di \zeta_1 d\mu_{g}=& -\int_{\SS^3}\zeta_1\di \zeta_2 d\mu_{g} +  \int_{\SS^3} \di (\zeta_1 \zeta_2) d\mu_{g}\\
    =&-\int_{\SS^3}\zeta_1\di \zeta_2 d\mu_{g} +\int_{\SS^3} \left(\Div_{g}(\zeta_1 \zeta_2 \di) -\zeta_1 \zeta_2 \Div_{g} \di\right) d\mu_{g} \\
    =&  -\int_{\SS^3}\zeta_1\di \zeta_2 d\mu_{g}
    \end{aligned}
    \end{gather}
for any $\zeta_1,\zeta_2 \in C^{\infty}(\cS^3;\CC).$

Tanno \cite{Tanno} calculated  eigenvalues of the Laplacian operator on  Riemannian manifold $(\cS^3,g_{\tau}).$
We introduce the following complex function spaces 
with $k,\ell \in\ZZ.$ 
\begin{itemize}
    \item $\tilde \cP_k := \{$homogeneous polynomials on $\CC^2$ of degree $k \}$ for $k\geq 0$; Otherwise $\tilde \cP_k := \{0\} ;$
     \item $ \cP_k: = \{f|_{\cS^3}:f\in \tilde \cP_k\};$
     \item $\cQ_k : =\{f \in \cP_k:\la f+k(k+2) f=0 \};$
    \item $\cQ_k^{k-2\ell} : = \{ f \in\cQ_k : \di f=(k-2\ell)\sqrt{-1}f \}.$
\end{itemize}
 Combined with the spaces defined above, we provide a more precise characterization of the eigenfunction spaces of the Laplacian operator on  Riemannian manifold $(\cS^3,g_{\tau})$, which is crucial for the proof of Theorem \ref{theo.dim}.
\begin{lemm}\label{spec}
    Let $\gamma_k(k\geq 0)$ be the $k$-th eigenvalue (counted without multiplicity) of $\la$, i.e., $\la f = -\gamma_k f$ for $f \in C^\infty(\cS^3;\CC)$ with $f\neq 0,$ then $\gamma_k = k(k+2),$ and the corresponding eigenfunction space is $\cQ_k$ 
    with $\dim \cQ_k=(k+1)^2$ (as a complex linear space). Moreover,  $\cQ_k = \oplus_{\ell=0}^{k}\cQ_k^{k-2\ell}$ and  the dimension of $\cQ_k^{k-2\ell}$ is $k+1$ (as a complex linear space). For every $k\geq 0,0\leq \ell\leq k$, the constant $\gamma_{k,\ell}:=k(k+2) + (k-2\ell)^2(\fh{1}{\tau}-1)$ is an eigenvalue of $\la_{\tau}$ with the corresponding eigenfunction space $\cQ_k^{k-2\ell}.$
\end{lemm}
\begin{proof}
The lemma is classical, so we sketch the proof.  For  $ k,\ell\in \ZZ, k\geq 0, 0\leq \ell \leq k,$
let us denote  
\begin{itemize}
    \item $\tilde \cQ_k :=\{f \in \tilde \cP_k:f \text{ is harmonic on } \CC^2 \};$
     \item $\tilde \cP_k^{k-2\ell} := \{ f \in \tilde \cP_k :f_0 = f|_{\cS^3}, \di f_0 =(k-2\ell)\sqrt{-1}f_0 \};$
    \item $\tilde \cQ_k^{k-2\ell} := \{ f \in \tilde \cQ_k : f_0 = f|_{\cS^3}, \di f_0 =(k-2\ell)\sqrt{-1}f_0 \}.$
\end{itemize}
     Then it is easy to see that the space $\cQ_k$ is isomorphic to the space $\tilde \cQ_k$ and the space $\cQ_k^{k-2\ell}$ is isomorphic to the space $\tilde \cQ_k^{k-2\ell}$. Moreover, we denote  function spaces
\begin{itemize}
    \item  $ (\abs{z_1}^2+ \abs{z_2}^2) \tilde \cP_{k-2} := \{(\abs{z_1}^2+ \abs{z_2}^2)f(z_1,z_2): f\in \tilde \cP_{k-2} \text{ and }(z_1,z_2)\in \CC^2 \};$
     \item $(\abs{z_1}^2+ \abs{z_2}^2)\tilde \cP_{k-2}^{k-2\ell}:=\{(\abs{z_1}^2+ \abs{z_2}^2)f(z_1,z_2): f\in \tilde \cP_{k-2}^{k-2\ell} \text{ and } (z_1,z_2)\in \CC^2 \} .$
\end{itemize}
 From \cite[P76]{Axler}, we have 
   \begin{gather}
   \tilde \cP_k = \tilde Q_k \oplus(\abs{z_1}^2+ \abs{z_2}^2) \tilde \cP_{k-2},\\
   \tilde \cP_k^{k-2\ell} = \tilde \cQ_k^{k-2\ell}\oplus(\abs{z_1}^2+ \abs{z_2}^2)\tilde \cP_{k-2}^{k-2\ell}\,.
   \end{gather}
   So we get 
   \begin{gather}
   \dim \cQ_k=\dim \tilde \cQ_k =\dim \tilde \cP_k-\dim \tilde \cP_{k-2} =   \tbinom{k+3}{3}-\tbinom{k+1}{3} = (k+1)^2\,,
   \end{gather}
   \begin{gather}
   \begin{aligned}
   \text{ and } \dim \cQ_k^{k-2\ell} =& \dim \tilde\cQ_k^{k-2\ell} = \dim \tilde \cP_k^{k-2\ell}-\dim \tilde \cP_{k-2}^{k-2\ell}\\
   =& (k-\ell+1)(\ell +1)-\left((k-2)-(\ell-1) +1\right)\left((\ell-1)+ 1\right) = k+1 \,.
   \end{aligned}
   \end{gather}
   Thus, we obtain $\cQ_k = \oplus_{\ell=0}^{k}\cQ_k^{k-2\ell},$ 
   which follows directly from the definitions of $\cQ_k$ and $\cQ_k^{k-2\ell}.$  

    Restrict $-\la_{\tau}$ to $\cQ_k^{k-2\ell},$ then for any $f \in \cQ_k^{k-2\ell},$
    \begin{gather}
    \begin{aligned}
   -\la_{\tau}f=& -\left(\fh{1}{\tau} \di^2 +\dj^2+\dk^2\right)f \\
   = & -\la f+\left(1-\fh{1}{\tau}\right)\di^2f \\
    = &\left(k(k+2)+(k-2\ell)^2\left(\fh{1}{\tau}-1\right)\right)f\,.
    \end{aligned}
    \end{gather}
This completes the proof.
\end{proof}

  \subsection{Group actions and  Hopf map}

  It is well known that SU(2) is a 2-sheeted simply covering group of SO(3) \cite[Proposition 1.19]{Hall}. 
    We write out the group homomorphism directly. Let $\Phi : \mathrm{SU(2)} \to  \mathrm{SO(3)} $ defined by
      \begin{gather}\label{def.su2.hom}
     \left[
      \begin{matrix}
               z_1  & z_2 \\
         -\bar {z}_2  &  \bar{z}_1
    \end{matrix}
      \right]  \to 
      \left[
      \begin{matrix}
         \abs{z_1}^2 - \abs{z_2}^2 & 2\mathrm{Re}(z_1 \bar{ z}_2)  & 2\mathrm{Im}(z_1 \bar{ z}_2 )\\
         -2\mathrm{Re}(z_1z_2)  & \mathrm{Re}(z_1^2-z_2^2) & \mathrm{Im}(z_1^2+z_2^2)\\
        2\mathrm{Im}(z_1z_2)  & \mathrm{Im}(-z_1^2+z_2^2) & \mathrm{Re}(z_1^2+z_2^2)
\end{matrix}
      \right]           
      \end{gather}
      for any $ z_1,z_2 \in \CC $ with $|z_1|^2+|z_2|^2 = 1\,.$
     Then  $\Phi $ is a group homomorphism \cite[Proposition 1.19]{Hall}, i.e., 
     \begin{gather}
     \Phi(PQ) = \Phi(P) \Phi(Q) \qquad\text{for each } P,Q \in \mathrm{SU(2)},
     \end{gather} where $PQ$ is the matrix product, and $\Phi(P) \Phi(Q)$ is the matrix product of $\Phi(P), \Phi(Q) \in $ SO(3).
     \begin{rema}
         Our definition of $\Phi$ is slightly different from the definition of $\Psi_U$ in \cite[P25]{Hall}. Actually, $\Phi(U) = A^{-1}\Psi_UA,$ where  $U = \left[ \begin{smallmatrix}
               z_1  & z_2 \\
         -\bar {z}_2  &  \bar{z}_1
         \end{smallmatrix}
         \right]$ and $A = \left[\begin{smallmatrix}
             1 & 0 & 0\\
             0 & 1 & 0\\
             0 & 0 & -1
         \end{smallmatrix}\right]$.
     \end{rema}
 
Lawson--Osserman \cite{Lawson} observed that the group SU(2) acts naturally on $\CC^2$ and on $\RR^3$ (as SO(3) = SU(2)/$\ZZ_2$) and 
the Hopf map $\eta$ defined in \eqref{eq.Hopf.map} is \emph{equivariant} with respect to these actions. 
In fact, from \eqref{identity.number.field} and \eqref{def.su2.hom}, we immediately have
     \begin{gather} \label{prop.eta.hom}
     \eta(PQ) = \eta(P) \Phi(Q)
     \end{gather}
     for each
      $P= (z_1,z_2)= z_1+z_2 \bj \in $ SU(2) and $ Q=(w_1,w_2) =  w_1+w_2 \bj\in$ SU(2), where $\eta(P) \in \RR^3$ is a row vector.

Moreover, Lawson--Osserman \cite{Lawson} also observed that SU(2) can act on $\RR^7.$
Due to  \eqref{identity.number.field}  and \eqref{def.su2.hom}, there exists a group homomorphism 
 \begin{equation} \label{eq.isometry.action}
      \Psi:\mathrm{ SU(2)} \to \mathrm{SO(7)},\hspace{1cm}
     (x,y)\Psi(Q) = (xQ,y\Phi(Q))
 \end{equation}
 for any $x\in\CC^2 \cong \HH \cong \RR^4,$ $ y \in \RR^3,$ and $Q \in $ SU(2) $ \cong $ Sp(1) $\subset$ SO(4). 
 So  SU$(2)$ can act on $\RR^7$ as an isometry subgroup of $\RR^7$ by the group homomorphism $\Psi.$

\section{The Jacobi operator and frames on the link of Lawson--Osserman cone}\label{subsec.jac.oper}

    Recall that  the graph $\BC$ of $u$ defined in \eqref{eq.lo.lip.graph} is the celebrated \emph{Lawson--Osserman cone} \cite{Lawson}. Thus, the link $M:=\BC \cap \SS^6$ is the graph of the map $h:\SS^3\left(\fh{2}{3}\right) \to \SS^2\left(\fh{\sqrt{5}}{3}\right)$, where 
     \begin{gather} \label{graph.func.link}
     \begin{aligned}
     h(z_1,z_2)= \fh{3\sqrt{5}}{4} (\abs{z_1}^2 - \abs{z_2}^2,2z_1\bar {z}_2)\qquad \text{ for any } z_1,z_2 \in \CC \text{ with } \abs{z_1}^2 + \abs{z_2}^2 = \fh{4}{9}.
     \end{aligned}
     \end{gather}
      Due to \eqref{prop.eta.hom} and \eqref{eq.isometry.action}, both $M$ and $\BC$
     are invariant as subsets in $\RR^7$ under the action of SU(2)  by the group homomorphism $\Psi$,  i.e.,
     for any $ Q \in $ SU(2), any $ p_1 \in M \subset \RR^7,$ and any $ p_2 \in \BC\subset \RR^7\,,$ 
     \begin{gather} \label{prop.invariant}
     p_1 \Psi(Q) \in M \,, \hspace{1cm} p_2 \Psi(Q)\in \BC\,.
     \end{gather}

Let $\{e_1,e_2,e_3\} $ and  $\{e_4,e_5,e_6 \}$ denote two global orthonormal frames for the tangent bundle and the normal bundle, respectively, to be selected later. Let $\nu$ be the position vector of $M$ in $\RR^7$, which is normal to $\SS^6.$ 
      We use $(x_1 +\sqrt{-1}y_1,x_2+\sqrt{-1}y_2,x_3,x_4 +\sqrt{-1}y_4)$ to represent the vector $(x_1,y_1,x_2,$ $y_2,x_3,x_4,y_4) \in \RR^7.$
      From the map $h,$  we have a diffeomorphic map $G:  \SS^3\left(\fh{2}{3}\right) \to M\subset \SS^6$, where
    \begin{gather}\label{jacob.graph.11}
    \begin{aligned}
        G(z_1,z_2) = \left(z_1,z_2,\fh{3\sqrt{5}}{4} (\abs{z_1}^2 - \abs{z_2}^2), \fh{3\sqrt{5}}{2} z_1\bar {z}_2\right) 
        \end{aligned}
        \end{gather}
       for any $z_1,z_2\in \CC $ with $\abs{z_1}^2+\abs{z_2}^2=\fh{4}{9}.$
        Then we can naturally extend $G$ to a map from $\RR^4$ to $\RR^7,$ i.e.,
        \begin{gather}
        \begin{aligned}
    &G(x_1,y_1,x_2,y_2)
       \\ 
     =&  \left(x_1,y_1,x_2,y_2, \fh{3\sqrt{5}}{4}(x_1^2+y_1^2-x_2^2 -y_2^2),\fh{3\sqrt{5}}{2}(x_1x_2+y_1y_2),\fh{3\sqrt{5}}{2}(y_1x_2-y_2x_1)\right).  
       \end{aligned}
    \end{gather} 
  At the point $b_0 = (\fh{2}{3},0,0,0) \in \SS^3\left(\fh{2}{3}\right),$  the tangent space $T_{b_0}\SS^3\left(\fh{2}{3}\right)$ is spanned by orthonormal vectors $b_1 = (0,1,0,0), b_2 = (0,0,1,0),$ and $ b_3 = (0,0,0,1).$ Then 
    \begin{gather}\label{jacobi.frame}
    \begin{aligned}
         \d^{\RR^4}_{b_1}G\big|_{b_0} =& \pa{G}{y_1} \bigg |_{b_0}=(0,1,0,0,0,0,0), \\ 
         \d^{\RR^4}_{b_2}G\big|_{b_0} =& \pa{G}{x_2}\bigg|_{b_0}= (0,0,1,0,0,\sqrt{5},0), \\
         \d^{\RR^4}_{b_3}G\big|_{b_0} =& \pa{G}{y_2}\bigg|_{b_0}=(0,0,0,1,0,0,-\sqrt{5}).
         \end{aligned}
    \end{gather}
  Then combined \eqref{jacobi.frame}, at the point $p_0= G(b_0)= (\fh{2}{3}, 0,0,0,\fh{\sqrt{5}}{3},0,0),$ we can assume 
      \begin{gather}\label{e123p0}
        \begin{aligned}
            \nu=\left(\fh{2}{3}, 0,0,0,\fh{\sqrt{5}}{3},0,0\right) ,&\hspace{1cm} \ e_1=(0,1,0,0,0,0,0),\\
            e_2= \left(0,0,\fh{1}{\sqrt{6}},0,0,\sqrt{\fh{5}{6}},0\right),& \hspace{1cm} \  e_3= \left(0,0,0,\fh{1}{\sqrt{6}},0,0,-\sqrt{\fh{5}{6}}\right).
        \end{aligned}
    \end{gather}
      Since $M \subset \SS^3\left(\fh{2}{3}\right) \times \SS^2\left(\fh{\sqrt{5}}{3}\right)\subset \SS^6,$
   at the point $p_0$ we can take $$e_4 = \left(\fh{\sqrt{5}}{3},0,0,0,-\fh{2}{3},0,0\right) \in N_{p_0}\left(\SS^3\left(\fh{2}{3}\right) \times \SS^2\left(\fh{\sqrt{5}}{3}\right)\right) \subset N_{p_0}M.$$
    Moreover, at the point $p_0 \in M$, we may choose $e_5,e_6 \in T_{p_0}\left(\SS^3\left(\fh{2}{3}\right) \times \SS^2\left(\fh{\sqrt{5}}{3}\right)\right)\cap N_{p_0}M$ by 
$$e_5 = \left(0,0,\sqrt{\fh{5}{6}},0,0,-\fh{1}{\sqrt{6}},0\right),\hspace{1cm}e_6= \left(0,0,0,\sqrt{\fh{5}{6}},0,0,\fh{1}{\sqrt{6}}\right).$$

  Noting that the link $M$ is SU$(2)$ invariant as in \eqref{prop.invariant}. For every $Q= \fh{3}{2}(z_1+z_2\bj)$ with $z_1,z_2\in \CC$ and $\abs{z_1}^2+ \abs{z_2}^2 = \fh{4}{9}$,  we set $p_Q = p_0\Psi(Q)$ where $p_0= (\fh{2}{3}, 0,0,0,\fh{\sqrt{5}}{3},0,0),$ 
then $$p_Q = \left(z_1,z_2,\fh{3\sqrt{5}}{4} (\abs{z_1}^2 - \abs{z_2}^2),\fh{3\sqrt{5}}{2} z_1\bar{z}_2\right) = \nu(p_Q).$$ 
It is easy to see that the 
map: 
\begin{gather}
\Psi_0:\text{SU(2) } \to M,\quad 
\Psi_0(Q)=  p_0\Psi(Q) \text{ for every } Q \in SU(2)\,, 
\end{gather}
is a diffeomorphism. Moreover,  due to \eqref{eq.isometry.action}, SU$(2)$ can act on $\RR^7 $ as a subgroup of SO$(7).$
So we can obtain global vector fields $$e_i(p_Q) = e_i(p_0)\Psi(Q)\qquad \text{on } M \text{ for }1\leq i \leq 6.$$ Combined the fact that $\{e_1(p_0), e_2(p_0), e_3(p_0)\} $ is an orthonormal basis of the tangent space $T_{p_0}M$, we obtain the global orthonormal frame $\{e_1,e_2,e_3\}$ of the tangent bundle of $M$. Moreover, since $\{e_4(p_0), e_5(p_0), e_6(p_0)\} $ is an orthonormal basis of the normal space $N_{p_0}M$, we have the global orthonormal frame $\{e_4,e_5,e_6\}$ of the normal bundle of $M$ in $\SS^6$. All $e_i(p)$ for $1\leq i \leq6 $ can be calculated as $\nu(p),$   i.e.,
\begin{gather}\label{subsec.frame}
\begin{aligned}
          &e_1  = \Big( \fh{3}{2}\ii z_1, \fh{3}{2}\ii z_2,0,0\Big) ,\\
          &e_2 =  \sqrt{\fh{3}{8}}\Big(- \bar{z}_2,  \bar {z}_1,-\fh{3\sqrt{5}}{2}( \bar{z_1z_2} +  z_1z_2),\fh{3\sqrt{5}}{2}( z_1^2 - \bar z_2^2)\Big),\\
         &e_3 =  \sqrt{-\fh{3}{8}}\Big(- \bar{z}_2,  \bar {z}_1,-\fh{3\sqrt{5}}{2}( \bar{z_1z_2} - z_1z_2),-\fh{3\sqrt{5}}{2}(z_1^2 + \bar z_2^2)\Big),\\
        &e_4 =  \Big(\fh{\sqrt{5}}{2}z_1,\fh{\sqrt{5}}{2}z_2,-\fh{3}{2}(\abs{z_1}^2 - \abs{z_2}^2),-3z_1\bar{z}_2\Big),\\
        &e_5 =  \sqrt{\fh{15}{8}} \Big(- \bar{z}_2,  \bar {z}_1,\fh{3\sqrt{5}}{10}( \bar{z_1z_2} +  z_1z_2),-\fh{3\sqrt{5}}{10}( z_1^2 - \bar z_2^2)\Big),\\
       & e_6 =  \sqrt{-\fh{15}{8}}\Big(- \bar{z}_2,  \bar {z}_1,\fh{3\sqrt{5}}{10}( \bar{z_1z_2} - z_1z_2),\fh{3\sqrt{5}}{10}( z_1^2 + \bar z_2^2)\Big),\\
        &\nu =   \Big(z_1,z_2,\fh{3\sqrt{5}}{4}(\abs{z_1}^2- \abs{z_2}^2),\fh{3\sqrt{5}}{2}z_1 \bar{z}_2\Big)  .
    \end{aligned}
\end{gather}

    Next, we calculate the connections w.r.t. the above frames.
    We denote $\d $ as the Levi--Civita connection on $\SS^6.$ See Appendix \ref{Appedix.A} for more details about the calculations of connections. We list the results below, i.e.,
    \begin{gather}  \label{eq.connection}
  \begin{aligned}
     \d_{e_1} e_1  &=  -\fh{\sqrt{5}}{2} e_4,    &
     \d_{e_1} e_2  &=  -\fh{11}{4} e_3+ \fh{\sqrt{5}}{4} e_6, &  \d_{e_1} e_3  &=  \fh{11}{4} e_2 -\fh{\sqrt{5}}{4} e_5, \\
      \df_{e_1} e_4  &= 0, & \df_{e_1} e_5 &=-\fh{7}{4} e_6 ,& \df_{e_1} e_6 &= \fh{7}{4} e_5,\\
      \d_{e_2} e_1  &=  \fh{1}{4} e_3 +\fh{\sqrt{5}}{4} e_6,    &
     \d_{e_2} e_2  &=  \fh{\sqrt{5}}{4} e_4, &  \d_{e_2} e_3  &=  -\fh{1}{4} e_1 , \\
      \df_{e_2} e_4  &= \fh{3}{4}e_5, & \df_{e_2} e_5 &=-\fh{3}{4} e_4 ,& \df_{e_2} e_6 &= 0,\\
      \d_{e_3} e_1  &=  -\fh{1}{4} e_2 -\fh{\sqrt{5}}{4} e_5,    &
     \d_{e_3} e_2  &=  \fh{1}{4} e_1, &  \d_{e_3} e_3  &=  \fh{\sqrt{5}}{4} e_4 , \\
      \df_{e_3} e_4  &= \fh{3}{4}e_6, & \df_{e_3} e_5 &=0 ,& \df_{e_3} e_6 &= -\fh{3}{4}e_4.
      \end{aligned} 
      \end{gather}
Let $\d^M$ denote the Levi--Civita connection on $M.$
       We find that 
      \begin{gather}
      \d^{M}_{e_1}e_1 = 0\,, \hspace{1cm}
      \d^{M}_{e_2}e_2 = 0\,, \hspace{1cm}
      \d^{M}_{e_3}e_3 = 0\,,
      \end{gather}
and  \begin{gather}
      \begin{aligned}
      \la^{\bot} e_4 = -\fh{9}{8}e_4\,, \hspace{1cm} 
      \la^{\bot} e_5= -\fh{29}{8}e_5\, ,\hspace{1cm} 
    \la^{\bot} e_6 = -\fh{29}{8}e_6.
      \end{aligned}
      \end{gather}
For all $X \in \Gamma(NM)$ and $\xi \in C^\infty (M),$ 
      \begin{gather}
      \begin{aligned}
      \la^{\bot} (\xi X )=\xi  \la^{\bot} X+ 2\sum_{i=1}^3 (e_i\xi)\df_{e_i} X + (\la_M \xi) X  . 
      \end{aligned}
      \end{gather}
Given $X = \xi_4e_4 + \xi_5e_5 +\xi_6e_6 \in \Gamma(NM)$ with $\xi_4,\xi_5,\xi_6 \in C^\infty (M),$ we have
      \begin{gather} \label{eq.cal.normalla}
          \begin{aligned}
              \la^{\bot}(\xi_4e_4) &=  (\la_M \xi_4)e_4  +\xi_4\la^{\bot}e_4 + 2\left(\fh{3}{4}(e_2\xi_4)e_5 + \fh{3}{4}(e_3\xi_4)e_6\right)
              \\& = (\la_M \xi_4)e_4-\fh{9}{8}\xi_4e_4 + \fh{3}{2}(e_2\xi_4)e_5 + \fh{3}{2}(e_3\xi_4)e_6\,, \\
              \la^{\bot}(\xi_5e_5) &=  (\la_M \xi_5)e_5  +\xi_5\la^{\bot}e_5 + 2\left(-\fh{7}{4}(e_1\xi_5)e_6 -\fh{3}{4}(e_2\xi_5)e_4\right)
              \\& = (\la_M \xi_5)e_5-\fh{29}{8}\xi_5e_5 - \fh{7}{2}(e_1\xi_5)e_6 - \fh{3}{2}(e_2\xi_5)e_4 \,, \\
              \la^{\bot}(\xi_6e_6) & =  (\la_M \xi_6)e_6  +\xi_6\la^{\bot}e_6 + 2\left(\fh{7}{4}(e_1\xi_6)e_5 -\fh{3}{4}(e_3\xi_6)e_4\right) 
              \\ & = (\la_M \xi_6)e_6-\fh{29}{8}\xi_5e_5 + \fh{7}{2}(e_1\xi_6)e_5 - \fh{3}{2}(e_3\xi_6)e_4\,.
        \end{aligned}
      \end{gather}
   Recall  that 
   \begin{gather}
   \begin{aligned}
  B(e_i, e_j) = (\d_{e_i} e_j)^{N}\,, \hspace{1cm}
   \widetilde{B} (X) = \sum_{i,j = 1}^3 \ang{B(e_i,e_j),X} B(e_i,e_j).\\
   \end{aligned}
   \end{gather}
   Hence, combined \eqref{eq.connection},
   we have 
   \begin{gather}\label{eq.cal.tildeB}
   \tilde B(e_4) = \fh{15}{8}e_4\,, \hspace{1cm} \tilde B(e_5) = \fh{5}{8}e_5\,,\hspace{1cm}
   \tilde B(e_6) = \fh{5}{8}e_6\,.
   \end{gather}

   In our frame defined above, $\la^{\bot}$ and $\tilde B$ are diagonalized simultaneously, which is convenient for the following study of the Jacobi operator \eqref{jocob1}.
   Moreover, due to the sectional curvature of  $\SS^6$ equaling  to 1,  
   \begin{gather} \label{eq.cal.seccurvature}
   \tilde{R}(X) =\sum_{i=1}^{3}R(X,e_i)e_i =3X\,.
   \end{gather}

     As in  \cite{Simons}, the \emph{Jacobi operator} $\cJ_M$ of the link $M$ of Lawson--Osserman cone satisfies
     \begin{gather} \label{jocob1}
            \cJ_M X = \la^{\bot} X + \tilde{B} (X) + \tilde{R}(X)\,.
      \end{gather}
Combined \eqref{eq.cal.normalla}, \eqref{eq.cal.tildeB} and \eqref{eq.cal.seccurvature}, we have
\begin{equation} \label{jcb1} 
\begin{aligned} 
       \cJ_M X  =& \left( \la_M \xi_4 +\fh{15}{4}\xi_4 -\fh{3}{2}e_2\xi_5 -\fh{3}{2}e_3\xi_6\right)e_4\\
         &+\left(\la_M \xi_5+\fh{7}{2}e_1\xi_6 +\fh{3}{2}e_2\xi_4\right)e_5
         +\left(\la_M \xi_6 -\fh{7}{2} e_1\xi_5+ \fh{3}{2}e_3 \xi_4\right)e_6\,.
\end{aligned}
\end{equation}
We say that a constant $\lambda\in \RR$ is an eigenvalue of the Jacobi operator $\cJ_M,$ if $\cJ_M X = -\lambda X$ for some $ X \in \Gamma(NM)$ with $X\neq 0$.
We have an estimation about the dimension of the space of Killing Jacobi fields on $M$.
\begin{lemm} \label{estimate.dim.Killing.jacobi}
    The space $\cK(M)$ of Killing Jacobi fields of the link of Lawson--Osserman cone has dimension $\geq 17.$
\end{lemm}
\begin{proof}
   
   Define a linear map $$\cT : \fs\fo(7) \to \cK(M), \cT(A)|_{x} = (xA  )^N$$ for any $A \in \fs \fo(7),$ $ x \in M \subset \SS^{6}$.  Since $\dim(\fs\fo(7)) = 21,$ we only need to prove $\dim(\ker( \cT)) \leq 4.$ Given 
   $A=  
  \left[ \begin{smallmatrix}
    J & W\\
    -W^t &K
    \end{smallmatrix}\right]
     \in \fs\fo(7),$ where $J \in\fs\fo(4),$ $K\in \fs\fo(3),$ and $W \in \RR^{4\times 3}$. Due to  \eqref{eq.frame.polynomails},
    $$ \nu Ae_4^t  = -\fh{27}{8}(x_1,y_1,x_2,y_2)W(x_1^2+y_1^2-x_2^2-y_2^2,2(x_1x_2+y_1y_2),2(x_2y_1-x_1y_2))^t.$$ If $A\in \ker(\cT),$ then the above equality implies $W=0.$ So $\ker(\cT) \subset \fs\fo(4) \oplus\fs\fo(3) \subset \fs\fo(7)  .$ 
    Let $$\cT_1 : \fs\fo(4) \oplus\fs\fo(3) \to \cK(M), \cT_1 = \cT|_{\fs\fo(4) \oplus\fs\fo(3)},$$ then $\ker(\cT_1) = \ker(\cT).$  So we only need to check that $\dim (\imag(\cT_1)) \geq 5$ due to $\dim(\fs\fo(4) \oplus\fs\fo(3))=9$.

   Now we take $K =0, $ and $J_{ij}=(E_{ij} -E_{ji})\oplus O_3 \in \fs\fo(4) \oplus\fs\fo(3) \subset \fs\fo(7)$ for  $  1\leq i,j \leq4,$ where $E_{ij} : = [a_{k\ell}] \in \RR^{4\times 4}$ with $a_{k \ell} = \delta_{ki}\delta_{j\ell } $ for $1\leq k,\ell \leq4$.  We list the value of the linear map $\cT_1$ in the following Table \ref{table.value.linear}.

\begin{table}[!htbp]
    \centering \small
    \caption{The value of linear map $\cT_1$}
    \label{table.value.linear}
    \renewcommand\arraystretch{1.4}
        \tabcolsep=0.5cm
        \scalebox{0.75}{
        \begin{tabular}{|c | c c c c c c|}
    \hline
     $ J_{ij} $ & $J_{12}$ & $J_{13}$ & $J_{14}$ & $J_{23}$ & $J_{24}$ & $J_{34}$  \\  \hline
      $\sqrt{\fh{8}{15}} \nu J_{ij}e_5^t$  & $x_1y_2+x_2y_1$ & $x_1^2+x_2^2$ & $-x_1y_1 +x_2y_2$ & $x_1y_1 -x_2y_2$ & $-y_1^2 -y_2^2$ & $-x_1y_2-x_2y_1$ \\ \hline
      
         $\sqrt{\fh{8}{15}} \nu J_{ij}e_6^t$  & $-x_1x_2+y_1y_2$ & $x_1y_1 +x_2y_2$ & $x_1^2 +y_2^2$ & $x_2^2 + y_1^2$ & $x_1y_1 +x_2y_2$ & $x_1x_2-y_1y_2 $ \\  \hline
         \end{tabular} 
         }  
         \end{table}
It is clear that $\dim(\imag(\cT_1)) \geq 5.$ This completes the proof.
\end{proof}
 \begin{rema}
     Actually we can prove that $\dim \cK(M) = 17$ by writing out  $\ker(\cT)$ below: $$\ker (\cT) = \Span \{J_{12}+J_{34},J_{13}+J_{24}+2K_{12},J_{14}-J_{23}-2K_{13},J_{12}+K_{23}\}\subset \fs\fo(7) .$$ 
      Here, $K_{ij} = O_4\oplus (E_{ij} - E_{ji}) \in \fs\fo(4) \oplus\fs\fo(3) \subset \fs\fo(7) $ for $1\leq i,j\leq 3,  $ and $E_{ij} : = [a_{k\ell}] \in \RR^{3\times 3}$ with $a_{k \ell} = \delta_{ki}\delta_{j\ell } $ for $1\leq k,\ell \leq3$.
 \end{rema}

From Simons \cite{Simons}, $\dim \cV(M)\geq n+q+1$ and the eigenfunction space with  eigenvalue $-n$ of the Jacobi operator of $M$ has dimension $\geq n+q+1$ unless $M^n$ is the totally geodesic minimal submanifold $\SS^n$ in $\SS^{n+q}$, where $\cV(M)$ is defined in $\S 2$ (see also Urbano \cite{Urbanoclli}). 

\begin{lemm}\label{estimate.dimension.-3}
    Let $V_{-3}$ be the eigenfunction space with eigenvalue $-3$ of the Jacobi operator of the link of Lawson--Osserman cone, then $\dim V_{-3} \geq 7.$
\end{lemm}

  \section{Spectra of the Jacobi operator of the link of Lawson--Osserman cone} \label{sec.pectra.Jac}

 Let $M$ denote the link of Lawson--Osserman cone. From \eqref{jacob.graph.11}, $\cS^3$ is diffeomorphic to $M$ through the map $\tilde G,$ i.e., \begin{gather}
 \tilde G:\cS^3\to M \hspace{1cm}
 \tilde G(p) = G\left(\fh{2}{3}p\right) \text{ for every } p\in \cS^3.
 \end{gather}
There are right invariant vector fields $\di,\dj,\dk$ on $\cS^3$ defined as in \eqref{harmonic.vec.field}.  
Combined \eqref{jacobi.frame} and \eqref{e123p0}, we have  $$\tilde G_*(\di) = \fh{2}{3}e_1, \hspace{1cm} \tilde G_*( \dj)= \sqrt{\fh{8}{3}}e_2,\hspace{1cm}  \tilde G_*(\dk)= \sqrt{\fh{8}{3}}e_3.$$ 
From
\begin{equation} \label{eq.isom.link}
\begin{aligned} 
       (\tilde G^*g_M)(\partial_i,\partial_j)=g_M(\tilde G_*\partial_i,\tilde G_*\partial_j)\qquad\text{for each }i,j\in\{1,2,3\},
\end{aligned}
\end{equation}
we conclude that $ (M,g_M)$ is isometric to  $(\cS^3, \fh{8}{3} g_{\fh{1}{6}})$, where $g_{\fh{1}{6}}=\fh{1}{6}\omega_1^2+\omega_2^2+\omega_3^2$ is defined in \eqref{sph.var.metric}. 

For a vector field $X= \xi_4e_4 + \xi_5e_5 +\xi_6e_6 \in \Gamma(NM)$ with $\xi_4,\xi_5,\xi_6 \in C^\infty (M),$ we define $\phi_k=\xi_k\circ \tilde G\in C^\infty (\cS^3)$ for each $k=4,5,6$.
By \eqref{jcb1}, the system  $\cJ_M X = -\fh{3}{8}\lambda X$ $(\lambda \in \RR)$  is equivalent to the following system on $\cS^3$, i.e.,
 \begin{equation} \label{Jacobi.131}
\begin{cases}
\fh{3}{8}\la_{\fh{1}{6}} \phi_4 +\fh{15}{4} \phi_4 -\fh{3}{2}\sqrt{\fh{3}{8}} \dj \phi_5 -\fh{3}{2}\sqrt{\fh{3}{8}} \dk \phi_6 &=-\fh{3}{8}\lambda  \phi_4\\
     \fh{3}{8}\la_{\fh{1}{6}} \phi_5+\fh{21}{4}\di \phi_6 +\fh{3}{2}\sqrt{\fh{3}{8}}\dj\phi_4& = -\fh{3}{8}\lambda  \phi_5\\
     \fh{3}{8}\la_{\fh{1}{6}} \phi_6-\fh{21}{4}\di \phi_5 +\fh{3}{2}\sqrt{\fh{3}{8}}\dk \phi_4 &=-\fh{3}{8}\lambda  \phi_6
     \end{cases}\,,
\end{equation}
where $\la_{\fh{1}{6}}$ is defined in \eqref{eq.lap.tau}. 
   By simplifying the system \eqref{Jacobi.131}, we can obtain the following system 
  \begin{gather}
  \begin{aligned}
  \begin{cases}
         -(\la_\fh{1}{6} +10) \phi_4  +\sqrt{6} \dj \phi_5 + \sqrt{6} \dk \phi_6 &= \lambda \phi_4 \\
        - \sqrt{6}\dj \phi_4 -\la_\fh{1}{6} \phi_5-14\di \phi_6 &= \lambda \phi_5\\
        -\sqrt{6} \dk \phi_4 +14 \di \phi_5 - \la_\fh{1}{6} \phi_6 &= \lambda \phi_6
        \end{cases}
     \end{aligned}\,.
     \end{gather}

We will extend the real Jacobi operator $\mathcal{J}_M$ to the complex $\mathbf{J}_M$ on complex functional space, which is more convenient for calculations. In particular, eigenvalues of $\mathbf{J}_M$ and the dimension of its corresponding eigenfunction space are the same as ones of $\mathcal{J}_M$.
 
    Let $L$ be a differential operator on $ \cS^3, $ i.e.,
    \begin{align}\label{oper.L}
        L\left(\begin{array}{c} f_1\\ f_2 \\ f_3 \end{array} \right) =\left( \begin{array}{lcr}
           -(\la_{\fh{1}{6}}+10) & \sqrt{6}\dj &\sqrt{6} \dk  \\
            -\sqrt{6}\dj & -\la_{\fh{1}{6}} &-14\di\\
            -\sqrt{6}\dk & 14\di & -\la_{\fh{1}{6}} 
        \end{array}\right) \left(\begin{array}{c} f_1\\ f_2 \\ f_3 \end{array} \right)
    \end{align} 
    for any $f_1,f_2,f_3 \in C^\infty(\cS^3;\CC).$
    Let $\tilde L$ be another  differential operator on $ \cS^3, $ i.e.,
    \begin{align}\label{oper.L1}
        \tilde L\left(\begin{array}{c} f_1\\ f_2 \\ f_3 \end{array} \right) =  \left( \begin{array}{lcr}
        -\la &  &  \\
               &-\la & \\
               &  & -\la 
        \end{array} \right) \left(\begin{array}{c} f_1\\ f_2 \\ f_3 \end{array} \right)\,.
    \end{align}
    Due to \eqref{harmonic.lie.bra}, $L$ can commute with $\tilde L.$
    
For each $k\geq 0,$ suppose that the linear space $\cQ_k$ defined in section \ref{sec.harmo.hopf} has the fixed basis $\mathfrak S_k = \{u_1,\cdots ,u_{n}\}$. Let $\cQ_{k,1},\cQ_{k,2},\cQ_{k,3}$ denote the linear space with the fixed basis 
 $\mathfrak S_{k,1} = \left\{\left(\begin{smallmatrix}u_1 \\
0 \\
 0\end{smallmatrix} \right),\cdots , \left(\begin{smallmatrix} u_{n} \\
 0 \\
 0\end{smallmatrix} \right) \right\}\,,
 \mathfrak S_{k,2} = \left\{\left(\begin{smallmatrix}0 \\
 u_1 \\
 0\end{smallmatrix} \right),\cdots , \left(\begin{smallmatrix} 0 \\
 u_{n} \\
 0\end{smallmatrix} \right)  \right\}\,, 
 \mathfrak S_{k,3} = \left\{\left(\begin{smallmatrix}0 \\
0 \\
 u_1\end{smallmatrix} \right),\cdots , \left(\begin{smallmatrix} 0 \\
 0 \\
 u_n\end{smallmatrix} \right)  \right\} ,$  respectively.
  The $k$-th eigenvalue of $\tilde L$ is the same as the $k$-th eigenvalue of $-\la$, and the corresponding eigenfunction space is 
 \begin{gather} \label{space.Vk}
 V_k : = \cQ_{k,1} \oplus\cQ_{k,2} \oplus\cQ_{k,3}
 \end{gather}
 with fixed basis $\hat{\mathfrak{S}}_k =\mathfrak S_{k,1}\oplus \mathfrak S_{k,2}\oplus \mathfrak S_{k,3}. $ Moreover,
 $\dim V_k =3(k+1)^2$ due to Lemma \ref{spec}.   
Since $L$ can commute with $\tilde L,$  $V_k$ is also an invariant subspace of $L$ for all $k\geq 0.$
 As below, we consider differential operator $L$ on Riemannian manifold $(\SS^3,g)$. Then $L$ is a self-adjoint operator on each $V_k$, i.e.,

\begin{gather}
     \int_{\SS^3} \ang{X,LX} d\mu_{g} = \int_{\SS^3} \ang{LX,X} d\mu_{g}
\end{gather}
for every $X\in V_k.$
 To prove Theorem \ref{theo.dim}, we only need to consider the operator $L$ on each $V_k$.

\begin{lemm} \label{lemm:definitive.estimate}
   If $k \geq 5,$ then the operator $L$ is positive definite on $V_k,$ i.e., $$\int_{\SS^3} \ang{X,LX} d\mu_{g} > 0 \text{ for every } X \in V_k\setminus \{0\}.$$  
   
\end{lemm}
\begin{proof}
     Given $X=(f_1,f_2,f_3)^t \in C^\infty(\SS^3;\CC^3)
     $, 
integrating by parts gives
    \begin{gather}
    \begin{aligned}
     \int_{\SS^3}    \ang{ X,LX}d\mu_{g}=\int_{\SS^3}& \Big(f_1(-\la_{\fh{1}{6}} -10)\bar{f}_1 -  f_2\la_{\fh{1}{6}}\bar{f}_2  -  f_3\la_{\fh{1}{6}} \bar{f}_3
         -14f_2\di\bar{f}_3 +14f_3\di \bar{f}_2\\
         &+\sqrt{6}f_1 \dj \bar{f}_2-\sqrt{6}f_2\dj \bar{f}_1 + \sqrt{6} f_1 \dk \bar{f}_3 -\sqrt{6} f_3\dk\bar{f}_1\Big)d\mu_{g} \\
        =\int_{\SS^3}&  \Big(\sum_{\ell=1}^3(\abs{\d f_\ell}^2 +  5\abs{\di f_\ell}^2 ) -10\abs{f_1}^2 -14f_2\di\bar{f}_3+14f_3\di \bar{f}_2\\
       & +\sqrt{6}f_1 \dj \bar{f}_2-\sqrt{6}f_2\dj \bar{f}_1 + \sqrt{6} f_1 \dk \bar{f}_3 -\sqrt{6} f_3\dk\bar{f}_1\Big)d\mu_{g}\,.
    \end{aligned}
    \end{gather}
For each $k \geq 5, \gamma_{k} \geq 35$ due to Lemma \ref{spec}, which infers
    $$ \int_{\SS^3} \abs{\d f}^2 d\mu_{g}\geq 35 \int_{\SS^3}\abs{f}^2d\mu_{g}\qquad\text{for each }f \in \cQ_k.$$
For each $X \in V_k, $ by Cauchy--Schwarz inequality we have
    \begin{gather}
    \begin{aligned}
    \int_{\SS^3}\ang{ X,&LX}d\mu_{g}\geq  \int_{\SS^3}\bigg(\fh{1}{2}\sum_{\ell=1}^3\left(\abs{\d f_\ell}^2+35\abs{f_\ell}^2 +10 \abs{\di f_\ell}^2\right)  
    -10\abs{f_1}^2 \\
    & -\left(\fh{49}{5} \abs{f_2}^2 + 5 \abs{\di f_3}^2\right)
    - \left(\fh{49}{5} \abs{f_3}^2 + 5 \abs{\di f_2}^2\right)
    -\left(3\abs{f_1}^2+ \fh{1}{2} \abs{\dj{f_2}}^2\right) \\
    &-\left(3\abs{f_2}^2+ \fh{1}{2} \abs{\dj{f_1}}^2\right) -\left(3\abs{f_1}^2+ \fh{1}{2} \abs{\dk{f_3}}^2\right) -\left(3\abs{f_3}^2+ \fh{1}{2} \abs{\dk{f_1}}^2\right)\bigg)d\mu_{g}\\
    \geq & \int_{\SS^3}  \left(\fh{3}{2}\abs{f_1}^2 + \fh{47}{10}\abs{f_2}^2 + \fh{47}{10}\abs{f_3}^2\right)d\mu_{g} 
\geq\fh{3}{2}\int_{\SS^3}|X|^2d\mu_{g}.
\end{aligned}
    \end{gather}
This completes the proof.
   
\end{proof}

By Lemma \ref{lemm:definitive.estimate}, we only need to consider $L$ on $V_k$ with $k=0,1,2,3,4.$ Actually we will calculate the characteristic polynomial of $L$ on $V_k$ for $k=0,1,2,$ calculate the signs of $L$ on $V_3$ and verify that $L$ is positive definite on $ V_4.$ 

For every $k=0,1,2,3,4$ and every $0\leq \ell \leq k$, we suppose that $\cQ_k^{k-2\ell}$ (defined in section \ref{sec.harmo.hopf}) has fixed base $\mathfrak{S}_k^{k-2\ell}$. From Lemma \ref{spec},  each space $\cQ_k$ has the  fixed base $\mathfrak S_k$, i.e.,

\noindent
\begin{minipage}{0.48\textwidth}
\begin{itemize}
    \item $\cQ_0 = \cQ_0^0\,;$
    \item $\cQ_1 = \cQ_1^1\oplus \cQ_1^{-1} \,;$
    \item $\cQ_2=\cQ_2^2\oplus\cQ_2^{-2}\oplus\cQ_2^{0} \,;$
    \item $\cQ_3=\cQ_3^3\oplus\cQ_3^{-3}\oplus\cQ_3^1\oplus\cQ_3^{-1}\,;$
    \item $\cQ_4 = \cQ_4^4\oplus \cQ_4^{-4}\oplus\cQ_4^2\oplus\cQ_4^{-2}\oplus\cQ_4^0  \,.$
\end{itemize}
\end{minipage}
\hfill
\begin{minipage}{0.48\textwidth}
\begin{itemize}
    \item $\mathfrak{S}_0 = \mathfrak{S}_0^0\,;$
    \item $\mathfrak{S}_1 = \mathfrak{S}_1^1\oplus \mathfrak{S}_1^{-1} \,;$
    \item $\mathfrak{S}_2=\mathfrak{S}_2^2\oplus \mathfrak{S}_2^{-2}\oplus \mathfrak{S}_2^{0} \,;$
    \item $\mathfrak{S}_3=\mathfrak{S}_3^3\oplus \mathfrak{S}_3^{-3}\oplus \mathfrak{S}_3^1\oplus \mathfrak{S}_3^{-1}\,;$
    \item $\mathfrak{S}_4 = \mathfrak{S}_4^4\oplus \mathfrak{S}_4^{-4}\oplus \mathfrak{S}_4^2\oplus \mathfrak{S}_4^{-2}\oplus \mathfrak{S}_4^0  \,.$
\end{itemize}
\end{minipage}
We shall find the appropriate basis $\mathfrak{S}_k^{k-2\ell}$ of the space $\cQ_k^{k-2\ell}$ for every $k=0,1,2,3,4$ and every $ 0\leq \ell \leq k$ so that the matrix representation $L_k$ of $L$ on $V_k$ under the basis $\hat{\mathfrak{S}}_k =\mathfrak S_{k,1}\oplus \mathfrak S_{k,2}\oplus \mathfrak S_{k,3}$ is a Hermitian matrix, i.e.,
\begin{gather} \label{eq.matrix.rep}
L{\hat{\mathfrak{S}}_k} = {\hat{\mathfrak{S}}_k} L_k\,. 
\end{gather}
We list bases  $\mathfrak{S}_k^{k-2\ell}$ below:
\begin{itemize}
\item $\mathfrak{S}_0^{0} = \{1\}\,;$ 
    \item $\mathfrak{S}_1^{1} = \{z,w\}, \ \mathfrak{S}_1^{-1} = \{\bar{z},\bar{w}\}\,;$
    \item 
$\mathfrak{S}_2^{2} = \{z^2,zw, w^2\}\,,  \ 
\mathfrak{S}_2^{-2} =\{\bar{z}^2 , \bar{zw}, \bar{w}^2\}\,  , \ \mathfrak{S}_2^{0} = \left\{\fh{1}{\sqrt{2}}(\abs{z}^2 - \abs{w}^2)\,,\sqrt{2}z\bar{w},\sqrt{2}\bar{z}w \right\}\,;$
\item 

$\mathfrak{S}_3^3= \{ z^3, z^2w, zw^2,w^3\}\,,\\
\mathfrak{S}_3^{-3}= \{\bar{z}^3,\bar{z}^2\bar{w},  \bar{w}^2 \bar{z}, \bar{w}^3\}\, ,\\
\mathfrak{S}_3^{1} = \left\{ \sqrt{3}z^2\bar{w},(\fh{1}{\sqrt{3}}\abs{z}^2 - \fh{2}{\sqrt{3}}\abs{w}^2)z, (\fh{2}{\sqrt{3}}\abs{z}^2  -\fh{1}{\sqrt{3}}\abs{w}^2)w, \sqrt{3}w^2\bar{z} \right\}\,,\\
\mathfrak{S}_3^{-1} = \left\{ \sqrt{3}\bar{z}^2w,(\fh{1}{\sqrt{3}}\abs{z}^2 - \fh{2}{\sqrt{3}}\abs{w}^2)\bar{z}, (\fh{2}{\sqrt{3}}\abs{z}^2  -\fh{1}{\sqrt{3}}\abs{w}^2)\bar{w}, \sqrt{3}\bar{w}^2z \right\}\,;$
 \item 
$\mathfrak{S}_4^4 = \{z^4, z^3w , z^2w^2,zw^3,w^4 \} \,,$\\
$\mathfrak{S}_4^{-4}= \{\bar{z}^4, \bar{z}^3\bar{w} , \bar{z}^2\bar{w}^2,\bar{z}\bar{w}^3,\bar{w}^4 \}\,,$\\
 $
\mathfrak{S}_4^{2} = \big\{-2z^3\bar{w}, z^2(\fh{1}{2}\abs{z}^2 - \fh{3}{2}\abs{w}^2), zw(\abs{z}^2- \abs{w}^2), w^2(\fh{3}{2}\abs{z}^2 - \fh{1}{2}\abs{w}^2) , 2 w^3\bar{z} \big\}\,,$\\
 $\mathfrak{S}_4^{-2} = \big\{-2\bar{z}^3w, \bar{z}^2(\fh{1}{2}\abs{z}^2 - \fh{3}{2}\abs{w}^2), \bar{zw}(\abs{z}^2- \abs{w}^2), \bar{w}^2(\fh{3}{2}\abs{z}^2 - \fh{1}{2}\abs{w}^2) , 2 \bar{w}^3z \big\}\,,$ \\
$\mathfrak{S}_4^0 = \Big\{ \sqrt{6}(z\bar{w})^2, \sqrt{\fh{3}{2}} z\bar{w} (-\abs{z}^2 + \abs{w}^2), \sqrt{\fh{1}{6}}(\abs{z}^4 + \abs{w}^4 -4 \abs{z}^2\abs{w}^2),\\ \,
\indent \indent \indent
 \sqrt{\fh{3}{2}} \bar{z}w (\abs{z}^2 - \abs{w}^2), \sqrt{6}(\bar{z}w)^2 \Big\} \,.   $

\end{itemize}
\begin{rema}
  Let us explain briefly the method to find basis for the space $\cQ_{k}^{k-2\ell}$ as above. For the case $\cQ_{4}^{2}$, it is easy to see that $z^3\bar w \in \cQ_{4}^{2},$ which implies $(z+w)^3(\bar{z-w}) \in \cQ_{4}^{2}.$ 
Since $$(z+w)^3(\bar{z-w})  = - z^3\bar w+ z^2(\abs{z}^2-3\abs{w}^2)+3zw(\abs{z}^2-\abs{w}^2)+w^2(3\abs{z}^2-\abs{w}^2)+w^3\bar z\,,$$ we conclude that $\cQ_{4}^{2} = \Span\{- z^3\bar w,  z^2(\abs{z}^2-3\abs{w}^2),3zw(\abs{z}^2-\abs{w}^2),w^2(3\abs{z}^2-\abs{w}^2),w^3\bar z\} $, then we can choose appropriate coefficients for these vectors to get $\mathfrak{S}_{4}^2$ so that the matrix representation of $L$ on $V_4$ under the basis $\hat{\mathfrak{S}}_4$ is a Hermitian matrix.
\end{rema}

\begin{proof}[Proof of Theorem \ref{theo.dim}]
 Recall that $V_k : = \cQ_{k,1} \oplus\cQ_{k,2} \oplus\cQ_{k,3}$ as defined in \eqref{space.Vk} and the differential operator $L$ is defined in \eqref{oper.L}, we have  the matrix representation $L_k$ as in \eqref{eq.matrix.rep} of $L$ on the space $V_k$ under the basis $\hat{\mathfrak{S}}_k$.  By Lemma \ref{spec}, for each $k\in \NN$ and $\ell = 0,1,\cdots,k$, the constant $\gamma_{k,\ell}:=k(k+2) +5 (k-2\ell)^2$ is an eigenvalue of $\la_{\fh{1}{6}}$ with the corresponding eigenfunction space $\cQ_k^{k-2\ell}.$ By \eqref{harmonic.fun.der}, we can calculate the other terms of the matrix representation $L_k$ as in \eqref{eq.matrix.rep} of $L$ on the space $V_k$ under the basis $\hat{\mathfrak{S}}_k$. In particular, we shall deal with the following cases separately.

Case I. 
$\dim V_0 = 1 \times 3 =3 $, and we have  the matrix representation $L_0$ of $L$ on the space $V_0$ under the basis $\hat{\mathfrak{S}}_0$, i.e., 
\begin{gather}
L_0=\diag(-10,0,0).
\end{gather}
The characteristic polynomial of $L_0$ is $\det(\lambda I_{3}-L_0)=\lambda^2(\lambda+10)$ and $\dim(\ker L_0) = 2.$

Case II. $\dim V_1 = (1+1)^2 \times 3 =12,$ and we have the matrix representation $L_1$ of $L$ on the space $V_1$ 
under the basis $\hat{\mathfrak{S}}_1$, i.e.,
\begin{gather} \label{matrix.rep.L_1}
\begin{aligned}
& L_1= \begin{bmatrix}
\begin{smallmatrix}
    -2I_4 & \sqrt{6}A_2  & \sqrt{6}A_3  \\
     -\sqrt{6}A_2   & 8I_4 & -14A_1 \\
     -\sqrt{6}A_3   &14A_1  &  8I_4
        \end{smallmatrix}
\end{bmatrix}, \text{ with } A_1=  \ii 
\left[
\begin{smallmatrix}
    I_2 & \\
    & -I_2      
\end{smallmatrix}
\right]\,,\\
&
A_2=  \begin{bmatrix}
\begin{smallmatrix}
     &  &  &1 \\
        &  &-1  & \\
        &1  &  &  \\
       -1 &  &  &
\end{smallmatrix}
\end{bmatrix}, \
A_3= \ii \left[
\begin{smallmatrix}
     &  &  & -1 \\
        & & 1 & \\
        & 1 &  &  \\
       -1 &  &  &
       \end{smallmatrix}
\right],
\text{ and by } \eqref{harmonic.fun.der},\de_{i} \mathfrak S_1 =\mathfrak S_1A_i \text{ for } i  =1,2,3.
\end{aligned}
\end{gather}
The characteristic polynomial of $L_1$ is $\det(\lambda I_{12}-L_1)=(\lambda+8)^4\lambda^4(\lambda-22)^4$ and  $\dim(\ker L_1) = 4.$ The details of calculations are given in Appendix \ref{app.l1}.

Case III.
 $\dim V_2 = (2+1)^2 \times 3 =27,$ and we have the matrix representation $L_2$ of $L$ on the space $V_2$ under the basis $\hat{\mathfrak{S}}_2$, i.e.,
\begin{gather}\label{matrx.rep.l2}
\begin{aligned}
   & L_2=\begin{bmatrix}
\begin{smallmatrix}
    T_2 -10I_9 & \sqrt{6}A_2  & \sqrt{6}A_3  \\
     -\sqrt{6}A_2   & T_2& -14A_1 \\
     -\sqrt{6}A_3   &14A_1  &  T_2
     \end{smallmatrix}
\end{bmatrix}  \text{ with }
T_2=\left[
\begin{smallmatrix}
   28I_3 & &\\
   & 28I_3& \\
   &  & 8I_3
   \end{smallmatrix}
\right],
A_1= 2\ii \left[
\begin{smallmatrix}
 I_3 & & \\
 & -I_3 & \\
 & & O_3
 \end{smallmatrix}
 \right], \\
 & A_2 = \sqrt{2}\left[
 \begin{smallmatrix}
     O_6 & -\beta \\
     \beta^t & O_3
     \end{smallmatrix}
 \right] , \ \beta = \begin{bmatrix}
\begin{smallmatrix}
    0 & -1 & 0 \\
    1 & 0&0 \\
    0 & 0 & 1\\
    0 & 0 & -1 \\
    1 & 0 & 0\\
    0 & 1 & 0    
\end{smallmatrix}
\end{bmatrix}, \ 
 A_3 = \sqrt{-2}\left[
 \begin{smallmatrix}
     O_6 & \gamma \\
     \gamma^t & O_3
     \end{smallmatrix}
 \right], \ 
\gamma = \begin{bmatrix}
    \begin{smallmatrix}
   0 & -1 & 0 \\
    1 & 0&0 \\
    0 & 0 & 1\\
    0 & 0 &1 \\
    -1 & 0 & 0\\
    0 & -1 & 0     
    \end{smallmatrix}
\end{bmatrix},\\
& \text{and by } \eqref{harmonic.fun.der},\ 
\de_{i}\mathfrak S_2
    =\mathfrak S_2 A_{i}  \text{ for } i=1,2,3\,.
\end{aligned}
 \end{gather}
The characteristic polynomial of $L_2$ is $\det(\lambda I_{27}-L_2)=(\lambda+8)^3\lambda^3(\lambda- 6)^9(\lambda-20)^6(\lambda-56)^6$  and  $\dim(\ker L_2) = 3.$ The details of calculations are given in Appendix \ref{app.l2}.

Case IV.
  $\dim V_3 = (3+1)^2 \times 3 =48,$ and  we have the matrix representation $L_3$ of $L$ on the space $V_3$ under the basis $\hat{\mathfrak{S}}_3$, i.e.,
\begin{gather}\label{matrix.rep.L_3}
\begin{aligned}
    &L_3=\begin{bmatrix}
\begin{smallmatrix}
    T_2 -10I_{16} & \sqrt{6}A_2  & \sqrt{6}A_3  \\
     -\sqrt{6}A_2   & T_2& -14A_1 \\
     -\sqrt{6}A_3   &14A_1  &  T_2
      \end{smallmatrix}
\end{bmatrix} \text{ with }
T_2=\left[ 
\begin{smallmatrix}
   60I_8 & \\
   & 20I_8    
\end{smallmatrix}
\right]\,,
A_1=\ii
\left[
\begin{smallmatrix}
    3 I_4& & & \\
         &-3 I_4 & & \\
         & & I_4 & \\
         & & &- I_4 \\
\end{smallmatrix}
\right], \\
&A_2=\left[
\begin{smallmatrix}
    O_8 & J_1  \\
    -J_1 & K_1    
\end{smallmatrix}
\right]\,,
J_1= \sqrt{3}\diag(1,-1,-1,-1,1,-1,-1,-1 ), \
K_1= 2\begin{bmatrix}
\begin{smallmatrix}
    & & & & & & &1 \\
    & & & & & & 1& \\
    & & & & & -1& & \\
    & & & & -1& & & \\
    & & &1 & & & & \\
    & &1 & & & & & \\
    &-1 & & & & & & \\
   -1 & & & & & & & \\
   \end{smallmatrix}
\end{bmatrix},\\
&A_3=\left[
\begin{smallmatrix}
  O_8 & J_2  \\
    J_2 & K_2    
\end{smallmatrix}
\right], \ J_2= \sqrt{-3}\diag(-1,1,1,1,1,-1,-1,-1),
\\
&K_2= 2\ii\begin{bmatrix}
\begin{smallmatrix}
    & & & & & & &-1 \\
    & & & & & & -1& \\
    & & & & & 1& & \\
    & & & & 1& & & \\
    & & &1 & & & & \\
    & &1 & & & & & \\
    &-1 & & & & & & \\
   -1 & & & & & & & \\
   \end{smallmatrix}
\end{bmatrix}, \text{ and by } \eqref{harmonic.fun.der}, \ \de_{i}\mathfrak S_3 = \mathfrak S_3A_{i}  \text{ for }i=1,2,3.
\end{aligned}
\end{gather}

We shall calculate the signs of $L_3$ in Appendix \ref{app.l3}, hence verifying that $L_3$ is semi-positive definite and  $\dim(\ker(L_3)) =8.$

Case V.
 $\dim V_4 = (4+1)^2 \times 3 =75,$ and we have the matrix representation $L_4$ of $L$ on the space $V_4$ under the basis $\hat{\mathfrak{S}}_4$, i.e.,
\begin{gather} \label{appendix.B.L4}
  \begin{aligned}
 & L_4=\begin{bmatrix} \begin{smallmatrix}
    T_2 -10I_{25} & \sqrt{6}A_2  & \sqrt{6}A_3  \\
     -\sqrt{6}A_2   & T_2& -14A_1 \\
     -\sqrt{6}A_3   &14A_1  &  T_2  
\end{smallmatrix}
\end{bmatrix}
 \text{with }
T_2 = \begin{bmatrix}
\begin{smallmatrix}
    104I_{10} & & \\
    & 44I_{10} & \\
    & & 24 I_5  
\end{smallmatrix}
\end{bmatrix} ,\\
&A_1 = \ii\left[
\begin{smallmatrix}
    4 I_5 & & & & \\
    & -4  I_5 & &  & \\
    & & 2  I_5 & &\\
    & & & -2  I_5 & \\
    & & & & O_5
        \end{smallmatrix}
\right], \ 
A_2 = \begin{bmatrix}
\begin{smallmatrix}
    O_5 & O_5 & -2I_5 & O_5 & O_5 \\
     O_5 & O_5 & O_5 & -2I_5 & O_5\\
     2I_5 & O_5 & O_5 & O_5 & -\sqrt{6}I_5 \\
     O_5 & 2I_5 & O_5 &  O_5  & -\sqrt{6}J_1\\
      O_5 & O_5 & \sqrt{6}I_5 & \sqrt{6}J_1 & O_5     
\end{smallmatrix}
\end{bmatrix}, \
J_1 = \begin{bmatrix}
\begin{smallmatrix}
    0&0&0&0&1\\
    0&0&0&-1&0\\
    0&0&1&0&0\\
    0&-1&0&0&0\\
    1&0&0&0&0
    \end{smallmatrix}
\end{bmatrix},\\
&A_3 = 
\ii\left[
\begin{smallmatrix}
     O_5 & O_5 & 2 I_5 & O_5 & O_5 \\
     O_5 & O_5 & O_5 & -2 I_5 & O_5\\
     2  I_5 & O_5 & O_5 & O_5 & \sqrt{6}I_5 \\
     O_5 & -2  I_5 & O_5 &  O_5  & -\sqrt{6}J_1\\
      O_5 & O_5 & \sqrt{6}I_5 & -\sqrt{6}J_1 & O_5 
       \end{smallmatrix}
\right] ,
\text{ and by } \eqref{harmonic.fun.der},\de_{i}\mathfrak S_4 =\mathfrak S_4 A_i \text{ for }i =1,2,3 
 \,.
\end{aligned}
\end{gather}
Then  $L_4$ 
 is positive definite, which we verify by computing its eigenvalues in Appendix \ref{code}.

Actually, we can  accurately compute the characteristic polynomial and the eigenvalues of the matrix representation $L_i$ of $L$ on the space $V_i$ under the basis $\hat{\mathfrak{S}}_i$ for each $i = 1,2,3,4$ with the help of code in Appendix \ref{code}.

Therefore, $$ \dim(\ker L) = \sum_{k=0}^4 
\dim(\ker L_k) = 2+4+3+8+0=17.$$
From the eigenvalues of $L_0,L_1,L_2,$   together with the fact that $L_3$ is semi-positive definite and $L_4$ is  positive definite, we conclude that the first eigenvalue of $L$
 is $-10$ with multiplicity 1, the second eigenvalue of $L$ is $-8$ with multiplicity 7 and the third eigenvalue of $L$ is 0 with multiplicity 17. Combined \eqref{jcb1} and \eqref{Jacobi.131}, the first eigenvalue of the Jacobi operator $\cJ_M$ is   $-10\times \fh{3}{8} = -\fh{15}{4}.$ Moreover, due to the fact that $-10$ is the eigenvalue of $L_0$ and we have the matrix representation $L_0$ of $L$ on $V_0$ under the basis $\hat{\mathfrak{S}}_0,$ the normal vector field $e_4$ in \eqref{subsec.frame} is the  corresponding eigenfunction of the first eigenvalue of $\cJ_M$.

Combined \eqref{jcb1}, \eqref{Jacobi.131} and Lemma \ref{estimate.dim.Killing.jacobi}, all Jacobi fields  are generated by Killing fields on $\SS^6$ and the dimension of the space of Jacobi fields on the link of Lawson--Osserman cone is $17$.
Combined \eqref{jcb1}, \eqref{Jacobi.131} and  Lemma \ref{estimate.dimension.-3},  the eigenfunction space $V_{-3}$ of $\cJ_M$ with eigenvalue $ -3$ has dimension $7$ and  $V_{-3}=\cV(M)$.

\end{proof}

\section{Rigidity of the link of Lawson--Osserman Cone}

In this section, we shall give a detailed proof of Theorem \ref{link-LO-rigidty}. We first calculate certain geometric quantities for graphs over a closed embedded minimal submanifold  $M^n $ in $ \SS^{n+q}$, where $M^n$ is not assumed to be orientable for $q \geq 2$, and $M^n$ is two-sided by a simple topological argument hence orientable 
for $q=1.$

 We define a map $G_M:\Gamma(NM)\to \SS^{n+q}$ by
\begin{gather}
G_M(X(p)) = \fh{p+X(p)}{\sqrt{1+\abs{X(p)}^2}}\qquad  \text{for each }X \in \Gamma(NM)\ \text{and }p \in M. \end{gather}
 If $\sum_{i=0}^{2} \abs{(\d^\bot)^i X} $ is sufficiently small, then $G_M(X)$ is an $n$-dimensional embedded submanifold in $\SS^{n+q}.$
Given $X \in \Gamma(NM)$ with $\sum_{i=0}^{2} \abs{(\d^\bot)^i X}$ sufficiently small, we define a map $F_X: M  \to \SS^{n+q} $ by
\begin{equation} \label{eq:sph.exp}
    F_X(p) = G_M(X(p))
    = \fh{p+X(p)}{\sqrt{1+\abs{X(p)}^2}}\, .
\end{equation}

Let the symbol '*' denote the linear combination, and let $X_{ij}^k(0\leq i,j\leq k)$ be smooth tensor functions depending only on $X,\df X,\cdots, (\df)^k X,$  $n ,q$ and $M,$ with all derivatives of  $X_{ij}^k$ uniformly bounded. Then we define  general vector functions
  \begin{gather}
      Q_k(X)= \sum_{i,j=0}^k(\df )^iX* (\df )^jX*X_{ij}^{k} \qquad\text{for each } k\geq 0,
  \end{gather}
satisfying that
 \begin{align} \label{sph.high.est} 
 \sum_{0\leq i \leq \ell}\abs{\d^iQ_k(X) } \leq  c_\ell \left( \sum_{0\leq i\leq \ell+k} \abs{(\df)^{i} X} \right)^2 
 \end{align} 
 for some constant $c_\ell >0$ depending only  on $n,q,k,\ell,$ and $M$, where $\d$ denotes the full gradient of $Q_k(X)$ in $\SS^{n+q}.$

\begin{lemm}\label{lemm:equation.graph}

Given a closed embedded minimal submanifold  $M^n $ in $ \SS^{n+q}$ and a fixed $\ell \in \NN$, there exists a sufficiently small $\eps_1>0$ depending only on $\ell, n,q$ and $M$, such that
if $  X  \in \Gamma(NM)$ with
\begin{equation} \label{eq:small.graph.funtion}
    \sum_{i=0}^{\ell+2} \abs{(\d^\bot)^i X} \leq \varepsilon_1 \,,
\end{equation}
and  $\Sigma = F_X(M)$ is minimal in $\SS^{n+q}$, then 
\begin{align} \label{eq:jacob.1} \la^{\bot} X+\tilde{B}(X) + nX = Q_2(X) \,,
\end{align}       
where $ Q_2(X) $ satisfies \eqref{sph.high.est} with $k = 2.$
\end{lemm}
Lemma \ref{lemm:equation.graph} is classical, and it follows from \cite[$\S 4$, 4.14]{isolated.singularity}.

Assume that $X \in \Gamma(NM)$ satisfies the condition of Lemma \ref{lemm:equation.graph}, then the system in  \eqref{eq:jacob.1} satisfies the Legendre--Hadamard condition \cite{Giaquita,Morrey}. The classical $L^2$ estimates and Schauder estimates \cite{Giaquita,Morrey} for the linear elliptic system 
give that
\begin{equation}
\begin{aligned}\label{estimate:L2}
    \sum_{\ell=1}^{m} \int_M \abs{(\df)^\ell X} ^2\leq c_{m,M} \int_M \abs{X}^2 \text{ for } m\in \NN, \text{ and }
\end{aligned}
\end{equation}

\begin{equation}\begin{aligned}\label{ineq.schauder.estimate}
    \abs{X}_{j,\alpha} := \sum_{\ell=0}^j \abs{(\df)^\ell X}_0 + [(\df)^j X]_\alpha \leq c_{j,\alpha,M}  \abs{X}_0 \ \ \text{ for each }  0 <  \alpha < 1 \text{ and } j\in \NN\,, \text{with}
\end{aligned}\end{equation}
\begin{gather*}  \abs{(\df)^\ell X}_0: = \norm{(\df)^\ell X}_{C^0(M)} 
\text{ and } [(\df)^j X]_\alpha := \sup\limits_{x,y\in M,x\neq y}\fh{|(\df)^j X(x)-(\df)^j X(y)|}{d_M(x,y)^\alpha}.
\end{gather*}
By Sobolev inequality, we have

\begin{gather}
    \abs{X}_0 \leq C \sum_{\ell=0}^{m} \left( \int_M \abs{(\df)^\ell X} ^2 \right)^{\fh{1}{2}}\  \text{ for }  m > \fh{n}{2} \text{ with } C= C(n,m,M)\,.
\end{gather}
Combined \eqref{estimate:L2}, we have 
\begin{gather}
    \abs{X}_0 \leq C_1 \left(\int_M \abs{X}^2 \right)^{\fh{1}{2}} \leq C_2 \abs{X}_0 \ \text{ with } C_1= C_1(n,M),C_2=C_2(n,M). \label{ineq:l2.norm}
\end{gather}
\begin{lemm} \label{lemm:blow.up1}
    Let $X_k(k\in \ZZ^+)$ be a sequence of smooth normal vector field on $M$ satisfying the condition of Lemma \ref{lemm:equation.graph} with $\vartheta_k := \norm{X_{k}}_{L^2}\to 0(\vartheta_k >0)$ as $k\to \infty.$ For each $\alpha \in(0,1), j \in \NN$, after passing to a subsequence, $X_{k}/\norm{X_{k}}_{L^2}$ converges to a Jacobi field $Y_0$ on $M$ in $C^{j+2,\alpha}$ sense as $k\to \infty$ with $\norm{Y_0}_{L^2} = 1\,.$

\end{lemm}
\begin{proof}
   Let $Y_k = X_k / \norm{X_{k}}_{L^2}, $ and $Q_{Y_k} =\fh{1}{\vartheta_k}Q_2(X_{k})\,. $
   Then 
   \begin{gather}
\la^{\bot} Y_k+\tilde{B}(Y_k) + nY_k =  Q_{Y_k}
   \end{gather} 
   and \begin{gather}
       \sum_{0 \leq j \leq s} \abs{
       \d^j Q_{Y_{k}}} \leq c_s \vartheta_k\left(   \sum_{0 \leq j\leq s+2}  \abs{(\df)^j Y_k}\right)^2.
   \end{gather}
   Since $X_k$ satisfy  \eqref{eq:jacob.1} and $\vartheta_k \to 0$ as $k\to \infty,$ together with \eqref{ineq.schauder.estimate} and \eqref{ineq:l2.norm}, we have \begin{gather}
       \abs{Y_k}_{j+2,\alpha} \leq c_{j,\alpha} \abs{Y_k}_0 \leq C_{j,\alpha}.
   \end{gather}
  By Arzel$\grave{a}$--Ascoli theorem, after passing to a subsequence, $Y_{k}$ converges to $Y_0$ in $C^{j+2,\alpha}$ sense as $k\to \infty,$ and  $Y_0$ is a Jacobi field on $M$ with $\norm{Y_0}_{L^2} = 1.$
\end{proof}

The orientation-preserving isometry group of $\SS^{n+q}$ is SO$(n+q+1).$ For any $S \in$ SO$(n+q+1)$ and $p\in M  \subset \SS^{n+q} \subset\RR^{n+q+1}$, $pS \in  \SS^{n+q} \subset\RR^{n+q+1} $ is a row vector since $p$ is a row vector in $\RR^{n+q+1}$. We denote $M\cdot S :=\{p S:p\in M\}$ as a submanifold in $\SS^{n+q}$. Recall that $\fs\fo(n+q+1) $ denotes the Lie algebra of SO$(n+q+1)$. 
If $A \in \fs\fo(n+q+1),$ then $$\exp (A) =\sum\limits_{k=0}^{\infty} \fh{A^k}{k!} \in \mathrm{SO}(n+q+1).$$ In particular,  $\exp( O_{n+q+1} )= I_{n+q+1} $. As a subspace of $\RR^{(n+q+1)\times (n+q+1)}$, we assume $\fs \fo(n+q+1) $ has the standard inner product of $\RR^{(n+q+1)\times (n+q+1)}$. For $A\in \fs\fo(n+q+1)$, we denote $\ang{A,A} = \tr(A^tA)$ and $ \norm{A} = \ang{A,A}^{\fh{1}{2}}.$ There exists a neighborhood $U_1\subset \fs \fo(n+q+1)$ of $O_{n+q+1}$  and a neighborhood $U_2\subset $ SO$(n+q+1)$ of $I_{n+q+1}$ such that the map $\exp$ is a diffeomorphism from $U_1$ to $U_2$.

Let $Jac(M)$ denote the space of Jacobi fields on $M$.
 We define a projection $\Pi :L^2(\Gamma(NM)) \to Jac(M)$ by
\begin{gather}
\begin{aligned}
\Pi (X) = \sum_{i=1}^{m} \ang{X,\phi_i}_{L^2}\phi_i\qquad\qquad \text{for every } X \in  L^2(\Gamma(NM)),
\end{aligned}
\end{gather}
where $\{\phi_i\}_{i=1}^{m}$ denote an orthonormal basis of  $Jac(M).$
In particular, the definition of $\Pi$ is independent of the choice of  $\{\phi_i\}_{i=1}^{m}$.
    
\begin{lemm} \label{lemm.full.Jacobi.field}
 For a closed embedded minimal submanifold $M^n$ in $\SS^{n+q},$
    there exists a small $\eps_2>0$ depending only on $n,q,M$ such that if $A \in U_0 := \{A'\in \mathfrak{so}(n+q+1): \norm{A'}<\eps_2\}$, then for every constant $t\in[-1,1]$ there exists a unique $X_t \in \Gamma(NM)$ satisfying $M \cdot \exp(tA)  = G_M(X_t)$ and $\lim\limits_{t\to  0} \fh{X_t(p)}{t}=(pA)^{N}$ for any $p\in M,$ where we also see $X_t(p) \in N_pM$ as a row vector in $\RR^{n+q+1}. $
Moreover, if all Jacobi fields on $M$ are Killing Jacobi fields,
and $\Psi_1 : U_0 \to Jac(M)$ is a map defined by   $ \Psi_1 (A) = \Pi (X_A)$  with $ G_M(X_A) = M\cdot \exp(A) $  for every $ A \in U_0 $,
then $\Psi_1$ is surjective.
\end{lemm}
\begin{proof}
Since $M^n$ is a closed embedded minimal submanifold in $\SS^{n+q},$ the second fundamental form of $M$ is bounded. 
Given a small $\eps_2>0$ depending only on $n,q,M$, we choose $A\in U_0 := \{A'\in \mathfrak{so}(n+q+1): \norm{A'}<\eps_2\}$, then there exists a unique $X_t \in \Gamma(NM)$ with $M \cdot \exp(tA)  = G_M(X_t)$ for any $-1 \leq t \leq 1$ and $\lim\limits_{t\to0}X_t(p) = 0$ uniformly for all $p \in M$. For any $p\in M,$ let $\psi(t,p) = G_M(X_t(p)) \exp(-tA),$ then $\psi(t,p) $  is a diffeomorphism of $M$ for each fixed $t$. So we have 
  \begin{gather}
     \left(\lim\limits_{t \to 0} \fh{\psi(t,p) -p}{t} \right)^{N}  = 0, 
     \end{gather}
and
     \begin{gather}
      \begin{aligned}
     &\lim\limits_{t \to 0} \fh{\psi(t,p) -p}{t} = \lim\limits_{t \to 0} \fh{G_M(X_t(p)) \exp(-tA) -p}{t}  \\
     =&   \lim\limits_{t \to 0} \fh{G_M(X_t(p))-p +p  -p\exp(tA)}{t}\exp(-tA)
     =\lim\limits_{t\to0} \fh{G_M(X_t(p)) - p}{t} -pA\\
=&  \lim\limits_{t\to0} \fh{\fh{p+X_t(p)}{\sqrt{1+\abs{X_t(p)}^2}} - p}{t}-pA 
      = \lim\limits_{t\to0} \left(\fh{ X_t(p)}{t\sqrt{1+\abs{X_t(p)}^2}} + \fh{\fh{1}{\sqrt{1+\abs{X_t(p)}^2}}-1}{t} p \right)-pA\\ 
=& \lim\limits_{t\to0} \fh{X_t(p)}{t}-pA.
     \end{aligned}
       \end{gather}
 This implies $\lim_{t\to0} \fh{X_t(p)}{t}=(pA)^N$.

  Hence, $D\Psi_1(A)(p) = (pA)^{N}.$ If all Jacobi fields on $M$ are Killing Jacobi fields, then $D\Psi_1$ is  a surjective map at $O_{n+q+1} \in \fs \fo(n+q+1)$ and $\Psi_1 $ is a surjective map on $U_0$ by taking $\eps_2$ sufficiently small (depending only on $n,q,M$).
\end{proof}

\begin{theo} \label{theo:rigidty.link}
    Given a closed embedded  minimal submanifold $M^n$ in $\SS^{n+q}$ such that all Jacobi fields of $M$ are Killing Jacobi fields. There exists a constant $\delta\in(0,\varepsilon_1]$ depending only on $n,q$ and $M$ such that if $M' = G_M(Y)$ is an embedded minimal submanifold in $\SS^{n+q}$ with $\sum_{i=0}^{2} \abs{(\d^\bot)^i Y} \leq \delta$, then $M' = M\cdot\exp(A)$  for some $A\in \mathfrak{so}(n+q+1)$, i.e., $M'$ is a rigid motion   of $M$ in $\SS^{n+q}.$
\end{theo}
\begin{proof}
    We prove the theorem by contradiction. If there is a sequence $M_k(k\in \ZZ^+)$ satisfying the condition of the theorem and $\sum_{i=0}^{2} \abs{(\d^\bot)^i Y_k} \to 0$ as $k\to \infty,$  but none of $M_k$ is a rigid motion of $M$ in $\SS^{n+q+1}.$ By Lemma \ref{lemm:equation.graph} and Lemma \ref{lemm:blow.up1}, after passing to a subsequence, we can assume $Y_k / \norm{Y_k}_{L^2} \to Y_0$ in $C^{2,\alpha}$ sense for some $\alpha\in(0,1)$ as $k\to \infty$, where $Y_0 \in Jac(M)$ is a Jacobi field on $M$ and $\norm{Y_0}_{L^2} = 1.$ Then $\Pi(Y_k)/ \norm{Y_k}_{L^2} \to Y_0$ in $L^2$ sense, and $\norm{\Pi(Y_k)}_{L^2}/\norm{Y_k}_{L^2} \to 1$  as $k\to \infty$. So we have
    \begin{gather} \label{identity.subsequence} \ \Pi(Y_k)/ \norm{\Pi(Y_k)}_{L^2} \to Y_0 \text{ as } k\to \infty.
    \end{gather} 
    By Lemma \ref{lemm.full.Jacobi.field}, for  every sufficiently large $k$ , there exists $A_k \in  \fs \fo(n+q+1)$ such that $\Psi_1(A_k)  =\Pi(Y_k)\to 0$ as $k\to \infty.$ In particular, we can assume $\norm{A_k}\to 0$ as  $k\to \infty.$ Let $Y_k^*\in \Gamma(NM)$ satisfying $ M\cdot \exp (A_k) = G_M(Y_k^*)$, then
$$\Pi (Y_k^*) = \Psi_1(A_k)=\Pi(Y_k).$$
 For every sufficiently large $k$, we set $\hat{Y}_k = Y_k -Y_k^*,$ then
    $\hat{Y}_k \neq 0$ since none of $M_k$ is a rigid motion of $M$ in $\SS^{n+q+1}.$
    Since \begin{gather}\Pi(Y_k^*)/\norm{\Pi(Y_k^*)}_{L^2}  = \Pi(Y_k)/\norm{\Pi(Y_k)}_{L^2} \to Y_0 \text{ as } k\to \infty,
    \end{gather}
    we have  $\norm{Y_k^*}_{L^2} \geq \norm{\Pi(Y_k^*)}_{L^2}>0,$ and  $\sum_{i=0}^{2} \abs{(\d^\bot)^i Y_k^*} \to 0$ as $k\to \infty$ since $\norm{A_k} \to 0$ as $k\to \infty$. So by Lemma \ref{lemm:equation.graph} and Lemma \ref{lemm:blow.up1}, after passing to a subsequence, we can assume  \begin{gather}
    Y_k^*/ \norm{Y_k^*}_{L^2} \to \tilde{Y}_0 \text{ in } L^2 \text{ sense as } k \to \infty.
    \end{gather}
    From a same  argument as in the proof of \eqref{identity.subsequence}, we have \begin{gather}
     \tilde{Y}_0 = Y_0 \text{ and } \norm{Y_k^*}_{L^2} / \norm{Y_k}_{L^2} \to 1 \text{ as } k \to \infty.
    \end{gather} 

 Now,    we define an operator $\cN= \cJ -Q_2$, where \begin{gather}
    \cJ(X) = \la^{\bot} X +\tilde{B}(X) +n X ,\end{gather}
    and $Q_2$ is defined in \eqref{eq:jacob.1}. From Lemma \ref{lemm:equation.graph}, $\cN (Y_k) = \cN(Y_k^*) = 0$  and $\cJ(\hat{Y}_k) = Q_2(Y_k) - Q_2(Y_k^*).$ 
  For $\ell=0,1,2, k \geq 1,$   there exist smooth tensor functions $\cT_{\ell,k}$ such that 
    \begin{gather}
        Q_2(Y_k) - Q_2(Y_k^*) = \cT_{0,k} \hat{Y}_k + \ang{\cT_{1,k}, \df \hat{Y}_k} + \ang{\cT_{2,k}, (\df)^2 \hat{Y}_k} ,
    \end{gather}
    where $\abs{\cT_{\ell,k}} 
\leq c_M\norm{Y_k}_{L^2}$, and
$c_M$ is a constant  depending only on $n,q,M.$  
Let us define  $\cN_0 = \cJ$ and 
    \begin{gather}
    \begin{aligned}
        \cN_k &= \cJ -\cT_{0,k}-\left( \ang{\cT_{1,k}, \df(\cdot)} +\ang{\cT_{2,k}, (\df)^2(\cdot)}\right)\qquad \text{for each }k\ge1. 
        \end{aligned}
    \end{gather}
 Clearly,  $\cN_k$ converges smoothly  to $\cN_0$ as $k\to \infty$ and $\cN_k$ is elliptic for all $k$ large enough. Let $Z_k = \hat{Y}_k / \norm{\hat{Y}_k}_{L^2},$ then $\cN_k(Z_k) = 0 .$ By the Schauder estimate and $L^2$ estimate, $|(\df)^j Z_k|$ is uniformly bounded  depending only on $n,q,j$ and $M.$ 
By Arzel$\grave{a}$--Ascoli theorem, after passing to a subsequence, $Z_{k}$ converges to $Z_0 \in Jac(M)$ in $C^{2,\alpha}$-sense for some $\alpha\in(0,1)$ as $k\to \infty$. Then $\norm{Z_0}_{L^2} =1$ and $\cN_0(Z_0) = 0 .$  However, $0 = \Pi(Z_{k}) \to\Pi(Z_0) = Z_0$ in $L^2$ sense as $k\to \infty$,  which is a contradiction.
    \end{proof}
As a corollary, we immediately obtain Theorem \ref{link-LO-rigidty}. 
In fact, the conditions in Theorem \ref{theo:rigidty.link} and Theorem \ref{link-LO-rigidty} can be weakened with the aid of Allard's regularity theorem.

 Let us first introduce some language of varifolds by Almgren, see \cite{GMT} as a standard reference.  An $n$-varifold $V$ in $\SS^{n+q+1}\subset\RR^{n+q+1}$ is said to be \emph{stationary} in $\SS^{n+q+1}$ if
 \begin{gather}
     \int_{\SS^{n+q+1}} \mathrm{div}_VY d\mu_V=0\qquad\qquad \text{for each }Y\in\Gamma(T\SS^{n+q+1}).
 \end{gather}
Here, div$_VY$ is the divergence of $Y$ restricted on spt$V$, and $\mu_V$ denotes the Radon measure corresponding to $V$.
We define the density of $V$ at the point $p\in \RR^{n+q+1}$ by 
 \begin{gather}
     \Theta_V(p) = \lim\limits_{r \to 0^+} \Theta_V(p,r) \ , \text{ where } \Theta_V(p,r) = \fh{\mu_V(B_{r }(p))}{\omega_n r^n} \text{ for } r>0\,.
 \end{gather}
Given a countably $n$-rectifiable set $S \subset \RR^{n+q+1}$, we use $|S|$ 
to denote the multiplicity one $n$-varifold associated with $S$.  
\begin{lemm}  \label{lemm:graph}
     For a constant $\eps_3\in(0,1)$, and a closed embedded minimal submanifold $M^n$ in $\SS^{n+q}$, there exists a small constant $\rho >0$ depending only on $n,q,\eps_3$ and $M$ such that if $V$ is a stationary integral $n$-varifold  in $\SS^{n+q}$ with  $\mu_V (\SS^{n+q}) \leq (2-\eps_3)\cH^n(M)$ and the Hausdorff distance $d_{\cH}(M,\spt V) \leq \rho,$ then $V = |\Sigma|$, and $\Sigma$ is an embedded minimal submanifold in $\SS^{n+q}$ so that $\Sigma$ is a graph on $M,$ i.e., $\Sigma = G_M(X)$ for some $X\in \Gamma(NM)$ with small $\sum_{i=0}^{2} \abs{(\d^\bot)^i X} $ (depending only on $\rho$). 
\end{lemm}
\begin{proof} 
We prove the lemma by contradiction. Let $\{V_k\}_{k\geq1}$ be a sequence of stationary integral $n$-varifolds in $\SS^{n+q},$ such that $d_{\cH}(\spt V_k, M)  \to 0 $ as $k\to \infty$ and $\mu_{V_k}(\SS^{n+q})\leq (2-\eps_3)\cH^n(M)$, but none of $V_k$ satisfies the conclusion in the lemma.

 \textbf{Claim.} $V_k \weakto |M|$ in the sense of varifolds as $k\to \infty$.

 After passing to a subsequence, there is a stationary integral $n$-varifold $V^*$ in $\SS^{n+q}$ so that $V_k \weakto V^*$ in the sense of varifolds by the compactness theorem of stationary integral $n$-varifolds (see  \cite{Allard1}). So $d_{\cH}(\spt V^*, M) = 0$, which infers $V^* = \theta_0|M|$ for some $\theta_0\in \NN$ due to the constancy theorem(see \cite{GMT}). Since $\mu_{V_k}(\SS^{n+q}) < (2- \eps_3)\cH^n(M)  ,$ then $\theta_0=1$ and $V^* = |M|.$ This gives the proof of the \textbf{Claim}.

 Let $CM$ denote the cone with link $M$, and $CV_k$ denote the cone with link spt$V_k$ and multiplicity function $\theta_k$, where $\theta_k$ is 1-homogeneous and equal to the multiplicity function of $V_k$ on $\SS^{n+q}$.
The above \textbf{Claim} implies that $CV_k \weakto |CM|$ locally in the sense of varifolds as $k\to \infty$.
Noting that all $CV_k $ are stationary in $\RR^{n+q+1}$.
From Allard's regularity theorem in  \cite{Allard1}, we have $V_k = |\Sigma_k|$ for all $k$ large enough, where $\Sigma_k$ is an $n$-dimensional closed embedded minimal submanifold in $\SS^{n+q}$ and $\Sigma_k$ is the  graph on $M$ with $\sum_{i=0}^{2} \abs{(\d^\bot)^i X} \to 0$ as $k\to \infty$,  which is a contradiction. 
\end{proof}

Combining Theorem \ref{theo:rigidty.link} and Lemma \ref{lemm:graph}, we obtain the following rigidity result.
\begin{theo}\label{theo:rigidty.link*}
Given a closed embedded minimal submanifold $M^n \subset \SS^{n+q}$ such that all Jacobi fields of $M$ are Killing Jacobi fields. For a constant $\eps_3\in(0,1)$, there exists a constant $\sigma>0$ depending only on $n,q$ and $M$ such that if  $V$ is a stationary integral $n$-varifold  in $\SS^{n+q}$ with $\mu_{V} (\SS^{n+q})\leq (2-\eps_3)\cH^n(M)$ and $ d_{\cH}(M,\spt V ) \leq \sigma,$ then $V = |\Sigma|$ where $\Sigma$ is a rigid motion
   of $M$ in $\SS^{n+q}.$
\end{theo}

\section{Asymptotic to Lawson--Osserman Cone}
     
We shall calculate various geometric quantities for  graphs over $(n+1)$-dimensional regular minimal cone $\BC \subset  \RR^{n+q+1}.$  We denote $\BC_r := \BC \setminus \bar{ B_r(0)}.$ Let $M$ be the link of the minimal cone $\BC$ which is a closed embedded $n$-dimensional minimal submanifold in $ \SS^{n+q}.$ We can parametrize $\BC_1$ by $M \times \RR^+$, i.e., 
 \begin{gather}
     \tilde F_0: M \times \RR^+ \to \BC_1,\, 
     \tilde F_0(p,t) = e^tp \,.
 \end{gather}
 Then for each $x\in \BC_1,$ we can write $x = e^tp$, where $e^t = \abs{x}$ and $p = \fh{x}{\abs{x}}.$
 We denote $N\BC_1$ the normal bundle of $\BC_1$ in $\RR^{n+q+1},$ and denote $NM$ the normal bundle of $M$ in $\SS^{n+q}.$ Since $NM$ is a vector bundle of rank $q$ with the base space $M$ and $\RR^+$ is a vector bundle of rank 0 with the base space $\RR^+$ , we denote $NM\times \RR^+$ as the product vector bundle of rank $q$ with the base space $M\times \RR^+.$
    For each $Z \in \Gamma(N\BC_1),$ $Z(x)  =Z(e^tp) = e^t X(p,t) $ as vectors in $\RR^{n+q+1}$ with unique $X(p,t) \in N_{p}M\times \{t\}$ for each $e^tp = x \in \BC_1.$ So we may identify $e^tX$ as a section of $N\BC_1$ for every $X\in \Gamma(NM\times \RR^+).$ 
     We denote the map \begin{gather}
 G_{\BC_{R}} : N\BC_R \to   \RR^{n+q+1}, G_{\BC_{R}}(e^tX(p,t)) = e^t\fh{p+X(p,t)}{\sqrt{1+\abs{X(p,t)}^2}}
 \end{gather}
 for every $e^tX\in \Gamma(N\BC_R).$ If  $\sum_{ \substack{0\leq i,j\\ i+j\leq 2}} \abs{(\d^\bot)^i \fh{\de^j}{\de t^j}X}$ is sufficiently small, then $G_{\BC_R}(e^tX)$ is an embedded $(n+1)$-dimensional submanifold in $\RR^{n+q+1}.$
  As below, we may take $R=1.$ Let $e^tX\in \Gamma(N\BC_1)$ with $\sum_{ \substack{0\leq i,j\\ i+j\leq 2}} \abs{(\d^\bot)^i \fh{\de^j}{\de t^j}X}$ sufficiently small, then we define   a map $\tilde F_X: M \times \RR^+  \to \RR^{n+q+1} $  by
\begin{equation} \label{eq:cone.exp}
    \tilde F_X(p,t) = G_{\BC_{R}}(e^tX(p,t)) = e^t \fh{p+X(p,t)}{\sqrt{1+\abs{X(p,t)}^2}}\,.
\end{equation}

Let the symbol '*' denote the linear combination,
 and let $X_{ij \alpha \beta}^{k}(i,j, \alpha, \beta\geq 0, 0\leq i +j \leq k,0\leq \alpha+\beta \leq k)$ be smooth tensor functions depending on $(\df)^{s} \fh{\de^\ell}{\de t^\ell}X(s,\ell\geq 0, 0\leq s+\ell \leq k)$, $n ,q$ and $M,$ with all derivatives of  $X_{ij \alpha \beta}^{k}$ uniformly bounded. Then we define  general vector functions
  \begin{gather}
      \tilde Q_k(X)=  \sum_{\substack{0\leq i,j,\alpha,\beta\\i+j\leq k \\\alpha+\beta \leq k
      }}\left((\df )^i \fh{\de^j}{\de t^j}X* (\df )^\alpha\fh{\de^\beta}{\de t^\beta}X*X_{ij\alpha\beta}^{k} 
      \right) \qquad\text{for each } k\geq 0,
  \end{gather}
satisfying that 
 \begin{align}  \label{cone.high.est}
 \sum_{0\leq i \leq \ell}\abs{\bar \d^i\tilde Q_k(X) } \leq  c_\ell \left( \sum_{\substack{0\leq i,j\\
 i+j\leq \ell+k}} \abs{(\df)^{i} \fh{\de^j}{\de t^j}X} \right)^2 \end{align}
for some constant $c_\ell >0$ depending only  on $n,q,k,\ell,$ and $M$, where $\bar \d$ denotes the full gradient of $\tilde{Q}_k(X)$ on the vector bundle $NM\times \RR^+$.

\begin{lemm} \label{lemm:cone.equation.graph}
Given an (n+1)-dimensional regular cone  in $ \RR^{n+q+1}$ with link $M$ and a fixed $\ell \in \NN$, there exists a sufficiently small $\kappa_1>0$  depending only on $\ell, n,q$ and $M$, such that
if $  X  \in \Gamma(NM\times \RR^+)$ with
\begin{equation} \label{eq:cone.small.graph.funtion}
     \sum_{ \substack{0\leq i,j\\ i+j\leq \ell+2}} \abs{(\d^\bot)^i \fh{\de^j}{\de t^j}X}\leq \kappa_1 \,,
\end{equation}
 and  $\Sigma = \tilde{F}_X(M\times \RR^+)$ is minimal in $\RR^{n+q+1}\setminus \bar{B_1(0)}$, then
\begin{align} \label{eq:cone.jacob.1}
X_{tt} + (n+1)X_{t} + \la^{\bot}X +\tilde{B}(X)+nX =\tilde Q_2(X) \,, 
\end{align}
where $\tilde Q_2(X)$ satisfies \eqref{cone.high.est} with $k = 2$. 
\end{lemm}

Lemma \ref{lemm:cone.equation.graph} is classical, and it follows from \cite[$\S 4$, 4.14]{isolated.singularity}.
 
As below, we denote $\d$ as the Levi--Civita connection of the Riemannian manifold $(M^n\times \RR^+,g_M+dt^2).$ 
Let $Y\in C^2(M\times \RR^+;\RR^m)$ and we denote $\tilde \cQ(Y)$ as the quadratic nonlinear term, i.e.,
\begin{gather}\tilde \cQ(Y) = \sum_{ i,j=0}^2\d^i Y*\d^j Y*Z_{i j},\end{gather}
where $Z_{ij}$ are smooth tensor functions depending only on $Y,\d Y,\d^2Y,$    and $n,m,$ $M$. Moreover,  all the derivatives of  $\tilde \cQ(Y)$  are uniformly bounded, i.e.,
\begin{align}  
 \sum_{0\leq i \leq \ell}\abs{ \d^{i}\tilde \cQ(Y) } \leq  c_\ell \left( \sum_{i=0}^{\ell+2} \abs{\d^{i}Y} \right)^2 \end{align}
 for some constant $c_\ell >0$ depending only on $\ell,$  $n,m,$ and $M$.

  As below, we assume the cone $\BC$ is the Lawson--Osserman cone and $M$ is the link of $\BC$ in $\SS^6$. From \eqref{eq.isom.link}, $ (M,g_M)$ is isometric to  $(\cS^3, \fh{8}{3} g_{\fh{1}{6}}).$ 
 Let $\{e_4,e_5,e_6\}$ be the orthonormal frame of the normal bundle of $M$ in $\SS^6$ defined in \eqref{subsec.frame}. We define $X=\sum_{i=1}^3\xi_ie_{i+3} \in\Gamma(NM\times \RR^+),$
where $\xi_i \in C^2(M \times \RR^+)$.
From \eqref{jcb1} and \eqref{oper.L}, $X$ satisfies the system 
\begin{gather}\label{rigidty.cone.jacobi1}
    X_{tt}+ 4X_t+\la^{\bot}X+\tilde B(X)+3X = \tilde Q_2(X)\,.
\end{gather}
Then we define a vector field $\tilde{X}$ associated with $X$ by $\tilde{X}:= (\xi_1,\xi_2,\xi_3)\in C^2(\SS^3\times\RR^+;\RR^3)$. Then $\tilde{X}$  satisfies the system
\begin{gather}\label{cone.frame.eq.sys}
   \tilde{\cL}\tilde{X}:= \tilde{X}_{tt} +4\tilde{X}_t-\fh{3}{8}L(\tilde{X}) = \tilde \cQ(\tilde{X})
\end{gather}
with  $\tilde \cQ( \tilde{X}) = \tilde Q_2(X)$, where the operator $L$ is defined in \eqref{oper.L}. We denote $\fh{8}{3}\lambda_k(k\in \NN)$ as the $k$-th eigenvalue (counted without multiplicity) of the operator $L$, i.e., \begin{gather}
     LY = \fh{8}{3}\lambda_k Y\, \quad \text{ for some }  Y \in C^\infty (\SS^3;\CC^3) \text{ with } Y\neq 0,
\end{gather}
and we denote $\tilde V_k$ as the corresponding eigenspace with  $\dim \tilde V_k<\infty$. So $\lambda_k$ is the $k$-th eigenvalue of the Jacobi operator of the link of Lawson--Osserman cone $\BC.$
Since $\tilde L$ (defined in \eqref{oper.L1}) can commute with $L,$  the subspace $\tilde V_k$ is also an invariant subspace of $\tilde L.$ Then we can decompose the space $\tilde V_k = \oplus_{i\geq 0}(\tilde V_k\cap V_{i})$ with $V_i(i \in \NN)$ defined in \eqref{space.Vk}, and the sum is finite since $\dim \tilde V_k<\infty$.
We denote $\{\Phi_{k,i,\alpha}\}_{\alpha}$ the 
orthonormal basis of the space $\tilde V_k\cap V_{i}$ for each fixed $k,i \in\NN$, i.e.,  $$\dashint_{\SS^3}
\ang{\Phi_{k_1,i_1,\alpha_1}\,,\Phi_{k_2,i_2,\alpha_2}} d\mu_{g}  = \delta_{k_1k_2}\delta_{i_1i_2}\delta_{\alpha_1\alpha_2}\,.$$ 

     We shall investigate the decay order of the solution $X$ to \eqref{rigidty.cone.jacobi1}. 
Let  $\lambda_{k,\pm} = -2\pm \sqrt{4+\lambda_k},$ and we have 
    \begin{gather}
    \lambda_{2,+} = 0,\lambda_{1,+}=-1\,,\lambda_{0,+} = -\fh{3}{2}\,,\lambda_{0,-} = -\fh{5}{2}\,,\lambda_{1,-} = -3\,,\lambda_{2,-} = -4 .\end{gather}
    Let 
\begin{gather}
\xi_{k,i,\alpha}(t) = \dashint_{\SS^3} \ang{\tilde{X}(x,t), \Phi_{k,i,\alpha}(x)} d\mu_{g}
  \text{ and }
f_{k,i,\alpha}(t)=\dashint_{\SS^3} \ang{\tilde \cQ(\tilde{X})(x,t),
\Phi_{k,i,\alpha}(x)} d\mu_{g}\,.
\end{gather}
Hence $\xi_{k,i,\alpha}$ and $f_{k,i,\alpha}$ are smooth functions on $\RR^+\,,$
and from \eqref{cone.frame.eq.sys}, they satisfy
\begin{gather}
\xi_{k,i,\alpha}''+4\xi_{k,i,\alpha}'+\lambda_k\xi_{k,i,\alpha}= f_{k,i,\alpha}\,.
\end{gather}
From the standard ordinary differential equation theory, we can write
\begin{gather}\label{ode.eq1}
\xi_{k,i,\alpha}(t)=a_{k,i,\alpha}e^{t\lambda_{k,+}} +
b_{k,i,\alpha}e^{t\lambda_{k,-}}+ r_{k,i,\alpha}(t)\,,
\end{gather}
with $a_{k,i,\alpha}, b_{k,i,\alpha}\in \CC\,.$ 
So we have  \begin{gather}
\begin{aligned}
\tilde{X}(x,t) =& \sum_{k,i,\alpha}\xi_{k,i,\alpha}(t)\Phi_{k,i,\alpha}(x)\\
=&\sum_{k,i,\alpha}(a_{k,i,\alpha}e^{t\lambda_{k,+}} +
b_{k,i,\alpha}e^{t\lambda_{k,-}}+ r_{k,i,\alpha}(t))\Phi_{k,i,\alpha}(x)\,.
\end{aligned}
\end{gather}

\begin{prop}\label{prop.improv}
Fix a constant $\beta>0$. Let $\tilde{X}$ be a solution of \eqref{cone.frame.eq.sys} on $\SS^3 \times \RR^+ $ satisfying
that  $\|\tilde{X}\|_\beta<\infty$ and  $\lambda_{k,\pm} \neq -2\beta $ for all $k \geq 0\,$, then there exists a sufficiently small constant $\kappa_2\in (0,\kappa_1]$ such that if $\norm{\tilde{X}}_{2,0}<\kappa_2\,,$ then $\tilde{X}$ can be written as
$\tilde{X}=X_1+ R_1\,,$ where $\|X_1\|_\beta <\infty$ satisfies that  $\tilde \cL X_1 =0$ and
$\|R_1\|_{2\beta}<\infty\,$.
\end{prop}
\begin{proof}

Since  $\norm{\tilde{X}}_{2,0}<\kappa_2$ with $\kappa_2$ sufficiently small, the system in  \eqref{cone.frame.eq.sys} satisfies the Legendre--Hadamard condition \cite{Giaquita,Morrey}.  So the classical elliptic estimates give upper bounds on the
derivatives of $\tilde{X}$: more precisely, for any $\ell\in \NN\,$ and  any $s>1\,,$ there is a constant $C_\ell'$ independent of $s$
such that 
$$
\|\nabla^\ell \tilde{X} \|_{C^0(\SS^3\times [s,s+1])}\leq C_\ell'\|\tilde{X}\|_{C^0(\SS^3 \times [s-1,s+2])}\,.
$$
This implies that  $\|\d^\ell \tilde{X} \|_\beta<\infty\,$ for any $\ell\in \NN\,$. Since the term
$\tilde \cQ(\tilde{X})$ in \eqref{eq:cone.jacob.1} gathers all the nonlinear terms consisting of $\tilde{X}\,, \d \tilde{X}\, , \d^2\tilde{X}$  at least quadratic, we have
 $\|\d^\ell \tilde \cQ(\tilde{X})\|_{2\beta}<\infty$ for any $\ell \in \NN$.

Recall that $\dim V_i = 3(i+1)^2$ by Lemma \ref{spec}. Moreover, for all $i\geq 1\,,$ we have the following estimates of the
$L^\infty$ norm of the eigenfunctions (see \cite{Sogge}):
$$
\|\Phi_{k,i,\alpha}\|_\infty\le c_2\gamma_i^{\fh{1}{2}} \quad \text{ with }\gamma_i = i(i+2)\,\text{ and some fixed } c_2>1 \text{ independent of } i\,.$$
 Recall that  $\tilde L \Phi_{k,i,\alpha}=\gamma_i\Phi_{k,i,\alpha}\,$, where $\{\Phi_{k,i,\alpha}\}_{\alpha}$ is the 
orthonormal basis of the space $\tilde V_k\cap V_{i}$. Then, using integration by parts, for all $i\geq 1 $ and any $a\in\ZZ^+\,,$ we derive the following estimates:
\begin{align*}
|\xi_{k,i,\alpha}(s)| \leq \frac{\sup_{x\in \SS^3}|\nabla^{2a}\tilde{X}(x,s)|}
{(1+\gamma_i)^a}\, ,\quad \text{ and } \quad 
|f_{k,i,\alpha}(s)| \leq \frac{\sup_{x\in \SS^3}|\nabla^{2a}\tilde\cQ(\tilde{X})(x,s)|}
{(1+\gamma_i)^a}\,.
\end{align*}
Thus we get 
\begin{align*}
\|\xi_{k,i,\alpha}\|_\beta \leq \frac{\|\nabla^{2a}\tilde{X}\|_\beta}
{(1+\gamma_i)^a} \,,\quad \text{ and } \quad
\|f_{k,i,\alpha}\|_{2\beta} \leq \frac{\|\d^{2a}\tilde \cQ(\tilde{X})\|_{2\beta}}
{(1+\gamma_i)^a}\,.
\end{align*}
For $i=0\,,$ 
$$ \|f_{k,0,\alpha}\|_{2\beta}\leq \| \tilde \cQ(\tilde{X}) \|_{2\beta}\,.$$

Moreover, we have
$$
\xi_{k,i,\alpha}(t)=a_{k,i,\alpha}e^{t\lambda_{k,+}} +
b_{k,i,\alpha}e^{t\lambda_{k,-}}+ r_{k,i,\alpha}(t)
$$
with some estimates on the different terms(see Lemma 10 in Appendix A of \cite{Mazet1}). First we notice that
$|\lambda_{k,+}-\lambda_{k,-}| = $$|2\sqrt{4+\lambda_k}|$ and $|2\beta+\lambda_{k,\pm}|$ are uniformly
bounded from below by a positive constant and
$\frac{(2+|\lambda_{k,+}|^2+|\lambda_{k,-}|^2)^{1/2}}
{|\lambda_{k,+}-\lambda_{k,-}|}$ is uniformly bounded. Then, for all $i\geq 1\,,$ there
is a uniform constant $c_3$ independent of $k$ such that
\begin{align*}
\max(|a_{k,i,\alpha}|,|b_{k,i,\alpha}|)&\leq c_3 (
\|\xi_{k,i,\alpha}\|_\beta+ \|\xi'_{k,i,\alpha}\|_\beta+
\|f_{k,i,\alpha}\|_{2\beta})\\
&\leq c_3 \frac{\|\d^{2a}\tilde{X}\|_\beta+ \|\d^{2a+1}\tilde{X}\|_\beta+
\|\d^{2a}\tilde \cQ(\tilde{X})\|_{2\beta}}{(1+\gamma_i)^a}
\end{align*}
and  
\begin{align*}
\|r_{k,i,\alpha}\|_{2\beta} \leq c_3 \|f_{k,i,\alpha}\|_{2\beta} 
 \leq c_3 \|\tilde \cQ(\tilde{X})\|_{2\beta}\,.
\end{align*}
For $i=0\,,$ $$\max(|a_{k,0,\alpha}|,|b_{k,0,\alpha}|) \leq C'\, , \|r_{k,0,\alpha}\|_{2\beta} \leq C'\,.$$ 
If $\lambda_{k,+} > -\beta$, we have $\norm{e^{t\lambda_{k,+}}}_{\beta} = \infty$, so $a_{k,i,\alpha} = 0\,.$ If $\lambda_{k,-} > - \beta\,$, we  also have  $b_{k,i,\alpha} = 0\,.$ If $\lambda_{k,\pm}\leq -2\beta\,$, we have
$ \norm{e^{t\lambda_{k,\pm}}}_{2\beta} = 1$.
Finally we obtain the following equality:
\begin{gather} \label{eq.pro.sum}
\begin{aligned}
\tilde{X}(x,t)&=\sum_{-2\beta\leq \lambda_{k,+}\leq -\beta}\sum_{i,\alpha} a_{k,i,\alpha}
e^{t\lambda_{k,+}} \Phi_{k,i,\alpha}(x)
+\sum_{-2\beta\leq \lambda_{k,-} \leq -\beta}\sum_{i,\alpha} b_{k,i,\alpha} e^{t\lambda_{k,-}}
\Phi_{k,i,\alpha}(x)\\
&\quad+\sum_{\lambda_{k,+} < -2\beta}\sum_{i,\alpha} a_{k,i,\alpha}
e^{t\lambda_{k,+}} \Phi_{k,i,\alpha}(x)
+\sum_{\lambda_{k,-}< -2\beta}\sum_{i,\alpha}b_{k,i,\alpha} e^{t\lambda_{k,-}}
\Phi_{k,i,\alpha}(x)\\
&\quad+\sum_{k=0}^\infty \sum_{i,\alpha}r_{k,i,\alpha}(t) \Phi_{k,i,\alpha}(x)\, .
\end{aligned}
\end{gather}
We observe that the first two sums of \eqref{eq.pro.sum} are finite and are elements of the kernel of
$L$, so this is the expected function $X_1$. The Proposition follows from the \textbf{Claim} below.

 \textbf{Claim.} The remaining sums are convergent and have finite $2\beta $-norms. 

Let $I = \sum_{\lambda_{k,-}< -2\beta}\sum_{i,\alpha}b_{k,i,\alpha} e^{t\lambda_{k,-}}
\Phi_{k,i,\alpha}(x)$. In
the following computation, we use the fact that $\gamma_i = i(i+2)$
with $\dim V_i = 3(i+1)^2$ together with the $L^\infty$ estimates on $\Phi_{k,i,\alpha}\,.$ Let  $a \geq \fh{5}{2},$ then 
\begin{align*}
\|I\|_{2\beta }&\leq c_2\sum_{\lambda_{k,-}< -2\beta}\sum_{i\geq 1,\alpha} |b_{k,i,\alpha}|
 \gamma_i^{\fh{1}{2}} + \sum_{\lambda_{k,-}<-2 \beta}\sum_{\alpha}|b_{k,0,\alpha}|\\
&\leq c_2c_3\sum_{i\geq 1,\alpha} \frac{\|\d^{2a}\tilde{X}\|_\beta+
\|\d^{2a+1}\tilde{X}\|_\beta+
\|\d^{2a}\tilde \cQ(\tilde{X})\|_{2\beta}}{(1+\gamma_i)^a}
\gamma_i^{\fh{1}{2}} + 3C'\\
&\leq 3c_2c_3m(\|\nabla^{2a}\tilde{X}\|_\beta+ \|\d^{2a+1}\tilde{X}\|_\beta+
\|\d^{2a}\tilde \cQ(\tilde{X})\|_{2\beta})\sum_{i=1}^\infty (i+1)^{3-2a}+3C'\\
&<\infty\,.
\end{align*}
We can prove the other two sums in the \textbf{Claim} by the same method and we omit the details. This gives the
proof of the \textbf{Claim}.
\end{proof}

 Now, we consider the asymptotic behavior of minimal submanifolds converging to the Lawson--Osserman cone $\BC$ at infinity, 
and obtain the optimal decay rates for minimal submanifolds (with boundary) converging to $\BC$ as below. 
\begin{theo}
    If $V$ is an integral stationary 4-varifold in $\RR^7\setminus \bar{B_1(0)},$ and the Lawson--Osserman cone $\BC$ is the unique tangent cone of $V$ at infinity with multiplicity one, then there exists  $R>1$ so that $\spt V\setminus \bar{B_R(0)} = E+G_{\BC_R}(Z) $ with a constant vector $E\in \RR^7$, $Z \in \Gamma(N\BC_R)$ and $\abs{Z(x)} = O(|x|^{-\fh{1}{2}}).$  Moreover, if $\abs{Z(x)} = O(|x|^{-(\fh{1}{2}+\delta_0)})$ for some sufficiently small $\delta_0>0$, then $\abs{Z(x)} = O(|x|^{-\fh{3}{2}}).$
\end{theo}
\begin{proof}
    Due to Theorem \ref{theo.dim}, the Lawson--Osserman cone is integrable. Then from Allard--Almgren \cite{allard.radial} and Simon \cite[p273, Theorem 6.6]{isolated.singularity},
 there exists $R>1$ so that the integral stationary 4-varifold $V$ can be
described as $G_{\BC_R}(e^tX)$ over $\BC_R$ satisfying \eqref{eq:cone.jacob.1} with $X= O(e^{-\beta t})$ for some $\beta>0$ and $\sum_{ \substack{0\leq i,j\\ i+j\leq \ell+2}} \abs{(\d^\bot)^i \fh{\de^j}{\de t^j}X}\leq \kappa_3$ sufficiently small (depending on $R$).  Therefore, $\tilde{X}\in C^2(\SS^3\times \RR^+;\RR^3)$ associated with $X$ satisfies 
 $\|\tilde{X}\|_{\beta}<\infty$.  Decreasing slightly $\beta$
if necessary, we can assume that $-2\beta \neq \lambda_{k,\pm}$ and apply
Proposition \ref{prop.improv}  to get 
$$\tilde{X} = X_1+R_1\,,$$ where $X_1$ is  in the kernel of $L$ with
decay order between $-\beta$ and $-2\beta$ and $\|R_1\|_{2\beta}<\infty\,$. If there are no
elements in the kernel of $L$ with decay order between $-\beta$ and $-2\beta\,$, we get
$\|\tilde{X}\|_{2\beta}<\infty\,$; in that case we have improved the decay order of $\tilde{X}\,$.
So we can iterate this argument until we get the first non-vanishing element in
the kernel. The first decay order of elements in the kernel is given by 
$\lambda_{1,+}=-1\,$. By Theorem \ref{theo.dim}, the second eigenvalue of the Jacobi operator $\cJ_M$ is with multiplicity $7$ and the corresponding eigenspace is $\cV(M).$ Then the  term $X_1$ can be interpreted as a translation. Thus, the translated submanifold
$(\spt V\setminus \bar{B_{R}(0)})-E$ can be expressed as $G_{\BC_R} (e^t X^*),$ with the corresponding $ \tilde{X^*}$ associated to $X^*$ satisfying $\norm{\tilde{X^*}}_{1+\delta}<\infty$ for some $\delta>0$. Thus, $\spt V\setminus \bar{B_R(0)}= E+G_{\BC_R}(Z) $, where $Z(x)= e^tX^*(p,t)$ for each $e^tp = x\in \BC_R.$

Then, we study the asymptotic behavior of $(\spt V\setminus \bar{B_{R}(0)})-E\,.$
By Proposition \ref{prop.improv},
we get the first non-vanishing element in
the kernel of $L$. Since the first decay order of elements in the kernel is given by
$\lambda_{0,+}=-\fh{3}{2}\,,$ we have $\norm{ \tilde{X^*}}_{\fh{3}{2}}<\infty $ and $\tilde{X^*} = O(e^{-\fh{3}{2}t}).$  From $|x|=e^t$ and $Z(x) =  e^t X^*(p,t) $ for each  $e^tp = x\in \BC_R$, we conclude that $|Z(x)| =  O(|x|^{-\fh{1}{2}}).$

If  $\abs{Z(x)} = O(|x|^{-(\fh{1}{2}+\delta_0)})$ for some $\delta_0>0$ small, then by the same argument as above, $\abs{Z(x)} = O(|x|^{-\fh{3}{2}})$ due to $\lambda_{0,-} = -\fh{5}{2}.$
\end{proof}

\clearpage

\appendix
\section{Connections on the tangent bundle and the normal bundle}\label{Appedix.A}

    We give  details of  calculations of connections. All frames in \eqref{subsec.frame} can be naturally extended to vector fields on $\RR^7.$
    Actually, we have
{\fontsize{10}{13}
\begin{gather} \label{eq.frame.polynomails}
    \begin{aligned}
        e_1=& \fh{3}{2}(-y_1,x_1,-y_2,x_2,0,0,0) , \\
        e_2 =& \sqrt{\fh{3}{8}}(-x_2,y_2,x_1,-y_1,-3\sqrt{5}(x_1x_2-y_1y_2),\fh{3\sqrt{5}}{2}(x_1^2-y_1^2-x_2^2+y_2^2),3\sqrt{5}(x_1y_1+x_2y_2))\,,\\
        e_3 = &\sqrt{\fh{3}{8}}(-y_2,-x_2,y_1,x_1,-3\sqrt{5}(x_1y_2+x_2y_1),3\sqrt{5}(x_1y_1-x_2y_2),\fh{3\sqrt{5}}{2}(-x_1^2-x_2^2+y_1^2+y_2^2))\,,\\
        e_4=&(\fh{\sqrt{5}}{2}x_1,\fh{\sqrt{5}}{2}y_1,\fh{\sqrt{5}}{2}x_2,\fh{\sqrt{5}}{2}y_2,-\fh{3}{2}(x_1^2+y_1^2 -x_2^2-y_2^2 ),-3(x_1x_2+y_1y_2),-3(x_2y_1- x_1y_2))\,,\\
        e_5 = & \sqrt{\fh{15}{8}} (-x_2,y_2,x_1,-y_1,\fh{3\sqrt{5}}{5}(x_1x_2-y_1y_2),-\fh{3\sqrt{5}}{10}(x_1^2-y_1^2-x_2^2+y_2^2),-\fh{3\sqrt{5}}{5}(x_1y_1+x_2y_2))\,,\\
        e_6 = & \sqrt{\fh{15}{8}} (-y_2,-x_2,y_1,x_1,\fh{3\sqrt{5}}{5}(x_1y_2+x_2y_1),-\fh{3\sqrt{5}}{5}(x_1y_1-x_2y_2),-\fh{3\sqrt{5}}{10}(-x_1^2-x_2^2+y_1^2+y_2^2))\,,\\
        \nu = &(x_1,y_1,x_2,y_2,\fh{3\sqrt{5}}{4}(x_1^2+y_1^2-x_2^2-y_2^2),\fh{3\sqrt{5}}{2}(x_1x_2+y_1y_2),\fh{3\sqrt{5}}{2}(x_2y_1-x_1y_2))\,.
    \end{aligned}
\end{gather}
}

Then we calculate the connection formulas in \eqref{eq.connection}. Due to \eqref{eq.isometry.action}, SU$(2)$ acts on $\RR^7 $ as a subgroup of SO$(7),$ and the link $M$ of Lawson--Osserman cone is invariant as subset in $\RR^7$ under the action of SU(2) by the group homomorphism $\Psi.$  So we only need to calculate the coefficient of connection formulas in \eqref{eq.connection} at the point $p_0= (\fh{2}{3}, 0,0,0,\fh{\sqrt{5}}{3},0,0)\in M.$ Therefore we can calculate the derivative of above vector fields in $\RR^7$ at the point $p_0$ 
to get \eqref{eq.connection}. At the point $p_0\,,$ $x_1 = \fh{2}{3},y_1= x_2 = y_2 =0.$ We denote the derivative in $\RR^7$ as $D.$ Then we have $$
D_{e_i}e_j|_{p_0} = \sum_{\ell=1}^6\ang{D_{e_i}e_j|_{p_0},e_\ell|_{p_0}}e_\ell|_{p_0}-\delta_{ij}\nu|_{p_0}\,, D_{e_i}e_k|_{p_0}= \sum_{\ell=1}^6\ang{D_{e_i}e_k|_{p_0},e_\ell|_{p_0}}e_\ell|_{p_0}
$$
for $1\leq i,j\leq 3\,,4\leq k\leq 6$ with $\ang{D_{e_i}e_j|_{p_0},e_k|_{p_0}} = -\ang{D_{e_i}e_k|_{p_0},e_j|_{p_0}},$ i.e., 
\begin{gather}
\begin{aligned}
D_{e_1}e_1|_{p_0} =& \pa{e_1}{y_1}\bigg|_{p_0}
=(-\fh{3}{2},0,0,0,0,0,0) &=& -\fh{\sqrt{5}}{2}e_4 -\nu \,,\\
D_{e_1}e_2|_{p_0} =& \pa{e_2}{y_1}\bigg|_{p_0} = (0,0,0,-\sqrt{\fh{3}{8}},0,0,\sqrt{\fh{15}{2}}) &=& -\fh{11}{4}e_3 + \fh{\sqrt{5}}{4}e_6\,,\\
D_{e_1}e_3|_{p_0} =& \pa{e_3}{y_1}\bigg|_{p_0} = (0,0,\sqrt{\fh{3}{8}},0,0,\sqrt{\fh{15}{2}},0) &=& \fh{11}{4}e_2 - \fh{\sqrt{5}}{4}e_5\,,\\
D_{e_1}e_4|_{p_0} =& \pa{e_4}{y_1}\bigg|_{p_0} = (0,\fh{\sqrt{5}}{2},0,0,0,0,0) &=& \fh{\sqrt{5}}{2}e_1 \,,\\
D_{e_1}e_5|_{p_0} =& \pa{e_5}{y_1} \bigg|_{p_0} =(0,0,0,-\sqrt{\fh{15}{8}},0,0,-\sqrt{\fh{3}{2}}) &=& \fh{\sqrt{5}}{4}e_3-\fh{7}{4}e_6\,,\\
D_{e_1}e_6|_{p_0} =& \pa{e_6}{y_1} \bigg|_{p_0} =(0,0,\sqrt{\fh{15}{8}},0,0,-\sqrt{\fh{3}{2}},0) &=& -\fh{\sqrt{5}}{4}e_2+\fh{7}{4}e_5\,;\\
\end{aligned}
\end{gather}
\begin{gather}
\begin{aligned}
D_{e_2}e_1|_{p_0} =& \fh{1}{\sqrt{6}} \pa{e_1}{x_2}\bigg|_{p_0}
=(0,0,0,\sqrt{\fh{3}{8}},0,0,0) &=& \fh{1}{4}e_3+\fh{\sqrt{5}}{4}e_6  \,,\\
D_{e_2}e_2|_{p_0} =& \fh{1}{\sqrt{6}} \pa{e_2}{x_2}\bigg|_{p_0}
=(-\fh{1}{4},0,0,0,-\fh{\sqrt{5}}{2},0,0)&=&\fh{\sqrt{5}}{4}e_4-\nu  \,,\\
D_{e_2}e_3|_{p_0} =& \fh{1}{\sqrt{6}} \pa{e_3}{x_2}\bigg|_{p_0}
=(0,-\fh{1}{4},0,0,0,0,0) &=& -\fh{1}{4}e_1 \,,\\
D_{e_2}e_4|_{p_0} =& \fh{1}{\sqrt{6}} \pa{e_4}{x_2}\bigg|_{p_0}
=(0,0,\sqrt{\fh{5}{24}},0,0,-\sqrt{\fh{2}{3}},0) &=& -\fh{\sqrt{5}}{4}e_2 +\fh{3}{4}e_5 \,,\\
D_{e_2}e_5|_{p_0} =& \fh{1}{\sqrt{6}} \pa{e_5}{x_2}\bigg|_{p_0}
=(-\fh{\sqrt{5}}{4},0,0,0,\fh{1}{2},0,0) &=& -\fh{3}{4}e_4 \,,\\
D_{e_2}e_6|_{p_0} =& \fh{1}{\sqrt{6}} \pa{e_6}{x_2}\bigg|_{p_0}
=(0,-\fh{\sqrt{5}}{4},0,0,0,0,0) &=& -\fh{\sqrt{5}}{4}e_1\,;\\
\end{aligned}
\end{gather}
\begin{gather}
\begin{aligned}
D_{e_3}e_1|_{p_0} =& \fh{1}{\sqrt{6}} \pa{e_1}{y_2}\bigg|_{p_0}
=(0,0,-\sqrt{\fh{3}{8}},0,0,0,0) &=& -\fh{1}{4}e_2-\fh{\sqrt{5}}{4}e_5  \,,\\
D_{e_3}e_2|_{p_0} =& \fh{1}{\sqrt{6}} \pa{e_2}{y_2}\bigg|_{p_0}
=(0,\fh{1}{4},0,0,0,0,0) &=& \fh{1}{4}e_1    \,,\\
D_{e_3}e_3|_{p_0} =& \fh{1}{\sqrt{6}} \pa{e_3}{y_2}\bigg|_{p_0}=
(-\fh{1}{4},0,0,0,-\fh{\sqrt{5}}{2},0,0) &=& \fh{\sqrt{5}}{4}e_4-\nu\,,\\
D_{e_3}e_4|_{p_0} =& \fh{1}{\sqrt{6}} \pa{e_4}{y_2}\bigg|_{p_0}
=(0,0,0,\sqrt{\fh{5}{24}},0,0,\sqrt{\fh{2}{3}}) &=& -\fh{\sqrt{5}}{4}e_3 +\fh{3}{4}e_6 \,,\\
D_{e_3}e_5|_{p_0} =& \fh{1}{\sqrt{6}} \pa{e_5}{y_2}\bigg|_{p_0}
=(0,\fh{\sqrt{5}}{4},0,0,0,0,0) &=& \fh{\sqrt{5}}{4}e_1 \,,\\
D_{e_3}e_6|_{p_0} =& \fh{1}{\sqrt{6}} \pa{e_6}{y_2}\bigg|_{p_0}
=(-\fh{\sqrt{5}}{4},0,0,0,\fh{1}{2},0,0) &=& -\fh{3}{4}e_4\,.\\
\end{aligned}
\end{gather}

\section{Calculations about matrices} \label{appendix.b}
In this section, we give details about calculations of matrices in the proof of Theorem \ref{theo.dim}. Firstly, we calculate  characteristic polynomials of $L_1$ and $L_2$ by sequences of elementary row and column  operations. Then we calculate the signs of $L_3.$

\subsection{The characteristic polynomial of the matrix $L_1$}\label{app.l1}
From \eqref{matrix.rep.L_1}, we have
$$L_1 = \begin{bmatrix}
\begin{smallmatrix}
    -2I_4 & \sqrt{6}A_2  & \sqrt{6}A_3  \\
     -\sqrt{6}A_2   & 8I_4 & -14A_1 \\
     -\sqrt{6}A_3   &14A_1  &  8I_4    
\end{smallmatrix}
\end{bmatrix},\text{ and }  \lambda I_{12}-L_1 = \begin{bmatrix}
\begin{smallmatrix}
    (\lambda+2)I_4 & -\sqrt{6}A_2  & -\sqrt{6}A_3  \\
     \sqrt{6}A_2   & (\lambda - 8)I_4 & 14A_1 \\
     \sqrt{6}A_3   &  -14A_1  & (\lambda - 8)I_4 
     \end{smallmatrix}
\end{bmatrix} .$$
\begin{equation} \label{eq.matrix1}
\begin{aligned}
\text{Thus}, \abs{\lambda I_{12} -L_1} =& \begin{vmatrix}
\begin{smallmatrix}
    (\lambda+2)I_4 & -\sqrt{6}A_2 & -\sqrt{6}A_3 \\
    &(\lambda-8)I_4 +\fh{6}{\lambda+2}A_2^2 & 14A_1 + \fh{6}{\lambda+2} A_2A_3 \\
    & -14A_1 + \fh{6}{\lambda+2} A_3A_2 & (\lambda-8)I_4 +\fh{6}{\lambda+2}A_3^2
    \end{smallmatrix}
\end{vmatrix} \\
=& \begin{vmatrix}
\begin{smallmatrix}
    (\lambda+2)I_4 &  &  \\
    &  (\lambda-8 -\fh{6}{\lambda+2})I_4 & (14-\fh{6}{\lambda+2})A_1 \\
    & -(14-\fh{6}{\lambda+2})A_1 & (\lambda-8 -\fh{6}{\lambda+2})I_4
    \end{smallmatrix}
\end{vmatrix} \\
=&  \begin{vmatrix}
\begin{smallmatrix}
     (\lambda+2)I_4 & & \\
     &  (\lambda-8 -\fh{6}{\lambda+2})I_4 & \\
     & & ((\lambda-8 -\fh{6}{\lambda+2})-(\lambda-8 -\fh{6}{\lambda+2})^{-1}(14-\fh{6}{\lambda+2})^2 )I_4
     \end{smallmatrix}
\end{vmatrix} \\
 =& (\lambda+2)^4((\lambda-8 -\fh{6}{\lambda+2})^2-(14-\fh{6}{\lambda+2})^2) ^4 = (\lambda+8)^4\lambda^4(\lambda-22)^4 .  
\end{aligned}
\end{equation}

\subsection{The characteristic polynomial of the matrix $L_2$}\label{app.l2}
From \eqref{matrx.rep.l2}, we have $$L_2=\begin{bmatrix}
\begin{smallmatrix}
    T_2 -10 & \sqrt{6}A_2  & \sqrt{6}A_3  \\
     -\sqrt{6}A_2   & T_2& -14A_1 \\
     -\sqrt{6}A_3   &14A_1  &  T_2
     \end{smallmatrix}
\end{bmatrix},$$  and we first calculate $\det(L_2+(8-\lambda)I_{27}).$

$$L_2+(8-\lambda)I_{27} = \begin{bmatrix}
\begin{smallmatrix}
    S_2 -10I_9 & B_2  & B_3  \\
     -B_2   & S_2     & -14A_1 \\
     -B_3   &14A_1    &  S_2
     \end{smallmatrix}
\end{bmatrix},\text{ with } S_2 = \left[
\begin{smallmatrix}
   ( 36-\lambda) I_6 & \\
   & (16-\lambda)I_3
   \end{smallmatrix}
\right], 
A_1 = 2\sqrt{-1}\begin{bmatrix}
\begin{smallmatrix}
    I_3 & & \\
    & -I_3 &\\
    &   & O_3
    \end{smallmatrix}
\end{bmatrix},$$
$$
B_2 = \sqrt{6}A_2 = 2\sqrt{3}\left[
\begin{smallmatrix}
    O_6  & -\beta \\
    \beta^t & O_3
    \end{smallmatrix}\right]
,
\beta = \begin{bmatrix}
\begin{smallmatrix}
    0 & -1 & 0 \\
    1 & 0&0 \\
    0 & 0 & 1\\
    0 & 0 & -1 \\
    1 & 0 & 0\\
    0 & 1 & 0    
\end{smallmatrix}
\end{bmatrix},
B_3 = \sqrt{6} A_3 = 2\sqrt{-3}
\left[\begin{smallmatrix}
    O_6 & \gamma \\
    \gamma^t & O_3 
\end{smallmatrix}\right],  
\gamma = \begin{bmatrix}
    \begin{smallmatrix}
   0 & -1 & 0 \\
    1 & 0&0 \\
    0 & 0 & 1\\
    0 & 0 &1 \\
    -1 & 0 & 0\\
    0 & -1 & 0     
    \end{smallmatrix}
\end{bmatrix}. $$
Moreover, $\beta^t \gamma = O_3, \ \gamma^t \beta = O_3, \ \beta^t \beta = 2I_3,  \ \gamma^t\gamma = 2I_3 ,$
$$\gamma \beta^t = \begin{bmatrix}\begin{smallmatrix}
    1 & 0 & 0 & 0 & 0 & -1 \\
    0 & 1 & 0 & 0 & 1 & 0 \\
    0 & 0 & 1 & -1 & 0 & 0 \\
    0 & 0 & 1 & -1 & 0 & 0 \\
    0 & -1 & 0 & 0 & -1 & 0 \\
    1 & 0 & 0 & 0 & 0 & -1 \\    
\end{smallmatrix}
\end{bmatrix}, \
\beta\beta^t  = \begin{bmatrix}
\begin{smallmatrix}
    1 & 0 & 0 & 0 & 0 & -1 \\
    0 & 1 & 0 & 0 & 1 & 0 \\
    0 & 0 & 1 & -1 & 0 & 0 \\
    0 & 0 & -1 & 1 & 0 & 0 \\
    0 & 1 & 0 & 0 & 1 & 0 \\
    -1 & 0 & 0 & 0 & 0 & 1 \\
    \end{smallmatrix}
\end{bmatrix} , \
\gamma\gamma^t = \begin{bmatrix}
\begin{smallmatrix}
    1 & 0 & 0 & 0 & 0 & 1 \\
    0 & 1 & 0 & 0 & -1 & 0 \\
    0 & 0 & 1 & 1 & 0 & 0 \\
    0 & 0 & 1 & 1 & 0 & 0 \\
    0 & -1 & 0 & 0 & 1 & 0 \\
    1 & 0 & 0 & 0 & 0 & 1 \\
    \end{smallmatrix}
\end{bmatrix}.$$
Let $J= \left[
\begin{smallmatrix}
     0 & 0 & 1\\
     0 & -1 & 0 \\
     1 &0 & 0
  \end{smallmatrix}    
\right]$ with $ J^2 = I_3$ and $K= \left[\begin{smallmatrix}
    I_3 & \\
        & -I_3
\end{smallmatrix}\right].$ Then we have
$$
\gamma \beta^t = \left[\begin{smallmatrix}
    I_3 & -J \\
    J & -I_3
\end{smallmatrix}\right], \ \beta \gamma^t = \left[\begin{smallmatrix}
    I_3 & J \\
    -J & -I_3
\end{smallmatrix}\right],\
\beta\beta^t  = \left[\begin{smallmatrix}
    I_3 & -J \\
    -J & I_3
\end{smallmatrix}\right], \ \gamma\gamma^t = \left[\begin{smallmatrix}
    I_3 & J \\
    J & I_3
\end{smallmatrix}\right].  $$
Moreover, $
(S_2-10I_9)^{-1} =\left[ \begin{smallmatrix}
    (26-\lambda)^{-1}I_6 & \\
    & (6-\lambda)^{-1}I_3
\end{smallmatrix}\right] .
$ So we have
\begin{equation*}
\begin{aligned}
& \left[\begin{smallmatrix}
    B_2 \\
    B_3
\end{smallmatrix}\right]
(S_2-10I_9)^{-1}
\left[
\begin{smallmatrix}
    B_2 &
    B_3
\end{smallmatrix}\right] \\
=& 12\begin{bmatrix}
\begin{smallmatrix}
    O_6 & -\beta \\
    \beta^t & O_3\\
    O_6 & \sqrt{-1}\gamma \\
    \sqrt{-1} \gamma^t &O_3
    \end{smallmatrix}
\end{bmatrix} 
\left[
\begin{smallmatrix}
    (26-\lambda)^{-1}I_6 & \\
    & (6-\lambda)^{-1}I_3
    \end{smallmatrix}
\right]
\left[
\begin{smallmatrix}
    O_6 & -\beta & O_6 & \sqrt{-1} \gamma\\
     \beta^t & O_3 & \sqrt{-1} \gamma^t & O_3
     \end{smallmatrix}
\right]\\
=& 12\begin{bmatrix}
\begin{smallmatrix}
    O_6 & -\beta \\
    \beta^t & O_3\\
    O_6 & \sqrt{-1}\gamma \\
    \sqrt{-1} \gamma^t &O_3
    \end{smallmatrix}
\end{bmatrix}
\left[
\begin{smallmatrix}
     O_6 & -(26-\lambda)^{-1}\beta & O_6 & (26-\lambda)^{-1}\sqrt{-1} \gamma\\
     (6-\lambda)^{-1}\beta^t & O_3 & (6-\lambda)^{-1}\sqrt{-1} \gamma^t & O_3
     \end{smallmatrix}
\right]\\
=& 12\begin{bmatrix}
\begin{smallmatrix}
    -(6-\lambda)^{-1}\beta\beta^t &   & -(6-\lambda)^{-1}\sqrt{-1}\beta \gamma^t &  \\
     & -(26-\lambda)^{-1}\beta^t\beta &  & (26-\lambda)^{-1}\sqrt{-1}\beta^t\gamma \\
    (6-\lambda)^{-1}\sqrt{-1}\gamma\beta^t &  & -(6-\lambda)^{-1}\gamma\gamma^t & \\
     & -(26-\lambda)^{-1}\sqrt{-1}\gamma^t\beta &  & -(26-\lambda)^{-1} \gamma^t\gamma
     \end{smallmatrix}
\end{bmatrix}\\
=&12\begin{bmatrix}
\begin{smallmatrix}
    -(6-\lambda)^{-1}\beta\beta^t &   & -(6-\lambda)^{-1}\sqrt{-1}\beta \gamma^t &   \\
      & -2(26-\lambda)^{-1}I_3 &   & O_3 \\
    (6-\lambda)^{-1}\sqrt{-1}\gamma\beta^t &  & -(6-\lambda)^{-1}\gamma\gamma^t & \\
      & O_3 &  &  -2(26-\lambda)^{-1}I_3
      \end{smallmatrix}
\end{bmatrix}.
\end{aligned}
\end{equation*}
 Then we shall calculate $\det(L_2+(8-\lambda)I_{27}).$ Let $\lambda_1 = (36-\lambda-\fh{12}{6-\lambda}), \lambda_2 = \fh{12}{6-\lambda}, \lambda_3 = 28+\fh{12}{6-\lambda}, $ and we have
 \begin{equation}\label{det.l2.c}
 \begin{aligned}
 &\abs{L_2+(8-\lambda)I_{27}} \\
=& \begin{vmatrix}
\begin{smallmatrix}
      S_2 -10I_9 & B_2  & B_3  \\
          & B_2(S_2 -10I_9)^{-1}B_2+S_2     & B_2(S_2 -10I_9)^{-1}B_3-14A_1 \\
        &B_3(S_2 -10I_9)^{-1}B_2 +14A_1    &  B_3(S_2 -10I_9)^{-1}B_3+S_2
        \end{smallmatrix}
    \end{vmatrix}\\
= &\abs{ S_2 -10I_9 } \begin{vmatrix}
\begin{smallmatrix}
    (36-\lambda)I_6 -12(6-\lambda)^{-1}\beta\beta^t & & -28\sqrt{-1}K-12\sqrt{-1}(6-\lambda)^{-1}\beta \gamma^t &\\
    & (16-\lambda-\fh{24}{26-\lambda})I_3 & & \\
    28\sqrt{-1}K+12\sqrt{-1}(6-\lambda)^{-1} \gamma \beta^t & &(36-\lambda)I_6 -12(6-\lambda)^{-1}\gamma\gamma^t & \\
    & & & (16-\lambda-\fh{24}{26-\lambda})I_3
    \end{smallmatrix}
\end{vmatrix}\\
= & (26-\lambda)^6(6-\lambda)^3 (16-\lambda-\fh{24}{26-\lambda})^6\begin{vmatrix}
\begin{smallmatrix}
    \lambda_1I_3 & \lambda_2 J & - \sqrt{-1} \lambda_3I_3 & -\sqrt{-1}\lambda_2J \\ 
    \lambda_2J &  \lambda_1I_3 & \sqrt{-1}\lambda_2 J & \lambda_3\sqrt{-1}I_3 \\
    \sqrt{-1}\lambda_3I_3 & -\sqrt{-1}\lambda_2J &  \lambda_1I_3 & -\lambda_2 J \\
    \sqrt{-1}\lambda_2 J & -\lambda_3\sqrt{-1}I_3  &  -\lambda_2J &  \lambda_1I_3
    \end{smallmatrix}
\end{vmatrix}\\
= & (6- \lambda)^3\big((26-\lambda)(16-\lambda)-24\big)^6
\begin{vmatrix}
\begin{smallmatrix}
    \lambda_1I_3 &  - \sqrt{-1} \lambda_3I_3 & \lambda_2 J & -\sqrt{-1}\lambda_2J \\ 
  \sqrt{-1}\lambda_3I_3   &  \lambda_1I_3 & -\sqrt{-1}\lambda_2 J &  -\lambda_2 J \\
   \lambda_2J  & \sqrt{-1}\lambda_2J &  \lambda_1I_3 & \lambda_3\sqrt{-1}I_3 \\
    \sqrt{-1}\lambda_2 J &  -\lambda_2J & -\lambda_3\sqrt{-1}I_3  &  \lambda_1I_3
    \end{smallmatrix}
\end{vmatrix}  .
 \end{aligned}
 \end{equation}
Since $\left[
\begin{smallmatrix}
    \lambda_1I_3 &  - \sqrt{-1} \lambda_3I_3 \\
     \sqrt{-1}\lambda_3I_3   &  \lambda_1I_3
      \end{smallmatrix}
\right]
\left[
\begin{smallmatrix}
    \lambda_1I_3 & \lambda_3\sqrt{-1}I_3 \\
    -\lambda_3\sqrt{-1}I_3  &  \lambda_1I_3
     \end{smallmatrix}
\right] = (\lambda_1^2-\lambda_3^2)I_6 , \begin{vmatrix}
\begin{smallmatrix}
    \lambda_1I_3 &  - \sqrt{-1} \lambda_3I_3 \\
     \sqrt{-1}\lambda_3I_3   &  \lambda_1I_3
     \end{smallmatrix}
\end{vmatrix} = (\lambda_1^2 -\lambda_3^2)^3
$, and 
\begin{equation*}
\begin{aligned}
& \left[
\begin{smallmatrix}
     \lambda_2J  & \sqrt{-1}\lambda_2J \\
     \sqrt{-1}\lambda_2 J &  -\lambda_2J
     \end{smallmatrix}
\right]
\left[
\begin{smallmatrix}
    \lambda_1I_3 &  - \sqrt{-1} \lambda_3I_3 \\
     \sqrt{-1}\lambda_3I_3   &  \lambda_1I_3
      \end{smallmatrix}
\right]^{-1}
\left[
\begin{smallmatrix}
    \lambda_2 J & -\sqrt{-1}\lambda_2J\\
    -\sqrt{-1}\lambda_2 J &  -\lambda_2 J 
     \end{smallmatrix}
\right] \\ 
 = & \fh{\lambda_2^2}{\lambda_1^2-\lambda_3^2}\left[
 \begin{smallmatrix}
     J  & \sqrt{-1}J \\
     \sqrt{-1} J &  -J
      \end{smallmatrix}
\right]
\left[
\begin{smallmatrix}
    \lambda_1I_3 &   \sqrt{-1} \lambda_3I_3 \\
     -\sqrt{-1}\lambda_3I_3   &  \lambda_1I_3
      \end{smallmatrix}
\right]
\left[
\begin{smallmatrix}
     J & -\sqrt{-1}J\\
    -\sqrt{-1} J &  - J 
     \end{smallmatrix}
\right] \\
  = & \fh{\lambda_2^2}{\lambda_1^2-\lambda_3^2}\left[
  \begin{smallmatrix}
     J  & \sqrt{-1}J \\
     \sqrt{-1} J &  -J
      \end{smallmatrix}
\right] \left[
\begin{smallmatrix}
    (\lambda_1+\lambda_3)J & -\sqrt{-1}(\lambda_1+\lambda_3) J \\
    -\sqrt{-1}(\lambda_1+\lambda_3) J & - (\lambda_1+\lambda_3)J
     \end{smallmatrix}
    \right] \\
    = & \fh{2\lambda_2^2(\lambda_1+\lambda_3)}{\lambda_1^2-\lambda_3^2}\left[
    \begin{smallmatrix}
        I_3 & -\sqrt{1}I_3 \\
        \sqrt{_1}I_3  & I_3
         \end{smallmatrix}
    \right].
\end{aligned}
\end{equation*}
Then we have 
\begin{equation}\label{det.l2.d}
\begin{aligned}
& \begin{vmatrix}
\begin{smallmatrix}
    \lambda_1I_3 &  - \sqrt{-1} \lambda_3I_3 & \lambda_2 J & -\sqrt{-1}\lambda_2J \\ 
  \sqrt{-1}\lambda_3I_3   &  \lambda_1I_3 & -\sqrt{-1}\lambda_2 J &  -\lambda_2 J \\
   \lambda_2J  & \sqrt{-1}\lambda_2J &  \lambda_1I_3 & \lambda_3\sqrt{-1}I_3 \\
    \sqrt{-1}\lambda_2 J &  -\lambda_2J & -\lambda_3\sqrt{-1}I_3  &  \lambda_1I_3
    \end{smallmatrix}
\end{vmatrix} \\
= & (\lambda_1^2-\lambda_3^2)^3
\begin{vmatrix}
 \left[
    \begin{smallmatrix}
     \lambda_1I_3 & \lambda_3\sqrt{-1}I_3 \\
    -\lambda_3\sqrt{-1}I_3  &  \lambda_1I_3
    \end{smallmatrix}
   \right]
-\fh{2 \lambda_2^2}{\lambda_1-\lambda_3}
    \left[
    \begin{smallmatrix}
        I_3 & -\sqrt{1}I_3 \\
        \sqrt{_1}I_3  & I_3
        \end{smallmatrix}
    \right]
\end{vmatrix} \\
=  & (\lambda_1^2-\lambda_3^2)^3
\begin{vmatrix}
\begin{smallmatrix}
(\lambda_1-\fh{2 \lambda_2^2}{\lambda_1-\lambda_3})I_3 & \sqrt{-1}(\lambda_3+\fh{2 \lambda_2^2}{\lambda_1-\lambda_3})I_3 \\
-\sqrt{-1}(\lambda_3+\fh{2 \lambda_2^2}{\lambda_1-\lambda_3})I_3 & (\lambda_1-\fh{2 \lambda_2^2}{\lambda_1-\lambda_3})I_3
\end{smallmatrix}
\end{vmatrix}\\
=& (\lambda_1^2-\lambda_3^2)^3((\lambda_1-\fh{2 \lambda_2^2}{\lambda_1-\lambda_3})^2-(\lambda_3+\fh{2 \lambda_2^2}{\lambda_1-\lambda_3})^2)^3\\
 = & (\lambda_1-\lambda_3-\fh{4 \lambda_2^2}{\lambda_1-\lambda_3})^3(\lambda_1-\lambda_3)^3(\lambda_1+\lambda_3)^6 \\
 = & ((\lambda_1-\lambda_3)^2-4\lambda_2^2)^3(64-\lambda)^6 = (8-\lambda - \fh{48}{6-\lambda})^3(8-\lambda)^3(64-\lambda)^6.
\end{aligned}
\end{equation}
Combined \eqref{det.l2.c} and \eqref{det.l2.d}, we have 
\begin{align*}
 &\abs{L_2+(8-\lambda)I_{27}} \\=& (6- \lambda)^3\big((26-\lambda)(16-\lambda)-24\big)^6(8-\lambda - \fh{48}{6-\lambda})^3(8-\lambda)^3(64-\lambda)^6 \\
=& \big((26-\lambda)(16-\lambda)- 24\big)^6\big((6-\lambda)(8-\lambda)-48\big)^3(8-\lambda)^3(64-\lambda)^6 \\
=&(28-\lambda)^6(14-\lambda)^6(14-\lambda)^3\lambda^3(8-\lambda)^3(64-\lambda)^6 .
\end{align*}
Thus, we have
\begin{equation} \label{eq.matrix2}
    \det(\lambda I_{27} - L_2) =  (\lambda+8)^3\lambda^3(\lambda- 6)^9(\lambda-20)^6(\lambda-56)^6.
\end{equation}

\subsection{The signs of the matrix $L_3$}\label{app.l3}
We say that two matrices $A, K \in  \CC^{n\times n}$ are \emph{congruent} if there exists an invertible
matrix $D\in  \CC^{n\times n}$ such that $K = \bar D^tAD$, and we denote  $A \sim K.$
\emph{Sylvester’s Law of Inertia}
states that under a congruence transform, the signs of eigenvalues of a Hermitian
matrix do not change. So we can calculate the signs of $L_3$ to get $\dim \ker L_3$.
  
We have $L_3=\begin{bmatrix}
\begin{smallmatrix}
    T_2 -10I_{16} & \sqrt{6}A_2  & \sqrt{6}A_3  \\
     -\sqrt{6}A_2   & T_2& -14A_1 \\
     -\sqrt{6}A_3   &14A_1  &  T_2
\end{smallmatrix}
\end{bmatrix} ,$ and $L_3$ is defined in \eqref{matrix.rep.L_3}.  
Let \begin{gather}\label{det.l3.a}
B= \left[
\begin{smallmatrix}
     T_2 +6 A_2 (T_2 -10I_{16})^{-1}A_2& -14A_1 +6 A_2 (T_2 -10I_{16})^{-1}A_3 \\
       14A_1 + 6 A_3 (T_2 -10I_{16})^{-1}A_2 &  T_2 +6 A_3 (T_2 -10I_{16})^{-1}A_3
    \end{smallmatrix}
\right].
\end{gather}
Then we have \begin{gather}\label{det.l3.f}
L_3 \sim \begin{bmatrix}
\begin{smallmatrix}
     T_2 -10I_{16} &   &   \\
        & T_2 +6 A_2 (T_2 -10I_{16})^{-1}A_2& -14A_1 +6 A_2 (T_2 -10I_{16})^{-1}A_3 \\
       &14A_1 + 6 A_3 (T_2 -10I_{16})^{-1}A_2 &  T_2 +6 A_3 (T_2 -10I_{16})^{-1}A_3
    \end{smallmatrix}
\end{bmatrix} \sim 
\left[
\begin{smallmatrix}
    I_{16} & \\
    & B
\end{smallmatrix}
\right].\end{gather}
Moreover,
\begin{gather}\label{det.l3.b}
    \begin{aligned}
 \left[
 \begin{smallmatrix}
     6 A_2 (T_2 -10I_{16})^{-1}A_2&6 A_2 (T_2 -10I_{16})^{-1}A_3 \\
        6 A_3 (T_2 -10I_{16})^{-1}A_2 &  6 A_3 (T_2 -10I_{16})^{-1}A_3
     \end{smallmatrix}
 \right]
 = &
6 \begin{bmatrix}
\begin{smallmatrix}
  O_8 & J_1 \\
     -J_1 & K_1 \\
     O_8 & J_2 \\
     J_2 & K_2   
\end{smallmatrix}
 \end{bmatrix}
 \left[  \begin{smallmatrix}
     \fh{1}{50}I_8 & \\
     & \fh{1}{10}I_8
     \end{smallmatrix}
 \right]
 \left[ \begin{smallmatrix}
     O_8 & J_1 & O_8 & J_2 \\
     -J_1 & K_1& J_2 & K_2
     \end{smallmatrix}
 \right]\\
 =& 6\begin{bmatrix}
 \begin{smallmatrix}
     -3\lambda_2I_8 & \lambda_2J_1K_1 & \lambda_2 J_1J_2&  \lambda_2 J_1K_2 \\
     -\lambda_2 K_1J_1 & -(3\lambda_1+ 4 \lambda_2)I_8 & \lambda_2 K_1J_2 &
     -\lambda_1 J_1J_2 + \lambda_2K_1K_2 \\
     -\lambda_2 J_2J_1 & \lambda_2J_2K_1 &-3\lambda_2I_8 & \lambda_2J_2K_2 \\
     -\lambda_2K_2J_1 & \lambda_1J_2J_1 +\lambda_2 K_2K_1 & \lambda_2K_2J_2 & -(3\lambda_1+4\lambda_2)I_8
     \end{smallmatrix}
 \end{bmatrix} ,
 \end{aligned}
 \end{gather}
 where $\lambda_1 = \fh{1}{50}, \lambda_2= \fh{1}{10},$ and $J_1,K_1,J_2,K_2$ are defined  in \eqref{matrix.rep.L_3}. In particular, we have  
\begin{gather*}
 J_1^t = J_1,\  J_2^t = J_2,\  K_1^t = -K_1, \ K_2^t = K_2,\\
J_1^2 = 3I_8, \ K_1^2 = -4I_8,\  J_2^2 = -3I_8, \ K_2^2 = -4I_8,\ 
\\J_1J_2 = 3 \ii \left[
 \begin{smallmatrix}
     -I_4 & \\
     & I_4
 \end{smallmatrix}
 \right], \  K_1K_2 = 4 \ii
 \left[ \begin{smallmatrix}
     -I_4 & \\
     & I_4
 \end{smallmatrix}\right] ,\\
 J_1K_1 = -J_2 K_2 = J_3, \ K_1J_1 =K_2J_2= -J_3^t,\ \\
 J_1K_2 =J_2K_1 = \ii J_4, \  K_1 J_2 =K_2J_1 =  -\ii J_4^t,\\
 \text{ with }J_3 =2\sqrt{3}\begin{bmatrix}
 \begin{smallmatrix}
     &&&&&&&1\\
     &&&&&&-1& \\
     &&&&&1&& \\
     &&&&1&&& \\
     &&&1&&&& \\
     &&-1&&&&& \\
     &1&&&&&& \\
    1 &&&&&&& \\
      \end{smallmatrix}
    \end{bmatrix}
 , \ J_4 =2\sqrt{3}\begin{bmatrix}\begin{smallmatrix}
     &&&&&&&-1\\
     &&&&&&1& \\
     &&&&&-1&& \\
     &&&&-1&&& \\
     &&&1&&&& \\
     &&-1&&&&& \\
     &1&&&&&& \\
    1 &&&&&&& \\
      \end{smallmatrix}
 \end{bmatrix}.\end{gather*}
Let $ K_3 = \left[ \begin{smallmatrix}
     -I_4 & \\
     & I_4 
 \end{smallmatrix}\right],$  then $ K_3^2 =I_8. 
 $
 Moreover, $$ K_3J_3 = J_4 , \ 
 K_3 J_4 =J_3 ,\  J_3^t J_4 = J_4^tJ_3 =-12K_3, \   J_3^t J_3 = J_4^t J_4 =12I_8. $$
 Combined \eqref{det.l3.a} and \eqref{det.l3.b}, we have  \begin{gather}\label{det.l3.c}
 \begin{aligned}
B=& \begin{bmatrix}
 \begin{smallmatrix}
    (60 -18\lambda_2)I_8 & 6\lambda_2J_3 & (42+18\lambda_2)\ii K_3&  6 \ii\lambda_2 J_4 \\
     6\lambda_2 J_3^t & (20-6(3\lambda_1+ 4 \lambda_2))I_8 & -6\ii\lambda_2 J_4^t &
   ( 14+6 (4\lambda_2-3\lambda_1))\ii K_3\\
     -(42+18\lambda_2)\ii K_3 & 6\ii\lambda_2 J_4 &
     (60-18\lambda_2)I_8 & -6\lambda_2 J_3 \\
     -6 \ii\lambda_2 J_4^t & -( 14+6 (4\lambda_2-3\lambda_1))\ii K_3 & -6\lambda_2J_3^t & (20 -6(3\lambda_1+4\lambda_2))I_8
     \end{smallmatrix}
 \end{bmatrix}\\
 \sim &
 \begin{bmatrix}
 \begin{smallmatrix}
    (60 -18\lambda_2)I_8 & (42+18\lambda_2)\ii K_3& 6\lambda_2J_3 &  6 \ii\lambda_2 J_4 \\
     -(42+18\lambda_2)\ii K_3    & (60-18\lambda_2)I_8& 6\ii\lambda_2 J_4 & -6\lambda_2 J_3\\
     6\lambda_2 J_3^t & -6\ii\lambda_2 J_4^t    &
     (20-6(3\lambda_1+ 4 \lambda_2))I_8 &   ( 14+6 (4\lambda_2-3\lambda_1))\ii K_3\\
     -6 \ii\lambda_2 J_4^t &  -6\lambda_2J_3^t &
     -( 14+6 (4\lambda_2-3\lambda_1))\ii K_3 &(20 -6(3\lambda_1+4\lambda_2))I_8
     \end{smallmatrix}
 \end{bmatrix} = \left[\begin{smallmatrix}
    D_1 & D_2\\
    \bar{D_2}^t & D_3
\end{smallmatrix}\right],
\end{aligned}
\end{gather}
where $D_1,D_2,D_3\in \CC^{16\times 16}.$
Since $(60 - 18\lambda_2)^2 - (42+18\lambda_2)^2 = (60 - 1.8)^2 - (42+ 1.8)^2 >0,$
we have  $D_1 \sim I_{16}. $ Let $\lambda_3 = (60-18\lambda_2) = 58.2, \lambda_4 = (42+ 18\lambda_2)\ii = 43.8\ii.$ Then $D_1^{-1} = (\lambda_3^2+ \lambda_4^2)^{-1}\bar{D_1}.$ So we have
 \begin{gather}\label{det.l3.d}
\begin{aligned}
-\bar{D_2}^tD_1^{-1}D_2 = &- \fh{36\lambda_2^2}{\lambda_3^2+\lambda_4^2}
\left[\begin{smallmatrix}
    J_3^t & -\ii J_4^t \\
    -\ii J_4^t & -J_3^t
\end{smallmatrix}\right] 
\left[\begin{smallmatrix}
    \lambda_3I_8 & -\lambda_4K_3 \\
    \lambda_4K_3 & \lambda_3I_8
\end{smallmatrix}\right] 
\left[\begin{smallmatrix}
    J_3 & \ii J_4 \\
    \ii J_4& -J_3
\end{smallmatrix}\right]\\
 =& - \fh{36\lambda_2^2}{\lambda_3^2+\lambda_4^2}
\left[\begin{smallmatrix}
    J_3^t & -\ii J_4^t \\
    -\ii J_4^t & -J_3^t
\end{smallmatrix}\right] 
\left[\begin{smallmatrix}
    (\lambda_3-\lambda_4 \ii)J_3 & \ii(\lambda_3-\lambda_4 \ii) J_4 \\
  \ii(\lambda_3-\lambda_4 \ii) J_4  & -(\lambda_3-\lambda_4 \ii)J_3
\end{smallmatrix}\right]\\
= &- \fh{36\lambda_2^2 \times 24}{\lambda_3+\lambda_4 \ii} \left[\begin{smallmatrix}
     I_8 & -\ii K_3\\
     \ii K_3 & I_8
\end{smallmatrix}\right] = -\fh{\fh{36}{100}\times 24}{58.2- 43.8}\left[\begin{smallmatrix}
     I_8 & -\ii K_3\\
     \ii K_3 & I_8
\end{smallmatrix}\right]\\
=& -\fh{3}{5}\left[\begin{smallmatrix}
     I_8 & -\ii K_3\\
     \ii K_3 & I_8
\end{smallmatrix}\right].
    \end{aligned}\end{gather}
Combined \eqref{det.l3.c} and \eqref{det.l3.d}, we have 
\begin{gather}\label{det.l3.e}
B \sim \left[\begin{smallmatrix}
    D_1 & \\
    & D_3 -\bar{D_2}^tD_1^{-1}D_2 
\end{smallmatrix}\right]
\sim \left[\begin{smallmatrix}
    I_{16} & \\
    & D_3 -\bar{D_2}^tD_1^{-1}D_2 
\end{smallmatrix}\right], 
\end{gather}
with $
D_3 = \left[\begin{smallmatrix}
    (17.24)I_8 & (16.04) \ii K_3 \\
    -(16.04) \ii K_3 &  (17.24)I_8
\end{smallmatrix}\right], \
D_3 -\bar{D_2}^tD_1^{-1}D_2 = 16.64 \left[\begin{smallmatrix}
    I_8 & \ii K_3 \\
    -\ii K_3 & I_8
\end{smallmatrix}\right] \sim \left[\begin{smallmatrix}
    I_8 & \\
    & O_8
\end{smallmatrix}\right].$
 Thus, combined \eqref{det.l3.f} and \eqref{det.l3.e}, we have 
 \begin{equation}\label{matrix3}
     L_3 \sim \left[  \begin{smallmatrix}
     I_{40} & \\
     & O_8
 \end{smallmatrix}
 \right]. 
 \end{equation}

\section{Code for calculations about matrices} \label{code}
 In this section, we provide a MATLAB code in order to verify calculations about matrices. We perform \emph{symbolic computations} in MATLAB to obtain accurate results. For $k = 1,2,3,4,$ we can get the eigenvalues of corresponding matrices $L_k$ from the results obtained by running the MATLAB code. 
 
 In MATLAB, 
 \begin{itemize}
     \item "eye(n)" is $I_n$ for $n\in \ZZ^+$;
     \item "zeros(n)" is $O_n$ for $n\in \ZZ^+$;
     \item "zeros(m,n)" is a matrix in $\CC^{m\times n}$ whose all entries are $0$ for $m,n\in \ZZ^+$;
     \item 
     "flip(A)" is $[e_{ij}](1\leq i,j \leq n)$ with $e_{ij} = a_{(n+1-j)(n+1-i)},$ for a matrix $A = [a_{ij}]\in \CC^n$ with $n\in \ZZ^+$;
      \item "sqrt(x)" is $\sqrt{x}$ for $x\geq 0$;
      \item "1i"  is $\ii$;
      \item "[1,0,0,0;0,1,0,0;0,0,-1,0;0,0,0,-1]" is the matrix $\left[\begin{smallmatrix}
          1& 0 &0 &0\\
          0&1&0&0\\
          0&0&-1&0\\
          0&0&0&-1
      \end{smallmatrix}\right]$;
      \item "subs(L)" is a symbolic substitution function in MATLAB. It substitutes the symbolic variables in the symbolic expression L with their assigned values in the workspace, and evaluates the resulting expression symbolically;
      \item “disp(X)” is a standard display function in MATLAB that outputs the content of the expression X directly to the command window;
      \item  "simplify()" function in MATLAB performs algebraic simplification on symbolic expressions, reducing them to a simpler and more concise form through algebraic and trigonometric transformations;
      \item "solve(g, lambda)" is a symbolic function in MATLAB. It solves the symbolic equation g analytically for the specified variable lambda, and returns the exact analytical solution;
      \item "S$'$ " is the transpose of the column vector $S$ and "S$'$ " is a row vector.
 \end{itemize}
   The MATLAB code is listed below.

\begin{lstlisting}[language=Matlab]
%--begining of the code
%k = 1-----------------%
clc
clear
%-- definitions of matrix
syms I_4 A_1 A_2 A_3  ; syms q_1 lambda ;
q_1 = sqrt(6); I_4 = eye(4) ; 
A_1 = 1i*[1,0,0,0;0,1,0,0;0,0,-1,0;0,0,0,-1];
A_2= [0,0,0,1;0,0,-1,0;0,1,0,0;-1,0,0,0];
A_3= 1i*[0,0,0,-1;0,0,1,0;0,1,0,0;-1,0,0,0]; 
L= [-2*I_4, q_1*A_2, q_1*A_3;-q_1*A_2, 8*I_4 , -14 *A_1;-q_1*A_3 , 14* A_1 ,8*I_4];
M=subs(L); g = simplify(det(lambda*eye(12) - M));
disp('If k = 1, the characteristic polynomial for the matrix L is:');
disp(g);
S= solve(g,lambda);
disp('The eigenvalues of L are:'); disp(S');

%k = 2------------------%

syms I_3  O_3 ; syms A_1 A_2 A_3  T_2 ;syms q_1 q_2 lambda;syms beta gamma 
q_2 = sqrt(2); q_1 = sqrt(6);  I_3 = eye(3); O_3 = zeros(3);
A_1= [2*1i*I_3 , O_3, O_3; O_3, -2*1i*I_3,O_3; O_3, O_3 ,O_3 ];
beta = [0,1,0;-1,0,0;0,0,-1;0,0,1;-1,0,0;0,-1,0];
gamma = [0,-1,0;1,0,0;0,0,1;0,0,1;-1,0,0;0,-1,0];
A_2 = [zeros(6), q_2* beta ;-q_2*beta' , zeros(3)];
A_3 =[zeros(6), 1i*q_2*gamma; 1i*q_2*gamma', zeros(3)];
T_2 = [28*I_3, O_3,O_3;O_3, 28*I_3, O_3; O_3,O_3,8*I_3];
L=[T_2 - 10*eye(9), q_1*A_2, q_1*A_3;-q_1*A_2, T_2, -14*A_1; -q_1*A_3 ,14*A_1, T_2];
M=subs(L); g =simplify(det(lambda*eye(27) - M));
disp('If k = 2, the characteristic polynomail for the matrix L is');
disp(g);
S= solve(g,lambda);
disp('The eigenvalues of L are'); disp(S');

%k = 3-------------------%

syms I_4  O_4 ; syms I_16 ; syms J_1 K_1 J_2 K_2  O_8 I_8 N_1 ;
syms A_1 A_2 A_3  T_2 I_16 ; syms q_1 q_2 lambda; syms I_48 
q_2 = sqrt(3); q_1 = sqrt(6); I_4 = eye(4); I_8 = eye(8);
I_16 = eye(16); I_48 = eye(48); O_8 = zeros(8); O_4 =zeros(4);
N_1 = 1i*[I_4,O_4;O_4,-I_4];
A_1= [3*N_1 ,O_8;O_8, N_1];
J_1 =q_2* diag([1,-1,-1,-1,1,-1,-1,-1]);
K_1 = 2*flip(diag([-1,-1,1,1,-1,-1,1,1]));
A_2=[O_8,J_1;-J_1,K_1];
J_2= q_2*1i*diag([-1,1,1,1,1,-1,-1,-1]);
K_2 = 2*1i*flip(diag([-1,-1,1,1,1,1,-1,-1]));
A_3 = [O_8, J_2;J_2,K_2];
T_2 = [60*I_8, O_8;O_8, 20*I_8];
L=[T_2 - 10*I_16, q_1*A_2, q_1*A_3; -q_1*A_2, T_2, -14*A_1; -q_1*A_3 ,14*A_1, T_2];
M=subs(L); g = simplify(det(lambda*I_48 - M));
disp('If k = 3, the characteristic polynomail for the matrix L is');
disp(g);
S= solve(g,lambda);
disp('The eigenvalues of L are'); disp(S');

%k = 4-------------------%

syms I_5  O_5  J_1 ; syms I_75 ; syms O_10  I_10 ;
syms A_1 A_2 A_3  T_2 I_25 ; syms q_1 lambda ;
q_1 = sqrt(6); I_5 = eye(5); I_10 = eye(10);
I_25 = eye(25); I_75 = eye(75); O_5 = zeros(5); O_10 = zeros(10);
A_1 = [[4*1i*I_5,O_5;O_5,-4*1i* I_5],zeros(10,15) ; 
    zeros(15,10),[[2*1i*I_5, O_5;O_5,-2*1i*I_5],zeros(10,5);zeros(5,10),O_5]];
J_1 =flip(diag([1,-1,1,-1,1]));
A_2=[O_10,-2*I_10,zeros(10,5);2*I_10,O_10,[-q_1*I_5;-q_1*J_1];
    zeros(5,10),[q_1*I_5,q_1*J_1],O_5];
A_3=[O_10,[2*1i*I_5,O_5;O_5,-2*1i*I_5],zeros(10,5);
    [2*1i*I_5,O_5;O_5,-2*1i*I_5],O_10,[q_1*1i*I_5;-q_1*1i*J_1];
    zeros(5,10),[q_1*1i*I_5,-q_1*1i*J_1],O_5];
T_2 = [104*I_10, zeros(10,15); 
    O_10,[44*I_10,zeros(10,5)];zeros(5,10),[zeros(5,10) ,24*I_5]];
L=[T_2 - 10*I_25, q_1*A_2, q_1*A_3; -q_1*A_2, T_2, -14*A_1; -q_1*A_3 ,14*A_1, T_2];
M=subs(L); g =simplify(det(lambda*I_75 - M));
disp('If k = 4, the characteristic polynomail for the matrix L is');
disp(g);
S= solve(g,lambda);
disp('The eigenvalues of L are'); disp(S');
\end{lstlisting}

We execute the MATLAB code and present the corresponding results below.
{\fontsize{10}{13}
\begin{verbatim}
If k = 1, the characteristic polynomial for the matrix L is:
lambda^4*(- lambda^2 + 14*lambda + 176)^4
 
The eigenvalues of L are:
[-8, -8, -8, -8, 0, 0, 0, 0, 22, 22, 22, 22]
 
If k = 2, the characteristic polynomail for the matrix L is
-lambda^3*(lambda + 8)^2*(lambda - 20)^3*(lambda^2 - 26*lambda + 120)^3*(lambda^2 
-62*lambda + 336)^5*(- lambda^3 + 54*lambda^2 + 160*lambda - 2688)
 
The eigenvalues of L are
[-8, -8, -8, 0, 0, 0, 6, 6, 6, 6, 6, 6, 6, 6, 6, 
20, 20, 20, 20, 20, 20, 56, 56, 56, 56, 56, 56]
 
If k = 3, the characteristic polynomail for the matrix L is
lambda^8*(lambda^5 - 220*lambda^4 + 16820*lambda^3 
- 566720*lambda^2 + 8472000*lambda - 44808192)^8
 
The eigenvalues of L are
[0, 0, 0, 0, 0, 0, 0, 0, 12, 12, 12, 12, 12, 12, 12, 12, 
22, 22, 22, 22, 22, 22, 22, 22, 32, 32, 32, 32, 32, 32, 32, 32, 
52, 52, 52, 52, 52, 52, 52, 52, 102, 102, 102, 102, 102, 102, 102, 102]
 
If k = 4, the characteristic polynomail for the matrix L is
(lambda^2 - 166*lambda + 6720)^10*(lambda^3 - 46*lambda^2 + 560*lambda 
- 1280)^5*(lambda^4 - 266*lambda^3 + 20440*lambda^2 - 591360*lambda + 5529600)^10
 
The eigenvalues of L are
[16, 16, 16, 16, 16, 36, 36, 36, 36, 36, 36, 36, 36, 36, 36,
70, 70, 70, 70, 70, 70, 70, 70, 70, 70, 96, 96, 96, 96, 96, 96, 96, 96, 96, 96,
160, 160, 160, 160, 160, 160, 160, 160, 160, 160, 
15 - 145^(1/2), 15 - 145^(1/2), 15 - 145^(1/2), 15 - 145^(1/2), 15 - 145^(1/2),
35 - 265^(1/2), 35 - 265^(1/2), 35 - 265^(1/2), 35 - 265^(1/2), 35 - 265^(1/2), 
35 - 265^(1/2), 35 - 265^(1/2), 35 - 265^(1/2), 35 - 265^(1/2), 35 - 265^(1/2),
145^(1/2) + 15, 145^(1/2) + 15, 145^(1/2) + 15, 145^(1/2) + 15, 145^(1/2) + 15, 
265^(1/2) + 35, 265^(1/2) + 35, 265^(1/2) + 35, 265^(1/2) + 35, 265^(1/2) + 35, 
265^(1/2) + 35, 265^(1/2) + 35, 265^(1/2) + 35, 265^(1/2) + 35, 265^(1/2) + 35]
  \end{verbatim}
}

\ 

\bigskip

 {\footnotesize

\ 
 
\noindent{Acknowledgment: The authors would like to thank Yuanlong Xin and Zhihan Wang for helpful comments on an earlier draft. The first author is partially supported by NSFC 12371053 and NSFC 12526203.} 

\medskip

\ 

\noindent\textbf{Data availability:} No datasets were generated or analysed during the current study.
 
}

\bigskip

\noindent\textbf{\large Declarations} 

 \
 
 {\footnotesize

\noindent\textbf{Conflict of interest:} The authors declare that they have no conflict of interest.
\medskip

\medskip

 }

\bibliographystyle{plain}
\bibliography{ref}

@article {Lawson,
    AUTHOR = {Lawson, Jr., H. B. and Osserman, R.},
     TITLE = {Non-existence, non-uniqueness and irregularity of solutions to
              the minimal surface system},
   JOURNAL = {Acta Math.},
  FJOURNAL = {Acta Mathematica},
    VOLUME = {139},
      YEAR = {1977},
    NUMBER = {1-2},
     PAGES = {1--17},
      ISSN = {0001-5962,1871-2509},
   MRCLASS = {35J60 (53A10 58E12)},
  MRNUMBER = {452745},
       DOI = {10.1007/BF02392232},
       URL = {https://doi.org/10.1007/BF02392232},
}

@article {Tanno,
    AUTHOR = {Tanno, Sh\^ukichi},
     TITLE = {The first eigenvalue of the {L}aplacian on spheres},
   JOURNAL = {Tohoku Math. J. (2)},
  FJOURNAL = {The Tohoku Mathematical Journal. Second Series},
    VOLUME = {31},
      YEAR = {1979},
    NUMBER = {2},
     PAGES = {179--185},
      ISSN = {0040-8735,2186-585X},
   MRCLASS = {58G25 (53C20)},
  MRNUMBER = {538918},
MRREVIEWER = {R.\ S.\ Millman},
       DOI = {10.2748/tmj/1178229837},
       URL = {https://doi.org/10.2748/tmj/1178229837},
}

@article {Simons,
    AUTHOR = {Simons, James},
     TITLE = {Minimal varieties in riemannian manifolds},
   JOURNAL = {Ann. of Math. (2)},
  FJOURNAL = {Annals of Mathematics. Second Series},
    VOLUME = {88},
      YEAR = {1968},
     PAGES = {62--105},
      ISSN = {0003-486X},
   MRCLASS = {53.04 (35.00)},
  MRNUMBER = {233295},
MRREVIEWER = {W.\ F.\ Pohl},
       DOI = {10.2307/1970556},
       URL = {https://doi.org/10.2307/1970556},
}

@article {Ding_Yuan,
    AUTHOR = {Ding, Weiyue and Yuan, Yu},
     TITLE = {Resolving the singularities of the minimal {H}opf cones},
   JOURNAL = {J. Partial Differential Equations},
  FJOURNAL = {Journal of Partial Differential Equations},
    VOLUME = {19},
      YEAR = {2006},
    NUMBER = {3},
     PAGES = {218--231},
      ISSN = {1000-940X,2079-732X},
   MRCLASS = {49Q05 (53A10)},
  MRNUMBER = {2252978},
MRREVIEWER = {Marc\ Michel\ Soret},
}

@article {Lotay,
    AUTHOR = {Lotay, Jason D.},
     TITLE = {Stability of coassociative conical singularities},
   JOURNAL = {Comm. Anal. Geom.},
  FJOURNAL = {Communications in Analysis and Geometry},
    VOLUME = {20},
      YEAR = {2012},
    NUMBER = {4},
     PAGES = {803--867},
      ISSN = {1019-8385,1944-9992},
   MRCLASS = {53C38},
  MRNUMBER = {2981841},
MRREVIEWER = {Andreas\ Savas-Halilaj},
       DOI = {10.4310/CAG.2012.v20.n4.a5},
       URL = {https://doi.org/10.4310/CAG.2012.v20.n4.a5},
}

@book {Xin,
    AUTHOR = {Xin, Yuanlong},
     TITLE = {Minimal submanifolds and related topics},
    SERIES = {Nankai Tracts in Mathematics},
    VOLUME = {16},
   EDITION = {Second},
 PUBLISHER = {World Scientific Publishing Co. Pte. Ltd., Hackensack, NJ},
      YEAR = {2019},
     PAGES = {xvi+380},
      ISBN = {978-981-3236-05-9},
   MRCLASS = {53C42 (53-02 53A10 53C40)},
  MRNUMBER = {3837570},
}

@book {GMT,
    AUTHOR = {Simon, Leon},
     TITLE = {Lectures on geometric measure theory},
    SERIES = {Proceedings of the Centre for Mathematical Analysis,
              Australian National University},
    VOLUME = {3},
 PUBLISHER = {Australian National University, Centre for Mathematical
              Analysis, Canberra},
      YEAR = {1983},
     PAGES = {vii+272},
      ISBN = {0-86784-429-9},
   MRCLASS = {49-01 (28A75 49F20)},
  MRNUMBER = {756417},
MRREVIEWER = {J.\ S.\ Joel},
}

@book {Giaquita,
    AUTHOR = {Giaquinta, Mariano and Martinazzi, Luca},
     TITLE = {An introduction to the regularity theory for elliptic systems,
              harmonic maps and minimal graphs},
    SERIES = {Appunti. Scuola Normale Superiore di Pisa (Nuova Serie)
              [Lecture Notes. Scuola Normale Superiore di Pisa (New
              Series)]},
    VOLUME = {11},
   EDITION = {Second},
 PUBLISHER = {Edizioni della Normale, Pisa},
      YEAR = {2012},
     PAGES = {xiv+366},
      ISBN = {978-88-7642-442-7; 978-88-7642-443-4},
   MRCLASS = {35-02 (35B65 35J20 35J60 58E20)},
  MRNUMBER = {3099262},
       DOI = {10.1007/978-88-7642-443-4},
       URL = {https://doi.org/10.1007/978-88-7642-443-4},
}

@book {Morrey,
    AUTHOR = {Morrey, Jr., Charles B.},
     TITLE = {Multiple integrals in the calculus of variations},
    SERIES = {Classics in Mathematics},
      NOTE = {Reprint of the 1966 edition [MR0202511]},
 PUBLISHER = {Springer-Verlag, Berlin},
      YEAR = {2008},
     PAGES = {x+506},
      ISBN = {978-3-540-69915-6},
   MRCLASS = {49-02 (49J10 49J45)},
  MRNUMBER = {2492985},
       DOI = {10.1007/978-3-540-69952-1},
       URL = {https://doi.org/10.1007/978-3-540-69952-1},
}

@article {Allard1,
    AUTHOR = {Allard, William K.},
     TITLE = {On the first variation of a varifold},
   JOURNAL = {Ann. of Math. (2)},
  FJOURNAL = {Annals of Mathematics. Second Series},
    VOLUME = {95},
      YEAR = {1972},
     PAGES = {417--491},
      ISSN = {0003-486X},
   MRCLASS = {49F20},
  MRNUMBER = {307015},
MRREVIEWER = {M.\ Klingmann},
       DOI = {10.2307/1970868},
       URL = {https://doi.org/10.2307/1970868},
}

@article {allard.radial,
    AUTHOR = {Allard, William K. and Almgren, Jr., Frederick J.},
     TITLE = {On the radial behavior of minimal surfaces and the uniqueness
              of their tangent cones},
   JOURNAL = {Ann. of Math. (2)},
  FJOURNAL = {Annals of Mathematics. Second Series},
    VOLUME = {113},
      YEAR = {1981},
    NUMBER = {2},
     PAGES = {215--265},
      ISSN = {0003-486X},
   MRCLASS = {49F22 (53A10 58E12)},
  MRNUMBER = {607893},
MRREVIEWER = {E.\ Giusti},
       DOI = {10.2307/2006984},
       URL = {https://doi.org/10.2307/2006984},
}

@book {Axler,
    AUTHOR = {Axler, Sheldon and Bourdon, Paul and Ramey, Wade},
     TITLE = {Harmonic function theory},
    SERIES = {Graduate Texts in Mathematics},
    VOLUME = {137},
   EDITION = {Second},
 PUBLISHER = {Springer-Verlag, New York},
      YEAR = {2001},
     PAGES = {xii+259},
      ISBN = {0-387-95218-7},
   MRCLASS = {31-01 (30-01 46Exx)},
  MRNUMBER = {1805196},
       DOI = {10.1007/978-1-4757-8137-3},
       URL = {https://doi.org/10.1007/978-1-4757-8137-3},
}

@article {asymptotics,
    AUTHOR = {Simon, Leon},
     TITLE = {Asymptotics for a class of nonlinear evolution equations, with
              applications to geometric problems},
   JOURNAL = {Ann. of Math. (2)},
  FJOURNAL = {Annals of Mathematics. Second Series},
    VOLUME = {118},
      YEAR = {1983},
    NUMBER = {3},
     PAGES = {525--571},
      ISSN = {0003-486X,1939-8980},
   MRCLASS = {58G11 (35B40 49F99 58E20)},
  MRNUMBER = {727703},
MRREVIEWER = {Helmut\ Kaul},
       DOI = {10.2307/2006981},
       URL = {https://doi.org/10.2307/2006981},
}

@article{simonsolo,
  title={Minimal hypersurfaces asymptotic to quadratic cones in  $\mathbf{R}^{n+ 1}$},
  author={Simon, Leon and Solomon, Bruce},
  journal={Inventiones mathematicae},
  volume={86},
  number={3},
  pages={535--551},
  year={1986},
  publisher={Springer}
}

@incollection {isolated.singularity,
    AUTHOR = {Simon, Leon},
     TITLE = {Isolated singularities of extrema of geometric variational
              problems},
 BOOKTITLE = {Harmonic mappings and minimal immersions ({M}ontecatini,
              1984)},
    SERIES = {Lecture Notes in Math.},
    VOLUME = {1161},
     PAGES = {206--277},
 PUBLISHER = {Springer, Berlin},
      YEAR = {1985},
      ISBN = {3-540-16040-X},
   MRCLASS = {58E15 (58E20)},
  MRNUMBER = {821971},
MRREVIEWER = {Harold\ Parks},
       DOI = {10.1007/BFb0075139},
       URL = {https://doi.org/10.1007/BFb0075139},
}

@article{Simonuniqueness,
  title={Uniqueness of some cylindrical tangent cones},
  author={Simon, Leon},
  journal={Communications in Analysis and Geometry},
  volume={2},
  number={1},
  pages={1--33},
  year={1994},
  publisher={International Press of Boston}
}

@article{Sim94,
author = {Leon Simon},
title = {{Cylindrical tangent cones and the singular set of minimal submanifolds}},
volume = {38},
journal = {Journal of Differential Geometry},
number = {3},
publisher = {Lehigh University},
pages = {585 -- 652},
year = {1993},
doi = {10.4310/jdg/1214454484},
URL = {https://doi.org/10.4310/jdg/1214454484}
}

@article {Collins_Li,
    AUTHOR = {Collins, Tristan C. and Li, Yang},
     TITLE = {Uniqueness of some cylindrical tangent cones to special
              {L}agrangians},
   JOURNAL = {Geom. Funct. Anal.},
  FJOURNAL = {Geometric and Functional Analysis},
    VOLUME = {33},
      YEAR = {2023},
    NUMBER = {2},
     PAGES = {376--420},
      ISSN = {1016-443X,1420-8970},
   MRCLASS = {32S25 (32U40 53D12)},
  MRNUMBER = {4578462},
MRREVIEWER = {L\"uping\ Chen},
       DOI = {10.1007/s00039-023-00634-x},
       URL = {https://doi.org/10.1007/s00039-023-00634-x},
}

@article {Harvey_Lawson,
    AUTHOR = {Harvey, Reese and Lawson, Jr., H. Blaine},
     TITLE = {Calibrated geometries},
   JOURNAL = {Acta Math.},
  FJOURNAL = {Acta Mathematica},
    VOLUME = {148},
      YEAR = {1982},
     PAGES = {47--157},
      ISSN = {0001-5962,1871-2509},
   MRCLASS = {53C40 (49F20 53C65 58E15 58G30)},
  MRNUMBER = {666108},
       DOI = {10.1007/BF02392726},
       URL = {https://doi.org/10.1007/BF02392726},
}

@article{edelenluca,
  title={Regularity of minimal surfaces near quadratic cones},
  author={Edelen, Nick and Spolaor, Luca},
  journal={Annals of Mathematics},
  volume={198},
  number={3},
  pages={1013--1046},
  year={2023},
  publisher={Department of Mathematics, Princeton University Princeton, New Jersey, USA}
}

@article {Mazet1,
    AUTHOR = {Mazet, Laurent},
     TITLE = {Minimal hypersurfaces asymptotic to {S}imons cones},
   JOURNAL = {J. Inst. Math. Jussieu},
  FJOURNAL = {Journal of the Institute of Mathematics of Jussieu. JIMJ.
              Journal de l'Institut de Math\'ematiques de Jussieu},
    VOLUME = {16},
      YEAR = {2017},
    NUMBER = {1},
     PAGES = {39--58},
      ISSN = {1474-7480,1475-3030},
   MRCLASS = {53A10},
  MRNUMBER = {3591961},
MRREVIEWER = {Jianquan\ Ge},
       DOI = {10.1017/S1474748015000110},
       URL = {https://doi.org/10.1017/S1474748015000110},
}

@article {Sogge,
    AUTHOR = {Sogge, Christopher D.},
     TITLE = {Oscillatory integrals and spherical harmonics},
   JOURNAL = {Duke Math. J.},
  FJOURNAL = {Duke Mathematical Journal},
    VOLUME = {53},
      YEAR = {1986},
    NUMBER = {1},
     PAGES = {43--65},
      ISSN = {0012-7094,1547-7398},
   MRCLASS = {42B10 (42B15)},
  MRNUMBER = {835795},
MRREVIEWER = {Douglas\ Kurtz},
       DOI = {10.1215/S0012-7094-86-05303-2},
       URL = {https://doi.org/10.1215/S0012-7094-86-05303-2},
}

@article{szekelyhidi2020,
  title={Uniqueness of certain cylindrical tangent cones},
  author={Sz{\'e}kelyhidi, G{\'a}bor},
  journal={\url{https://arxiv.org/abs/2012.02065}},
  year={2020}
}

@article {firester2025,
    author={Firester, Benjy and Tsiamis, Raphael and Wang, Yipeng},
    TITLE = {Uniqueness of {C}ylindrical {T}angent {C}ones {$C_{p,q} \times
              \mathbb {R}$}},
    JOURNAL = {Calc. Var. Partial Differential Equations},
  FJOURNAL = {Calculus of Variations and Partial Differential Equations},
    VOLUME = {65},
      YEAR = {2026},
    NUMBER = {5},
    PAGES = {Paper No. 165, 12},
       ISSN = {0944-2669,1432-0835},
   MRCLASS = {53A10 (53A07 53C42)},
       DOI = {10.1007/s00526-025-03238-5},
       URL = { https://doi.org/10.1007/s00526-025-03238-5},
}

@book {bookurbano,
    AUTHOR = {Urbano, Francisco},
     TITLE = {Morse index of minimal submanifolds},
    SERIES = {Mathematical Surveys and Monographs},
    VOLUME = {292},
 PUBLISHER = {American Mathematical Society, Providence, RI},
      YEAR = {[2025] \copyright 2025},
     PAGES = {xiv+293},
      ISBN = {978-1-4704-8021-9},
   MRCLASS = {53A10 (49Q05 53-02 53C42 58E12)},
  MRNUMBER = {4965467},
}

@article {Urbanoclli,
    AUTHOR = {Urbano, Francisco},
     TITLE = {Minimal surfaces with low index in the three-dimensional
              sphere},
   JOURNAL = {Proc. Amer. Math. Soc.},
  FJOURNAL = {Proceedings of the American Mathematical Society},
    VOLUME = {108},
      YEAR = {1990},
    NUMBER = {4},
     PAGES = {989--992},
      ISSN = {0002-9939,1088-6826},
   MRCLASS = {53C40},
  MRNUMBER = {1007516},
MRREVIEWER = {Steen\ Markvorsen},
       DOI = {10.2307/2047957},
       URL = {https://doi.org/10.2307/2047957},
}

@article {MarquesNeves,
    AUTHOR = {Marques, Fernando C. and Neves, Andr\'e},
     TITLE = {Min-max theory and the {W}illmore conjecture},
   JOURNAL = {Ann. of Math. (2)},
  FJOURNAL = {Annals of Mathematics. Second Series},
    VOLUME = {179},
      YEAR = {2014},
    NUMBER = {2},
     PAGES = {683--782},
      ISSN = {0003-486X,1939-8980},
   MRCLASS = {53C42 (49Q20)},
  MRNUMBER = {3152944},
MRREVIEWER = {Andrea\ Mondino},
       DOI = {10.4007/annals.2014.179.2.6},
       URL = {https://doi.org/10.4007/annals.2014.179.2.6},
}

@article {Colding-Minicozzi,
    AUTHOR = {Holck Colding, Tobias and Minicozzi, II, William P.},
     TITLE = {Regularity of elliptic and parabolic systems},
   JOURNAL = {Ann. Sci. \'Ec. Norm. Sup\'er. (4)},
  FJOURNAL = {Annales Scientifiques de l'\'Ecole Normale Sup\'erieure.
              Quatri\`eme S\'erie},
    VOLUME = {56},
      YEAR = {2023},
    NUMBER = {6},
     PAGES = {1883--1921},
      ISSN = {0012-9593,1873-2151},
   MRCLASS = {53E10},
  MRNUMBER = {4706576},
MRREVIEWER = {Glen\ E.\ Wheeler},
}

@article {E-L-S,
    AUTHOR = {Evans, Christopher G. and Lotay, Jason D. and Schulze, Felix},
     TITLE = {Remarks on the self-shrinking {C}lifford torus},
   JOURNAL = {J. Reine Angew. Math.},
  FJOURNAL = {Journal f\"ur die Reine und Angewandte Mathematik. [Crelle's
              Journal]},
    VOLUME = {765},
      YEAR = {2020},
     PAGES = {139--170},
      ISSN = {0075-4102,1435-5345},
   MRCLASS = {53E10 (53D12)},
  MRNUMBER = {4129358},
MRREVIEWER = {Xianfeng\ Wang},
       DOI = {10.1515/crelle-2019-0015},
       URL = {https://doi.org/10.1515/crelle-2019-0015},
}

@article {Schulze,
    AUTHOR = {Schulze, Felix},
     TITLE = {Uniqueness of compact tangent flows in mean curvature flow},
   JOURNAL = {J. Reine Angew. Math.},
  FJOURNAL = {Journal f\"ur die Reine und Angewandte Mathematik. [Crelle's
              Journal]},
    VOLUME = {690},
      YEAR = {2014},
     PAGES = {163--172},
      ISSN = {0075-4102,1435-5345},
   MRCLASS = {53C44},
  MRNUMBER = {3200339},
MRREVIEWER = {Xusheng\ Liu},
       DOI = {10.1515/crelle-2012-0070},
       URL = {https://doi.org/10.1515/crelle-2012-0070},
}

@article {Sun-Zhu,
    AUTHOR = {Sun, Ao and Zhu, Jonathan J.},
     TITLE = {Rigidity and \L ojasiewicz inequalities for {C}lifford
              self-shrinkers},
   JOURNAL = {Calc. Var. Partial Differential Equations},
  FJOURNAL = {Calculus of Variations and Partial Differential Equations},
    VOLUME = {64},
      YEAR = {2025},
    NUMBER = {5},
     PAGES = {Paper No. 163, 30},
      ISSN = {0944-2669,1432-0835},
   MRCLASS = {53A10 (35A23)},
  MRNUMBER = {4912445},
MRREVIEWER = {Annalisa\ Cesaroni},
       DOI = {10.1007/s00526-025-03037-y},
       URL = {https://doi.org/10.1007/s00526-025-03037-y},
}

@incollection {Yau,
    AUTHOR = {Yau, Shing-Tung},
     TITLE = {Open problems in geometry},
 BOOKTITLE = {Differential geometry: partial differential equations on
              manifolds ({L}os {A}ngeles, {CA}, 1990)},
    SERIES = {Proc. Sympos. Pure Math.},
    VOLUME = {54, Part 1},
     PAGES = {1--28},
 PUBLISHER = {Amer. Math. Soc., Providence, RI},
      YEAR = {1993},
      ISBN = {0-8218-1494-X},
   MRCLASS = {53-02},
  MRNUMBER = {1216573},
       DOI = {10.1090/pspum/054.1/1216573},
       URL = {https://doi.org/10.1090/pspum/054.1/1216573},
}

@article {H-J-W,
    AUTHOR = {Hildebrandt, S. and Jost, J. and Widman, K.-O.},
     TITLE = {Harmonic mappings and minimal submanifolds},
   JOURNAL = {Invent. Math.},
  FJOURNAL = {Inventiones Mathematicae},
    VOLUME = {62},
      YEAR = {1980/81},
    NUMBER = {2},
     PAGES = {269--298},
      ISSN = {0020-9910,1432-1297},
   MRCLASS = {58E20 (53C40)},
  MRNUMBER = {595589},
MRREVIEWER = {Samuel\ I.\ Goldberg},
       DOI = {10.1007/BF01389161},
       URL = {https://doi.org/10.1007/BF01389161},
}

@article {Jost-Xin,
    AUTHOR = {Jost, J. and Xin, Y. L.},
     TITLE = {Bernstein type theorems for higher codimension},
   JOURNAL = {Calc. Var. Partial Differential Equations},
  FJOURNAL = {Calculus of Variations and Partial Differential Equations},
    VOLUME = {9},
      YEAR = {1999},
    NUMBER = {4},
     PAGES = {277--296},
      ISSN = {0944-2669,1432-0835},
   MRCLASS = {53A10 (49Q05)},
  MRNUMBER = {1731468},
MRREVIEWER = {Harold\ Parks},
       DOI = {10.1007/s005260050142},
       URL = {https://doi.org/10.1007/s005260050142},
}

@article {Jost-Xin-Yang,
    AUTHOR = {Jost, J. and Xin, Y. L. and Yang, Ling},
     TITLE = {A spherical {B}ernstein theorem for minimal submanifolds of
              higher codimension},
   JOURNAL = {Calc. Var. Partial Differential Equations},
  FJOURNAL = {Calculus of Variations and Partial Differential Equations},
    VOLUME = {57},
      YEAR = {2018},
    NUMBER = {6},
     PAGES = {Paper No. 166, 21},
      ISSN = {0944-2669,1432-0835},
   MRCLASS = {58E20 (53A10)},
  MRNUMBER = {3863563},
MRREVIEWER = {M\'arcio\ Silva\ Santos},
       DOI = {10.1007/s00526-018-1439-2},
       URL = {https://doi.org/10.1007/s00526-018-1439-2},
}

@article {Ding-Jost-Xin,
    AUTHOR = {Ding, Qi and Jost, J. and Xin, Y. L.},
     TITLE = {Minimal graphs of arbitrary codimension in {E}uclidean space
              with bounded 2-dilation},
   JOURNAL = {Math. Ann.},
  FJOURNAL = {Mathematische Annalen},
    VOLUME = {390},
      YEAR = {2024},
    NUMBER = {3},
     PAGES = {3355--3418},
      ISSN = {0025-5831,1432-1807},
   MRCLASS = {53A07 (49Q15)},
  MRNUMBER = {4803455},
MRREVIEWER = {Serena\ Dipierro},
       DOI = {10.1007/s00208-024-02830-y},
       URL = {https://doi.org/10.1007/s00208-024-02830-y},
}

@book {Hall,
    AUTHOR = {Hall, Brian},
     TITLE = {Lie groups, {L}ie algebras, and representations},
    SERIES = {Graduate Texts in Mathematics},
    VOLUME = {222},
   EDITION = {Second},
      NOTE = {An elementary introduction},
 PUBLISHER = {Springer, Cham},
      YEAR = {2015},
     PAGES = {xiv+449},
      ISBN = {978-3-319-13466-6; 978-3-319-13467-3},
   MRCLASS = {22-01 (17-01)},
  MRNUMBER = {3331229},
       DOI = {10.1007/978-3-319-13467-3},
       URL = {https://doi.org/10.1007/978-3-319-13467-3},
}

\end{document}